\documentclass[11pt,reqno]{amsart}
\usepackage[margin=1in]{geometry}
\usepackage{amsmath,amssymb,amsthm,graphicx,amsxtra, setspace}
\usepackage[utf8]{inputenc}
\usepackage{mathrsfs}
\usepackage{hyperref}
\usepackage{upgreek}
\usepackage{mathtools}
\usepackage[dvipsnames]{xcolor}
\usepackage[mathcal]{euscript}
\usepackage{amsmath,units}
\usepackage{pgfplots}
\usepackage{tikz}
\usepackage{verbatim}
\allowdisplaybreaks
\usepackage{dsfont}
\usepackage{fontenc}
\usepackage{textcomp}
\usepackage{marvosym}
\usepackage{eurosym}
\usepackage{upgreek}
\usepackage[pagewise]{lineno}

\DeclareMathAlphabet{\mathpzc}{OT1}{pzc}{m}{it}
\DeclareMathOperator*{\esssup}{ess\,sup}

\usepackage[cyr]{aeguill}

\hypersetup{colorlinks=true,%
	citecolor=Red,%
	filecolor=blue,%
	linkcolor=blue,%
}
\usepackage{graphicx}

\usepackage{xpatch}
\makeatletter   
\xpatchcmd{\@tocline}
{\hfil\hbox to\@pnumwidth{\@tocpagenum{#7}}\par}
{\ifnum#1<0\hfill\else\dotfill\fi\hbox to\@pnumwidth{\@tocpagenum{#7}}\par}
{}{}
\makeatother    
\makeatletter
\def\l@subsection{\@tocline{2}{0pt}{4pc}{6pc}{}}
\def\l@subsubsection{\@tocline{3}{0pt}{8pc}{8pc}{}}
\makeatother

\makeatletter
\def\l@section{\@tocline{1}{12pt}{0pt}{}{\bfseries}}% <- added
\makeatother

\newtheorem{theorem}{Theorem}[section]
\newtheorem{lemma}[theorem]{Lemma}
\newtheorem{proposition}[theorem]{Proposition}

\newtheorem{corollary}[theorem]{Corollary}
\newtheorem{definition}[theorem]{Definition}

\newtheorem{remark}[theorem]{Remark}
\newtheorem{hypothesis}[theorem]{Hypothesis}

\let\originalleft\left
\let\originalright\right
\renewcommand{\left}{\mathopen{}\mathclose\bgroup\originalleft}
\renewcommand{\right}{\aftergroup\egroup\originalright}

\newcommand{\Tr}{\mathop{\mathrm{Tr}}}
\renewcommand{\d}{\/\mathrm{d}\/}

\def\w{\textbf{W}^{\varepsilon}_{{\theta}^{\varepsilon}}}

\def\L{\mathbb{L}}
\def\A{\mathrm{A}}
\def\I{\mathrm{I}}

\def\C{\mathrm{C}}
\def\f{\boldsymbol{f}}

\def\J{\mathrm{J}}
\def\B{\mathrm{B}}
\def\D{\mathrm{D}}
\def\y{\mathbf{y}}

\def\Y{\mathbf{Y}}
\def\Z{\boldsymbol{Z}}
\def\E{\mathbb{E}}
\def\X{\boldsymbol{X}}
\def\Y{\boldsymbol{Y}}
\def\x{\boldsymbol{x}}

\def\g{\boldsymbol{g}}

\def\y{\boldsymbol{y}}
\def\z{\boldsymbol{z}}
\def\v{\boldsymbol{v}}
\def\w{\boldsymbol{w}}
\def\W{\mathrm{W}}

\def\Q{\mathrm{Q}}

\def\N{\mathbb{N}}

\def\V{\mathbb{V}}
\def\R{\mathbb{R}}
\def\wi{\widetilde}
\def\Q{\mathrm{Q}}

\def\P{\mathbb{P}}
\def\u{\boldsymbol{u}}

\def\H{\mathbb{H}}

\newcommand{\eps}{\varepsilon}

\renewcommand{\d}{\/\mathrm{d}\/}

\newcommand{\Addresses}{{% additional braces for segregating \footnotesize
		\footnote{
		\noindent \textsuperscript{1,2}Department of Mathematics, Indian Institute of Technology Roorkee-IIT Roorkee,
		Haridwar Highway, Roorkee, Uttarakhand 247667, INDIA.\par\nopagebreak
		\noindent  \textit{e-mail:} \texttt{Manil T. Mohan: maniltmohan@ma.iitr.ac.in, maniltmohan@gmail.com}
		
		\textit{e-mail:} \texttt{Sagar Gautam: sagar\_g@ma.iitr.ac.in, sagargautamkm@gmail.com}
		
		\noindent \textsuperscript{*}Corresponding author.

			\textit{Key words:} convective Brinkman-Forchheimer equations, damped Navier-Stokes equations, viscosity solutions, Hamilton-Jacobi-Bellman equations, Large deviation principle, Laplace limit
			
			Mathematics Subject Classification (2020): Primary 35D40, 60F10, 49L25; Secondary 76D05, 35R60

}}}

\begin{document}
	
	%	\linenumbers
	
	\title[Viscosity solution and LDP for stochastic CBF equations]{
		A viscosity solution approach to the large deviation principle for stochastic convective Brinkman-Forchheimer equations
				\Addresses}
		\author[S. Gautam and M. T. Mohan]
	{Sagar Gautam\textsuperscript{1} and Manil T. Mohan\textsuperscript{2*}}

	\maketitle
   	
	\begin{abstract}

%	  The large deviation results for the sequence of random variables $\{\Y_n\}_{n \geq 1}$, as established by Vardhan and Bryc, focus on the asymptotic behaviour of functional of the form 
%		$$\frac{1}{n}\log\E[e^{n f(\Y_n)}].$$ 
%	In the case of Markov processes, the functional $\frac{1}{n}\log\E[e^{n f(\Y_n(t))}|\Y_n(0)=\y]$ forms a nonlinear semigroup. Therefore, the study of large deviations for sequences of Markov processes can be approached by analyzing the asymptotic behavior of the corresponding sequences of nonlinear semigroups. Such a limit induces the Hamilton-Jacobi-Bellman equation; therefore, viscosity solution methods play a key role. 
	
	This article develops the viscosity solution approach to the large deviation principle for the following two- and three-dimensional stochastic convective Brinkman-Forchheimer equations on the torus $\mathbb{T}^d,\ d\in\{2,3\}$ with small noise intensity: 
	\begin{align*}
		\d\u_n+[-\mu\Delta\boldsymbol{u}_n+ (\boldsymbol{u}_n\cdot\nabla)\boldsymbol{u}_n +\alpha\boldsymbol{u}_n+\beta|\boldsymbol{u}_n|^{r-1}\boldsymbol{u}_n+\nabla p_n]\d t=\f\d t+\frac{1}{\sqrt{n}}\Q^{\frac12}\d\W, \ \nabla\cdot\boldsymbol{u}_n=0, 
	\end{align*}
	where $\mu,\alpha,\beta>0$, $r\in[1,\infty)$, $\Q$ is a trace class operator and $\W$ is Hilbert-valued calendrical Wiener process.  We build our analysis on the framework of Varadhan and Bryc, together with the techniques of [J. Feng et.al., Large Deviations for Stochastic Processes, American Mathematical Society (2006) vol. \textbf{131}]. 
    By employing the techniques from the comparison principle, we identify the Laplace limit as the convergence of the viscosity solution of the associated second-order singularly perturbed Hamilton-Jacobi-Bellman equation. 
	A key advantage of this method is that it establishes a Laplace principle without relying on additional sufficient conditions such as Bryc's theorem, which the literature commonly requires. 
	Unlike other approaches to large deviation principle, this method reformulates the problem into an infinite-dimensional Hamilton-Jacobi Bellman equation, where well-posedness is ensured through viscosity solutions approach. This multi-step framework highlights both the probabilistic structure of rare events and the rich PDE techniques required for their rigorous study.
	 For $r>3$ and $r=3$ with $2\beta\mu\geq1$, we also derive the exponential moment bounds without imposing the classical orthogonality condition 
	$((\u_n\cdot\nabla)\u_n,\A\u_n)=0$, where $\A=-\Delta$, in both two-and three-dimensions. We first establish the large deviation principle in the Skorohod space. Then, by using the $\C-$exponential tightness, we finally establish the large deviation principle in the continuous space.

	\end{abstract}

	\tableofcontents
	
	\section{Introduction}\setcounter{equation}{0}
	In this article, we investigate the large deviation result for the stochastic convective Brinkman-Forchheimer equations (Navier-Stokes equations with damping) under small random perturbation, employing the framework of viscosity solution methods.
	
		\subsection{The model} The convective Brinkman-Forchheimer (CBF) equations describe the motion of incompressible fluid flows in a saturated porous medium. These equations are applicable when the fluid flow rate is sufficiently high and the porosity is not too small. Let us first provide the mathematical formulation of the stochastic convective Brinkman-Forchheimer (SCBF) equations. Let $t\geq0$ and $\mathbb{T}^d$, for $d\in\{2,3\}$,  be a $d-$dimensional torus. Let  $\big(\Omega,\mathscr{F},\{\mathscr{F}_s^t\}_{s\geq t},\P\big)$ be a complete filtered probability space satisfying the usual conditions (see Section \ref{stchpreLDP}). We aim to establish the large deviation principle of the following stochastic convective Brinkman-Forchheimer (SCBF) equations, which describe the evolution of the velocity vector field $\X_n(\cdot):[t,+\infty)\times\mathbb{T}^d\times\Omega\to\R^d$ and pressure $p_n(\cdot):[t,+\infty)\times\mathbb{T}^d\times\Omega\to\R$, with small noise intensity:
		{\small \begin{equation}\label{pscbf}
			\left\{
			\begin{aligned}
				\d\X_n(s,\xi)+[-\mu\Delta\X_n(s,\xi)+(\X_n(s,\xi)\cdot\nabla)&\X_n(s,\xi) +\alpha\X_n(s,\xi)\\+\beta|\X_n(s,\xi)|^{r-1}\X_n(s,\xi)+\nabla p_n(s,\xi)]
				\d t &=\g(s,\xi)\d s+\frac{1}{\sqrt{n}}\sqrt{\Q}\d\W(s), \ \text{ in } \ (t,+\infty)\times\mathbb{T}^d, \\ \nabla\cdot\X_n(s,\xi)&=0, \ \text{ in } \ [t,+\infty)\times\mathbb{T}^d, \\
				\X_n(t,\xi)&=\x \ \text{ in } \ \mathbb{T}^d,
			\end{aligned}
			\right.
		\end{equation}}
	where $\g(\cdot):[t,+\infty)\times\mathbb{T}^d\times\Omega\to\R^d$ is an external forcing and $\{\W(s)\}_{s\geq t}$ is a Hilbert-valued cylindrical Wiener process defined on a filtered probability space  $\big(\Omega,\mathscr{F},\{\mathscr{F}_s^t\}_{s\geq t},\P\big)$  and $\Q$ is a trace class operator. For the uniqueness of the pressure, one can impose the condition $
	\int_{\mathbb{T}^d}p_n(s,\xi)\d\xi=0,  \text{ in }  [t,+\infty).$ The constant $\mu>0$ denotes the \emph{Brinkman coefficient} (effective viscosity) and the constants $\alpha>0$ and $\beta>0$ are due to Darcy-Forchheimer law which are termed as \emph{Darcy} (permeability of the porous medium) and \emph{Forchheimer} (proportional to the porosity of the material) coefficients, respectively. For $\alpha=\beta=0$, one obtains the classical stochastic Navier-Stokes equations (SNSE). The parameter $r\in[1,+\infty)$ is known as the \emph{absorption exponent}  and the case  $r=3$ is referred as the critical exponent (\cite{KWH}). 

  \subsection{Literature survey}
   We now review the literature on the SCBF equations, the viscosity solution of the associated Hamilton-Jacobi-Bellman (HJB) equation, and the large deviation principle.
   
  \subsubsection{The SCBF equations and viscosity solution of the associated HJB equation}
Let us first discuss the literature available for the solvability of the SCBF equations. By exploiting a monotonicity property of the linear and nonlinear operators as well as a stochastic generalization of the Minty-Browder technique, the authors in \cite{KKMTM,MTM8} established the global existence of a unique strong solution $$\X_n(\cdot)\in\C([0,T];\L^2(\mathbb{T}^d))\cap\mathrm{L}^2(0,T;\H^1(\mathbb{T}^d))\cap\mathrm{L}^{r+1}(0,T;\L^{r+1}(\mathbb{T}^d)),\ \mathbb{P}\text{-a.s.},$$  for $\x\in\L^2(\mathbb{T}^d)$, which satisfies the energy equality (It\^o's formula) for SCBF equations (in bounded domains and on a torus) driven by multiplicative Gaussian noise. 
  Under suitable assumptions on the initial data ($\x\in\H^1(\mathbb{T}^d)$) and noise coefficients, the authors have also showed the following regularity result: 
  \begin{align*}
  	\X_n(\cdot)\in\C([0,T];\H^1(\mathbb{T}^d))\cap\mathrm{L}^2(0,T;\H^2(\mathbb{T}^d))\cap\mathrm{L}^{r+1}(0,T;\L^{p(r+1)}(\mathbb{T}^d)),\ \mathbb{P}\text{-a.s.},
  \end{align*}  
  where $p\in[2,\infty)$ for $d=2$ and $p=3$ for $d=3$. 
 
 The concept of viscosity solutions was first introduced by Crandall and Lions in \cite{MGL} (also see \cite{MGEL,Hsh1}) for the first order HJB equation in finite dimensions. 
 %Later, Ishii \cite{Hsh1} proved the uniqueness of the viscosity solution (possibly unbounded) by establishing a comparison principle. 
 %They proved the existence and uniqueness of viscosity solution under some appropriate hypothesis.
 %   It is Lions in \cite{} who first observe that the value function of the deterministic optimal control problem is the viscosity solution of the associated HJB equation.
Lions later extended this framework to second-order cases \cite{PL1, PL2, PL3}, motivated by optimal control problems involving diffusion processes. For further reading, we refer interested readers to the survey article \cite{MGHL} as well as monographs \cite{WHF,JYXYZ}, which provide a comprehensive treatment of viscosity solutions and a detailed account of finite-dimensional HJB equations. In the infinite dimensional setting, these equations were first studied by Barbu and Da Prato (for instance, see \cite{VBGDP}), setting the problem in classes of convex functions, using semigroup and perturbation methods. Crandall and Lions carried out the development of viscosity solution theory to infinite-dimensional settings through a series of works \cite{MGL1, MGL2, MGL3, MGL4}. Later, Lions \cite{PL4} extended the framework to address the unbounded second-order HJB equations arising in the optimal control of the Zakai equation (see also \cite{Hsh, Kshi}).
 
   \subsubsection{The large deviation principle (LDP)} Large deviation theory is an active field within probability theory, focusing on the asymptotic estimates of probabilities of rare events associated with stochastic processes. The framework for large deviation theory along with its applications can be found in \cite{ABPD,ADOZ,PDRS,JFTGK,MFAW,DWS,SRS1}. 
   
The authors  in \cite{ICAM, SSPS}  established a Wentzell-Freidlin large deviation principle (LDP) for the two-dimensional stochastic Navier-Stokes equations (SNSE) driven by Gaussian noise. In \cite{MR1}, a Wentzell-Freidlin type LDP (see \cite{MFAW}) was established for the three-dimensional stochastic tamed Navier-Stokes equations with multiplicative Gaussian noise, either in the whole space or on the torus, whereas \cite{MR2} addressed a small-time LDP in bounded domains. The ergodic behavior of the two-dimensional SNSE and the LDP for the occupation measure over large times were studied in \cite{MGO}. In \cite{ZBSC}, the authors proved the LDP for the invariant measure associated with the SNSE perturbed by smooth additive noise on $\mathbb{T}^2$. Furthermore, \cite{SCNP} established the LDP for the invariant measure on $\mathbb{T}^2$ for the SNSE driven by Gaussian noise that is,white in time and colored in space.
   
   In the context of SCBF equations, the author in \cite{MTMLDP} established a Wentzell-Freidlin type LDP for two- and three-dimensional SCBF equations for $r>3$ and $r=3$ with $\beta\mu>1$ using the weak convergence approach of Budhiraja and Dupuis \cite{ABPD}, and the well-known results of Varadhan and Bryc \cite{ADOZ, SRS1}. Moreover, for two-dimensional SCBF equations, the authors in \cite{AKMTMLDP} studied the ergodic behaviour and proved the LDP for occupation measures for large time.

   An appealing perspective on studying rare event probabilities is through PDEs that characterize them, which typically take the form of HJB equations. This connection arises because large deviation rate functions often possess a variational representation, and the associated value function naturally satisfies an HJB equation. Consequently, the analysis of rare events provides a direct link between large deviation theory and optimal control. In \cite{WHF1}, exit time probabilities were investigated via stochastic optimal control techniques. By applying a logarithmic transformation to the generators of Markov processes, the author expressed the large deviation behaviour of exit probabilities in terms of the convergence of solutions to a sequence of HJB equations. A significant breakthrough was achieved in \cite{LCH, LCPE}, where the viscosity solution framework was introduced for the first time in this setting, offering a powerful and influential tool for the analysis of such problems. For further developments and related contributions, we refer the reader to \cite{MBACDG,DGH,MKob,JRJW,AS3} and the references therein.
   
   In the current work, we develop a rigorous mathematical framework for establishing the LDP in the setting of SCBF equations. Our approach combines infinite-dimensional stochastic analysis with the viscosity solution theory of HJB equations. In particular, we show that under small random perturbations, the logarithmic transform of expectation functionals naturally gives rise to an infinite-dimensional HJB  equation in the Hilbert space of velocity fields. The viscosity solution framework then provides the essential tool to ensure well-posedness and stability of this HJB formulation, thereby connecting probabilistic large deviation techniques with deterministic control-theoretic structures. This interplay between LDP, viscosity solutions, and HJB theory highlights the deep mathematical structure underlying stochastic fluid dynamics and offers a robust pathway for analyzing rare events in nonlinear dissipative systems.\emph{
   While researchers can study LDPs through various probabilistic tools, the viscosity solution approach distinguishes itself by reformulating the problem into an infinite-dimensional HJB equation and establishing its well-posedness and stability through a delicate analysis of viscosity solutions. This approach carries out several rigorous steps that link stochastic dynamics, nonlinear PDE techniques, and control-theoretic structures, thereby securing the LDP and enriching the mathematical understanding of rare events in stochastic fluid systems.}
    
   In this work, we focus on establishing the large deviation principle for Hilbert-valued diffusion processes subjected to small random perturbations, modeled by the SCBF equations \eqref{pscbf}. The literature on large deviations in the small-noise limit is extensive, and several notable results are presented in \cite{HBAM, ABPD1, SCMR2, SCGG, MHC, FCAM, PLC, JF, JFTGK, GKJX, SP, gdp, RSO}, and references therein.

    \subsection{Comparison with related works, contribution, difficulties and novelties}\label{comcont}
This study investigates two principal aspects of SCBF equations: the viscosity solutions of the associated second-order equation and the establishment of a large deviation principle. We begin by contextualizing our work through a critical comparison with existing literature. The paper then details our key contributions and outlines the primary methodological challenges encountered, explaining how our approach successfully addresses them.
 
   \subsubsection{Extension to three dimensional damped Navier-Stokes equations (NSE)}
   In the study presented in \cite{AS2}, the author explored large deviation results for two-dimensional SNSE using an analytical approach grounded in the theory of viscosity solutions. However, their findings are limited to two dimensions due to the unavailability of global well-posedness results for three-dimensional SNSE. The primary challenge lies in controlling the convective term $(\u\cdot\nabla)\u$, as appropriate Sobolev embeddings are unavailable for this case. 
The present work exploits the distinctive mathematical properties of the absorption term $|\u|^{r-1}\u$, which has a mathematical advantage over the convective term $(\u\cdot\nabla)\u$. Specifically, for $r>3$ in dimension $d\in\{2,3\}$ and $r=3$ in $d=3$ with $2\beta\mu\geq1$, the absorption term along with the diffusion term (that is, $-\Delta\u$) dominate the convective term (see Remark \ref{BLrem}). This dominance leads to the global solvability of  CBF equations (or damped NSE). Consequently, the current study examines the large deviation results for the damped version of the NSE in both two and three-dimensions.
    
     \subsubsection{Difficulties beyond the periodic domains}\label{advtg1}
    Unlike the case of the whole space or periodic domains, the analysis on bounded domains presents additional challenges. A major difficulty arises because the term $\mathcal{P}(|\u|^{r-1}\u)$ typically does not vanish at the boundary. Furthermore, the Helmholtz-Hodge projection operator $\mathcal{P}$ does not, in general, commute with the Laplacian $-\Delta$ (see \cite[Example 2.19]{JCR}). As a consequence, the identity
    \begin{align}\label{toreq}
    	&\int_{\mathcal{O}}(-\Delta\u(\xi))\cdot|\u(\xi)|^{r-1}\u(\xi)\d \xi \nonumber\\&=\int_{\mathcal{O}}|\nabla\u(\xi)|^2|\u(\xi)|^{r-1}\d \xi+ \frac{r-1}{4}\int_{\mathcal{O}}|\u(\xi)|^{r-3}|\nabla|\u(\xi)|^2|^2\d \xi,
    \end{align}
    cannot be applied directly in a bounded domain $\mathcal{O}$. On the other hand, in the periodic setting $\mathbb{T}^d$, the situation is more favourable: the projection $\mathcal{P}$ commutes with the Laplacian $-\Delta$ (cf. \cite[Theorem 2.22]{JCR}). This property allows us to apply the identity \eqref{toreq} and, consequently, is essential for deriving regularity estimates.
    
    \subsubsection{Challenges to handle convective term  on $\mathbb{T}^3$}\label{advtg2} The study referenced in \cite{AS2} examined large deviation results for two-dimensional SNSE in $\mathbb{T}^2$. A key advantage of working in $\mathbb{T}^2$ is the well-established fact that $(\mathcal{B}(\u),\A\u)=0,$ for $\u\in\mathrm{D}(\A),$ where $\mathcal{B}(\u)=\mathcal{P}[(\u\cdot\nabla)\u]$ (cf. \cite[Lemma 3.1, pp. 404]{Te1}). This property simplifies the calculations while deriving regularity  estimates and proving the exponential moment estimate without significant technical difficulties. In our work, we are considering the three-dimensional case, explicitly addressing the damped NSE or CBF equations. In this context, the property $(\mathcal{B}(\u),\A\u)=0$ will no longer be true in $\mathbb{T}^3$. Therefore, we need to estimate this term carefully. At the same time, due to the presence of the absorption term $|\u|^{r-1}\u$, as compared to NSE, we have to deal with the extra term $(\mathcal{C}(\u),\A\u),$ where $\mathcal{C}(\u)=\mathcal{P}(|\u|^{r-1}\u)$. Fortunately, on $\mathbb{T}^d$ for $d\in\{2,3\}$, this term can be handled, thanks to the equality \eqref{torusequ}, which is one of the advantages of our framework. Further, the term corresponding to the bilinear operator, that is, $(\mathcal{B}(\u),\A\u)$, can be handled with the help of diffusion term $\A\u$ and the equality \eqref{torusequ} for $r>3$ and $r=3$ with $2\beta\mu\geq1$ (see Remark \ref{BLrem}). In this way, we relax the fact $(\mathcal{B}(\u),\A\u)=0$ and establish the large deviation result for stochastic damped NSE in both two and three dimensions. 
    %Consequently, our findings are also applicable, to a certain extent, to a bounded domain.
    
    While the term $(\mathcal{B}(\u),\A\u)$ can be controlled with the help of the identity \eqref{torusequ}, one would typically need to restrict on $\mu$ or $\alpha$ to derive exponential moment estimates. This difficulty arises because, unlike the framework in \cite{AS2}, we cannot exploit the relation $(\mathcal{B}(\u),\A\u)=0$, as we are working on $\mathbb{T}^3$ also.
    	 On $\mathbb{T}^3$, we obtain the exponential moment bound for the solutions of the SCBF equations in \eqref{stap}, under the condition  (see Proposition \ref{propexpest})
    	\begin{align*}
    	\mbox{$\mu\geq \frac{1}{4}\max\left\{\frac{1}{\alpha},\frac{1}{\beta}\right\}$\  or \ $\mu\geq \frac{1}{\alpha^{\frac{r-3}{r-1}}}\left(\frac{(r-3)}{2(r-1)}\right)^{\frac{r-3}{r-1}}\left(\frac{1}{\beta(r-1)}\right)^{\frac{2}{r-1}}$,}
    	\end{align*} 
    On $\mathbb{T}^2$, since $(\mathcal{B}(\u),\A\u)=0$, the results hold true for any $\mu,\alpha,\beta>0$ (see Remark \ref{rem-2D}). 
    	%Thus, if $\mu>0$ is small, one may choose $\alpha>0$ to be sufficiently large, and vice versa. 
    	Furthermore, we provide a detailed proof of exponential moment estimates with a detailed explanation, carefully illustrating how the damping mechanism plays a crucial role.
   
    \subsubsection{Monotonicity of the operator $\mu\A+\mathcal{B}(\cdot)+\beta\mathcal{C}(\cdot)$}
   The author  in \cite{AS2} used a truncation  ( or quantization) of the $\mathcal{B}(\cdot)$ operator, that is, the operator $\mathcal{B}^q(\cdot),$ which is defined as
    \begin{equation}\label{eqn-bq}
    	\mathcal{B}^q(\u):=
    	\left\{
    	\begin{aligned}
    		\mathcal{B}(\u), \ \ \text{ when } \ \|\u\|_{\V}\leq q,\\
    		\frac{q^2}{\|\u\|_{\V}^2}\mathcal{B}(\u), \ \ \text{ when } \ \|\u\|_{\V}>q.
    	\end{aligned}
    	\right.
    \end{equation}
    With this truncation, the Navier-Stokes operator $\mu\A+\mathcal{B}^q(\cdot)$ satisfies the following local monotonicity estimate:
    \begin{align*}
    	\langle\mu\A\u+\mathcal{B}^q(\u)-\mu\A\v-\mathcal{B}^q(\v), \u-\v\rangle+\mathpzc{C}\|\u-\v\|_{\H}^2\geq
    	\frac{\mu }{2}\|\nabla(\u-\v)\|_{\H}^2,
    \end{align*}
    for some $\mathpzc{C}$ and all $\u,\v\in\mathcal{D}(\A^{1/2})$. The author  in \cite{AS2} employed this idea to establish the comparison principle for viscosity solutions. In contrast, in the current work, the presence of the nonlinear operator $\mathcal{C}(\cdot)$, corresponding to the absorption term $|\u|^{r-1}\u$, allows us to avoid any truncation procedure, thanks to its inherent monotonicity (see Subsection \ref{C1prop}). More precisely, for $r>3$ and $r=3$ with $2\beta\mu\geq1$ in both two and three dimensions, we do not need to truncate $\mathcal{B}(\cdot)$. In this case, the nonlinear absorption term itself guarantees the required monotonicity, and a corresponding local monotonicity estimate can be derived (see Lemma \ref{monoest}). Consequently, unlike the approach in \cite{AS2}, our proof of the comparison principle proceeds without employing any quantization procedure.
    
    \subsubsection{A PDE approach to large deviation principle for SCBF equations}
    In the study of SCBF equations, the large deviation principle has previously been established  in two and three dimensions under multiplicative Gaussian perturbations (see \cite{MTMLDP}). That work relies on the classical  Wentzell-Freidlin framework, together with the results of Varadhan and Bryc, and proves the Laplace principle by employing the weak convergence approach due to Budhiraja and Dupuis. Further, \cite{MTMLDP} investigates the exit time problem for solutions of the SCBF equations within the framework of small-noise asymptotics provided by large deviation theory. By employing the Sobolev embedding theorem and assuming that $\Tr(\A^{\alpha}\Q)<+\infty$ (in addition to $\Tr(\Q)<+\infty$) for some appropriate value of $\alpha>0$, the author in \cite{MTMLDP} established the LDP through an application of the contraction principle.
    
    In contrast, the present work adopts a completely different route than the one in \cite{MTMLDP}, inspired by the viscosity solution approach developed in \cite{AS2}. It allows us to directly establish the Laplace principle without relying on the weak convergence framework of Budhiraja and Dupuis. Specifically, we first identify the Laplace limit at a single time as a viscosity solution to the second-order HJB equation, and then extend the result to multiple times, using the theory developed in \cite{JFTGK} (see Section \ref{LaplaceLDP}-\ref{LaDePr}). 
%    Building on this, we establish the large deviation principle for stochastic CBF equations in both two and three dimensions in the path space $\C([0,+\infty);\H)$, under the only assumption that $\Tr(\Q)<\infty$. 
%    As an advantage, unlike the work of \cite{MTMLDP}, we did not require any Sobolev embedding and the extra assumption $\Tr(\A^{\alpha}\Q)<\infty$.
     \emph{To the best of our knowledge, this is the first application of the viscosity solution method to establish large deviation results for such systems, and it has not yet been investigated in the context of either the stochastic tamed NSE or the SNSE with damping in two and three dimensions.}
   \subsubsection{Comparison from the earlier works \cite{FGSSA,smtm1} on viscosity solutions} 
    The main objective of this work is to establish the large deviation principle for the SCBF equations. We are following the approach given by \cite{AS2}  as well as the methodology presented in \cite[Chapter 4]{JFTGK}. To sketch the main idea, it is noteworthy that, due to Varadhan and Bryc \cite{PDRS}, the large deviation principle for the sequence of random variables $\{X_n\}_{n \geq 1}$ is equivalent to showing the exponential tightness and the existence of the following Laplace limit:
   \begin{align*}
   	\lim\limits_{n\to+\infty}\frac{1}{n}\log\E[e^{-n f(X_n)}],
   \end{align*}
   for all $f\in\C_b(\mathcal{S})$. However, as mentioned in \cite{AS2}, the Laplace limit described above at a single time can be identified with the viscosity solution of a second-order HJB equation (see \eqref{LDP1}).  Consequently, the central question of establishing the Laplace limit at a single time is equivalent to showing the convergence of viscosity solutions of the singularly perturbed HJB equation, and we then use the entire existing theory of viscosity solutions. \emph{It is unlike the case of our previous work \cite{smtm1}, where the second-order HJB equation is associated with an optimal control problem for the SCBF equations, whose solution we identify with the corresponding value function in the viscosity sense. In contrast, within the current framework, the viscosity solution of the associated second-order HJB equation corresponds to the Laplace limit at a single time (see Section \ref{LaplaceLDP}).} 
   It is worth noting that, in establishing the existence of viscosity solutions, particularly when proving the existence of a Laplace limit at a single time, a key difficulty arises in verifying subsolution (or supersolution) inequalities. This verification involves integrals with exponential-type terms (see Proposition \ref{Lapvscso}).
   To address this issue, following the work \cite{AS2}, we introduce a slight modification in the choice of test functions compared to those used in \cite{smtm1}. More precisely, in Definition \ref{testD}, we require the test function $\varphi$ together with its derivative $\varphi_t,\D\varphi,$ and $\D^2\varphi$ to be bounded. \emph{Consequently, this framework naturally leads to the notion of bounded viscosity solutions, which differs from earlier works \cite{FGSSA,smtm1}, where the viscosity solution, identified with the value function of the associated HJB equation, was not necessarily bounded.}

   \subsection{Main result}
    In this subsection, we present the main result of our work. We begin with the abstract formulation of the model \eqref{pscbf} and then provide the definition of solution to the SCBF equations \eqref{pscbf}. The full definitions of symbols, functions spaces and operators involved here are given in detail in Sections \ref{mathfunop}-\ref{stchpreLDP}. 
      \subsubsection{Abstract formulation} Let $T\in[t,+\infty)$ and we set $\Y_n(\cdot):=\mathcal{P}\X_n(\cdot)$, $\mathcal{P}\x:=\y$,  $\mathcal{P}\g=\f$ and $\mathbf{W}(\cdot):=\mathcal{P}\mathbf{W}(\cdot)$. In this work, we investigate the large deviation principle of the following abstract SCBF equations perturbed with small noise
     \begin{equation}\label{stap}
     	\left\{
     	\begin{aligned}
     		\d\Y_n(s)&=[-\mu\A\Y_n(s)-\mathcal{B}(\Y_n(s))-\alpha\Y_n(s)-\beta\mathcal{C}(\Y_n(s))+\f(s)]\d s 
     		\\&\qquad+\frac{1}{\sqrt{n}}\Q^{\frac12}\d\mathbf{W}(s) ,  \  \text{ in } \ (t,T)\times\H, \\
     		\Y_n(t)&=\y\in\H,
     	\end{aligned}
     	\right.
     \end{equation}
     where $t\geq0$, $\mathbf{Y}_n(\cdot):[t,T]\to\H$ is the projected velocity vector field and $\f:[0,T]\to\V$. 
     We provide the following notions of solutions for the SCBF equations \eqref{stap} (cf. \cite{SMTM}). 
     
     \begin{definition}\label{def-var-strong} 
     	Assume that $\f:[0,T]\to\V$ is bounded continuous and $\Tr(\Q)<+\infty$.
     	A process $\Y_n(\cdot)\in M^2_\nu(t,T;\H)$ is called a \emph{variational solution} of \eqref{stap} with initial condition $\Y_n(t)=\y\in\H$ if 
     	$$\E\bigg[\sup_{s\in[t,T]}\|\Y_n(s)\|_{\H}^2+\int_t^T \|\Y_n(s)\|_{\V}^2\d s+\int_t^T\|\Y_n(s)\|_{\wi\L^{r+1}}^{r+1}\d s \bigg]<+\infty,$$ the process $\Y_n(\cdot)$ having a modification with paths in $\mathrm{C}([t,T];\H)\cap\mathrm{L}^2(t,T;\V)\cap\mathrm{L}^{r+1}(t,T;\wi\L^{r+1})$, $\mathbb{P}-$a.s., 
     	and for every $\phi\in\V\cap\wi\L^{r+1}$ and every $s\in[t,T]$, $\P-\text{a.s.}$, we have
     	\begin{align*}
     		(\Y_n(s),\phi)&=(\y,\phi)+\int_t^s  \langle-\mu\A\Y_n(\tau)-\mathcal{B}(\Y_n(\tau))-
     		\alpha\Y_n(\tau)-\beta\mathcal{C}(\Y_n(\tau))+\f(\tau),\phi\rangle\d\tau\nonumber\\&\quad+
     		\frac{1}{\sqrt{n}}\int_t^s \langle\Q^{\frac12}\d\mathbf{W}(r),\phi\rangle.
     	\end{align*}
     	
     	Moreover, a process $\Y_n(\cdot)\in M^2_\nu(t,T;\H)$ is called a \emph{strong solution} of \eqref{stap} with initial condition $\Y_n(t)=\y\in\V$ if 
     	$$\E\bigg[\sup_{s\in[t,T]}\|\Y_n(s)\|_{\V}^2+\int_t^T \|\A\Y_n(s)\|_{\H}^2\d s+\int_t^T\|\Y_n(s)\|_{\wi\L^{p(r+1)}}^{r+1}\d s\bigg]<+\infty,$$ where $p\in[2,+\infty)$ for $d=2$ and $p=3$ for $d=3$, the process $\Y$ having a modification with paths in $\mathrm{C}([t,T];\V)\cap\mathrm{L}^2(t,T;\V)\cap\mathrm{L}^{r+1}(t,T;\wi\L^{p(r+1)})$, $\mathbb{P}-$a.s., and  for every $s\in[t,T]$, $\P-\text{a.s.}$, we obtain
     	\begin{align*}
     		\Y_n(s)&=\y+\int_t^s  \left(-\mu\A\Y_n(\tau)-\mathcal{B}(\Y_n(\tau))-\alpha\Y_n(\tau)-
     		\beta\mathcal{C}
     		(\Y_n(\tau))+\f(\tau)\right)\d\tau\nonumber\\&\quad+
     		\frac{1}{\sqrt{n}}\int_t^s\Q^{\frac12}\d\mathbf{W}(\tau), \ \text{ in }\H. 
     	\end{align*} 
     \end{definition}
     The existence and uniqueness of variational as well as strong solutions of \eqref{stap}  can be found in \cite[Theorems 3.10 and 4.2]{MTM8}.  The primary objective of this work is to establish the LDP for the solutions $\Y_n(\cdot)$ of \eqref{stap}, using viscosity solution techniques together with the framework developed by Feng and Kurtz (cf. \cite{JFTGK}). 
     
     \subsubsection{Methodology} For  reader’s convenience, we summarise the large deviation framework of Feng and Kurtz in Subsection \ref{FKLDP} and highlight the key results that will be used throughout the paper.  Our approach follows the classical theorem of Varadhan and Bryc, which states that the process $\Y_n(\cdot)$ satisfies the LDP in a metric space $(\mathcal{E},\mathpzc{d})$ if and only if the family $\{\Y_n(\cdot)\}_{n\in\N}$ is exponentially tight and the Laplace principle holds, that is,
     \begin{align*}
     		\lim\limits_{n\to+\infty}\frac{1}{n}\log\E\big[e^{-ng(\Y_n)}\big],
     \end{align*}
     exists for all bounded and continuous function $g$. We now outline the main steps of our methodology, which are elaborated in Sections \ref{detstchjb}, \ref{LaplaceLDP}, and \ref{LaDePr}.
 \vskip 0.2cm
 \noindent
 \textbf{Step-I:} \emph{Laplace limit at a single time (Section \ref{LaplaceLDP}).}  For the single time case, we set $\mathcal{E}=\H$ with the embedding $\V\hookrightarrow\H$ compact (see Subsection \ref{zerofunc} on function spaces). We define 
	\begin{align*}
	\mathpzc{U}_n(t,\y)=-\frac{1}{n}\log\E\big[e^{-ng(\Y_n(T))}\big].
\end{align*}
 By formally applying the It\^o formula, we show that  $\mathpzc{U}_n$ is the viscosity solution of the following second-order PDE (see Proposition \ref{Lapvscso}):
 	\begin{equation}\label{scdhjb}
 	\left\{
 	\begin{aligned}
 		&(\mathpzc{U}_n)_t+\frac{1}{2n}\mathrm{Tr}(\Q\D^2\mathpzc{U}_n)-
 		\frac12\|\Q^{\frac12}\D \mathpzc{U}_n\|_{\H}^2\\&\quad+ (-\mu\A\y-\mathcal{B}(\y)-\alpha\y-\beta\mathcal{C}(\y)+\f(t),\D \mathpzc{U}_n) =0, \ \text{ in } \ (0,T)\times\V, \\
 		&\mathpzc{U}_n(T,\y)=g(\y),
 	\end{aligned}
 	\right.
 \end{equation}
 Passing to the limit as $n\to+\infty$ in above equation, we obtain the following first order PDE:
 	\begin{equation}\label{fsthjb}
 	\left\{
 	\begin{aligned}
 		&\mathpzc{U}_t-
 		\frac12\|\Q^{\frac12}\D \mathpzc{U}_n\|_{\H}^2+ (-\mu\A\y-\mathcal{B}(\y)-\alpha\y-\beta\mathcal{C}(\y)+\f(t),\D \mathpzc{U}) =0, \ \text{ in } \ (0,T)\times\V, \\
 		&\mathpzc{U}(T,\y)=g(\y).
 	\end{aligned}
 	\right.
 \end{equation}
Equation \eqref{fsthjb} is a first order HJB equation associated with the optimal control problem for the deterministic CBF equations (see Subsection \ref{passnhjb}). We prove that the corresponding value function is the unique viscosity solution of \eqref{fsthjb} (see Proposition \ref{Lapvscso1}). Furthermore, we show that the viscosity solutions $\mathpzc{U}_n$ of \eqref{scdhjb} converges to the viscosity solution $\mathpzc{U}$ of \eqref{fsthjb}. This convergence allows us to identify the Laplace limit at a single time (see Corollary \ref{comparisonds}). The proof relies on techniques from the comparison principle (see Theorems \ref{comparison} and \ref{comparisonappr}), which play a central role in establishing the convergence of $\mathpzc{U}_n$. 
 
 \vskip 0.2cm
 \noindent
 \textbf{Step-II:} \emph{Laplace limit at multiple times (Subsection \ref{LapLmuL}).}  After establishing the single time case, the extension to the path space setting follows through a variation of the classical approach introduced in \cite{JFTGK}. In this case, we take $\mathcal{E}=\D([0,T];\H)$ and proceed according to the framework outlined in Subsection 
 \ref{FKLDP}. Exploting the Markov property of the solution $\Y_n(\cdot)$ of the SCBF equations \eqref{stap}, together with the properties of $\mathpzc{U}_n$ (Propositions \ref{LapunLh} and \ref{vfLDP3}), we establish the existence of the Lapalce limit at multiple times (see Subsection \ref{LapLmuL}). This analysis further allows us to explicitly identify the rate function, which is given in  \eqref{ratefun} (see Proposition \ref{exptightD}).
 
%Finally, we mention that the well-posedness results for strong as well as weak solutions of the problem \eqref{CBF} in dimensions $d\in\{2,3\}$, for supercritical case $(r>3)$ and for critical case $r=3$ with $2\beta\mu\geq1$, is known in the literature. 
%\begin{table}[ht]
%	\begin{tabular}{|c|c|c|c|c|}
%		\hline
%		\textbf{Case}&Dimension &$ r$& Conditions on 
%		$\mu$ \& $\beta$ \\
%		\hline
%		\textbf{I}&$d=2,3$ &$r\in(3,\infty)$&  for any  
%		$\mu>0$ and $\beta>0$  \\
%		\hline
%		\textbf{II}&$d=2,3$ &$r=3$&for $\mu>0$ and
%		$\beta>0$ with $2\beta\mu\geq1$ \\
%		\hline
%	\end{tabular}
%	\vskip 0.1 cm
%	\caption{Values of $\mu,\beta$ and $r$ for the comparison principle.}
%	\label{Table1}
%\end{table}
%Due to some technical difficulties (see Subsection \ref{sub-tech} below), we restrict the values of $r$ to prove \emph{the comparison principle} (Table \ref{Table1}) and to prove the \emph{existence of viscosity solutions} (Table \ref{Table2}), which are the main goals of this work.
%\begin{table}[ht]
%	\begin{tabular}{|c|c|c|c|c|}
%		\hline
%		\textbf{Case}&Dimension &$ r$& Conditions on 
%		$\mu$ \& $\beta$ \\
%		\hline
%		\textbf{I}&$d=2$ &$r\in(3,\infty)$&  for any  
%		$\mu>0$ and $\beta>0$  \\
%		\hline
%		\textbf{II}&$d=3$ &$r\in(3,5)$&  for any  
%		$\mu>0$ and $\beta>0$  \\
%		\hline
%		\textbf{III}&$d=3$ &$r=5$&for any  
%		$\mu>0$ and $\beta>0$ \\
%		\hline
%	\end{tabular}
%	\vskip 0.1 cm
%	\caption{Values of $\mu,\beta$ and $r$ for the existence of viscosity solution.}
%	\label{Table2}
%\end{table}
 
 \vskip 0.2cm
 \noindent
 \textbf{Step-III:} \emph{Verifying exponential tightness (Subsection \ref{exptghtpath}).} The next step is to establish the exponential tightness of the solutions $\{\Y_n(\cdot)\}_{n\in\N}$ of the SCBF equations \eqref{stap} in the path space $\D([0,T];\H)$, as proved in Theroem \ref{exptightLDP6}. To this end, we first verify that the family $\{\Y_n(\cdot)\}_{n\in\N}$ satisfies the exponential compact containment condition and is weakly exponentially tight. Combining these results with Theorem \ref{LDPsemg6}, we conclude that $\{\Y_n(\cdot)\}_{n\in\N}$  is exponentially tight in $\D([0,T];\H)$.
 
 \vskip 0.2cm
 \noindent
 \textbf{Step-IV:} \emph{LDP in $\D([0,T];\H)$ (Subsection \ref{exptghtpath}).}
 Finally, having verified all the conditions stated in Proposition \ref{exptightD}, we conclude that the family of solutions $\{\Y_n(\cdot)\}_{n\in\N}$ of \eqref{stap} satisfies the LDP in $\D([0,T];\H)$.

 After establishing the LDP in $\D([0,T];\H)$, the remaining task is to extend the result to $\C([0,T];\H)$. The crucial step in this transition is to verify that the family of solutions $\{\Y_n(\cdot)\}_{n\in\N}$ satisfies the $\C-$exponential tightness property (see Definition \ref{cextght}). Once this property is established, Theorem \ref{refthm} can be applied to obtain the desired LDP in $\C([0,T];\H)$. 
    In particular, this yields the following main result (see Theorem \ref{mainthm}):
     
 \begin{theorem}
 	(LDP in $\C([0,T];\H)$) Assume that $\Tr(\Q)<\infty$ and $\f:[0,T]\to\V$ is bounded and continuous. Then, for $r>3$ and $r=3$ with $2\beta\mu\geq1$ in $d\in\{2,3\}$, the sequence of stochastic processes $\{\Y_n(\cdot)\}_{n\geq1}$, where $\Y_n(\cdot)=\Y_n(\cdot;0,\y)$ is a solution to the SCBF system \eqref{stap} with $\Y_n(0)=\y$, satisfies the large deviation principle in $\C([0,T];\H)$ with the rate function $I$ given by \eqref{ratefun}. 
 	  	That is, 
 	  		\begin{itemize}
 		  		\item For each closed set $\mathcal{F}\subset\C([0,T];\H)$, we have
 		  		\begin{align*}
 			  			\limsup\limits_{n\to+\infty}\frac{1}{n}\log\P\{\omega\in\Omega: \Y_n(\cdot,\omega)\in\mathcal{F}\}\leq-\inf\limits_{\xi\in\mathcal{F}}I(\xi).
 			  		\end{align*}
 		  		\item For each open set $\mathcal{G}\subset\C([0,T];\H)$, we have
 		  		\begin{align*}
 			  			\liminf\limits_{n\to+\infty}\frac{1}{n}\log\P\{\omega\in\Omega: \Y_n(\cdot,\omega)\in\mathcal{G}\}\geq-\inf\limits_{\xi\in\mathcal{G}}I(\xi).
 			  		\end{align*}
 		  	\end{itemize}
 \end{theorem}

  \subsection{Organization of the paper} The remainder of this paper is structured as follows. In Section \ref{mathfunop}, we introduce the functional framework required for the study of SCBF equations \eqref{pscbf}, together with a review of fundamental concepts such as linear, bilinear, and nonlinear operators. Section \ref{stchpreLDP} provides essential stochastic preliminaries, followed in Subsection \ref{FKLDP} by the definition and framework of the large deviation principle (LDP). In Section \ref{abscon}, we establish well-posedness results for SCBF equations \eqref{stap} and derive exponential estimates (Proposition \ref{propexpest}). Section \ref{detstchjb} is devoted to proving a comparison principle (Theorem \ref{comparison}) for the Hamilton-Jacobi-Bellman (HJB) equation \eqref{thjbcomp} associated with SCBF equations \eqref{stap}. In Section \ref{LaplaceLDP}, we show the existence of a Laplace limit at a fixed time, which we characterise as the convergence of viscosity solutions to the corresponding second-order HJB equation \eqref{thjbcomp} (Propositions \ref{Lapvscso}, \ref{Lapvscso1}, and Corollary \ref{comparisonds}). Finally, by applying the Feng-Kurtz framework of LDP from Subsection \ref{FKLDP}, we conclude in Section \ref{LaDePr} with the proof of our main result (Theorem \ref{mainthm}).
    
    \section{Mathematical framework}\label{mathfunop}\setcounter{equation}{0}
    	We begin by introducing the essential function spaces that will be used throughout the paper, along with the linear and nonlinear operators required to derive the abstract formulation of the SCBF system given in \eqref{stap}. Our analysis is conducted within the functional framework developed in \cite{JCR1} and adapted to the present setting.
    	
	\subsection{Function spaces}\label{zerofunc}
	Let $\C_{\mathrm{p}}^{\infty}(\mathbb{T}^d;\R^d)$ denote the space of all infinitely differentiable  functions $\u$ satisfying periodic boundary conditions
	$\u(x+\mathrm{L}e_{i},\cdot) = \u(x,\cdot)$, for $x\in \R^d$. The Sobolev space  $\H_{\mathrm{p}}^s(\mathbb{T}^d):=\mathrm{H}_{\mathrm{p}}^s(\mathbb{T}^d;\mathbb{R}^d)$ is the completion of $\C_{\mathrm{p}}^{\infty}(\mathbb{T}^d;\R^d)$  with respect to the $\H^s$-norm and the norm on the space $\H_{\mathrm{p}}^s(\mathbb{T}^d)$ is given by $$\|\u\|_{{\H}^s_{\mathrm{p}}}:=\left(\sum_{0\leq|\boldsymbol\alpha|
		\leq s} \|\D^{\boldsymbol\alpha}\u\|_{\mathbb{L}^2(\mathbb{T}^d)}^2\right)^{1/2}.$$ 	
	It is known from \cite[Proposition 5.39]{JCR1} that the Sobolev space of periodic functions $\H_{\mathrm{p}}^s(\mathbb{T}^d)$, for $s\geq0$ is the same as 
	$$\left\{\u:\u=\sum_{k\in\mathbb{Z}^d}\y_{k}\mathrm{e}^{2\pi i k\cdot \xi /  \mathrm{L}},\ \overline{\u}_{k}=\u_{-k}, \  \|\u\|_{{\H}^s_f}:=\left(\sum_{k\in\mathbb{Z}^d}(1+|k|^{2s})|\u_{k}|^2\right)^{1/2}<\infty\right\}.$$ We infer from \cite[Propositions 5.38]{JCR1} that the norms $\|\cdot\|_{{\H}^s_{\mathrm{p}}}$ and $\|\cdot\|_{{\H}^s_f}$ are equivalent. 
	\begin{remark}
		{We are not assuming the zero mean condition for the velocity field unlike the case of NSE, since the absorption term $\beta|\u|^{r-1}\u$ does not preserve this property (see \cite{MTT}). Therefore, we cannot use the well-known Poincar\'e inequality and we have to deal with the  full $\H^1$-norm.} 
	\end{remark}
	Let us define 
	\begin{align*} 
		\mathcal{V}:=\{\u\in\C_{\mathrm{p}}^{\infty}(\mathbb{T}^d;\R^d):\nabla\cdot\u=0\}.
	\end{align*}
	We define the spaces $\H$ and $\widetilde{\L}^{p}$ as the closure of $\mathcal{V}$ in the Lebesgue spaces $\L^2(\mathbb{T}^d):=\mathrm{L}^2(\mathbb{T}^d;\R^d)$ and $\L^p(\mathbb{T}^d):=\mathrm{L}^p(\mathbb{T}^d;\R^d)$ for $p\in(2,\infty]$, respectively. We endow the space $\H$ with the inner product and norm of $\L^2(\mathbb{T}^d),$ and are denoted by 
	\begin{align*}
		&(\u,\v):=(\u,\v)_{\L^2(\mathbb{T}^d)}=\int_{\mathbb{T}^d}\u(\xi)\cdot\v(\xi)\d \xi\\ \text{ and } \  &\|\u\|_{\H}^2:=\|\u\|_{\L^2(\mathbb{T}^d)}^2=\int_{\mathbb{T}^d}|\u(\xi)|^2
		\d \xi, \ \text{ for } \ \u,\v\in\H.
	\end{align*}
	For $p\in(2,\infty)$, the space $\widetilde{\L}^{p}$ is endowed with the norm of $\L^p(\mathbb{T}^d)$, which is defined by 
	\begin{align*}
		\|\u\|_{\widetilde{\L}^p}^p:=\|\u\|_{\L^p(\mathbb{T}^d)}^p
		=\int_{\mathbb{T}^d}|\u(\xi)|^p\d \xi \  \text{ for } \ \u\in\wi\L^p.
	\end{align*}
	For $p=\infty$, the space $\widetilde{\L}^{\infty}$ is endowed with the norm of $\L^{\infty}(\mathbb{T}^d)$, which is given by 
		\begin{align*}
		\|\u\|_{\widetilde{\L}^{\infty}}:=\|\u\|_{\L^{\infty}(\mathbb{T}^d)}
		=\esssup_{\xi\in\mathbb{T}^d}|\u(\xi)| \  \text{ for } \ \u\in\wi\L^{\infty}.
	\end{align*}
	We also define the space $\V$ as the closure of $\mathcal{V}$ in the Sobolev space $\H^1_{\mathrm{p}}(\mathbb{T}^d)$. We equip the space $\V$ with the inner product
	\begin{align*}
		(\u,\v)_{\V}&:=(\u,\v)_{\L^2(\mathbb{T}^d)}+(\nabla\u,\nabla\v)_{\L^2(\mathbb{T}^d)}\nonumber\\&=
		\int_{\mathbb{T}^d}\u(\xi)\cdot\v(\xi)\d \xi+ \int_{\mathbb{T}^d}\nabla\u(\xi)\cdot\nabla\v(\xi)\d \xi\ \text{ for } \ \u,\v\in\V,
     \end{align*}
     and the norm
     \begin{align*}
		\|\u\|_{\V}^2:=\|\u\|_{\L^2(\mathbb{T}^d)}^2+\|\nabla\u\|_{\L^2(\mathbb{T}^d)}^2=\int_{\mathbb{T}^d}|\u(\xi)|^2\d \xi + \int_{\mathbb{T}^d}|\nabla\u(\xi)|^2\d \xi \ \text{ for } \ \u,\v\in\V.
	\end{align*}
%	
%	Then, we characterize the spaces $\H$, $\widetilde{\L}^p$ and $\V$ with the norms  $$\|\y\|_{\H}^2:=\int_{\mathbb{T}^d}|\y(x)|^2\d x,\quad \|\y\|_{\widetilde{\L}^p}^p:=\int_{\mathbb{T}^d}|\y(x)|^p\d x\ \text{ and }\ \|\y\|_{\V}^2:=\int_{\mathbb{T}^d}(|\y(x)|^2+|\nabla\y(x)|^2)\d x,$$ respectively. 
	Let $\langle \cdot,\cdot\rangle $ represent the induced duality between the spaces $\V$  and its dual $\V^*$ as well as $\widetilde{\L}^p$ and its dual $\widetilde{\L}^{p'}$, where $\frac{1}{p}+\frac{1}{p'}=1$. Note that $\H$ can be identified with its own dual $\H^*$. From \cite[Subsection 2.1]{FKS}, we have that the sum space $\V^*+\widetilde{\L}^{p'}$ is well defined and  is a Banach space with respect to the norm 
	\begin{align*}
		\|\u\|_{\V^*+\widetilde{\L}^{p'}}&:=\inf\{\|\u_1\|_{\V^*}+\|\u_2\|_{\wi\L^{p'}}:\u=\u_1+\u_2, \u_1\in\V^* \ \text{and} \ \u_2\in\wi\L^{p'}\}\nonumber\\&=
		\sup\left\{\frac{|\langle\u_1+\u_2,\f\rangle|}{\|\f\|_{\V\cap\widetilde{\L}^p}}:\boldsymbol{0}\neq\f\in\V\cap\widetilde{\L}^p\right\},
	\end{align*}
	where $\|\cdot\|_{\V\cap\widetilde{\L}^p}:=\max\{\|\cdot\|_{\V}, \|\cdot\|_{\wi\L^p}\}$ is a norm on the Banach space $\V\cap\widetilde{\L}^p$.  %Also the norm $\max\{\|\cdot\|_{\V}, \|\cdot\|_{\wi\L^p}\}$ is equivalent to the norms  $\|\cdot\|_{\V}+\|\cdot\|_{\widetilde{\L}^{p}}$ and $\sqrt{\|\cdot\|_{\V}^2+\|\cdot\|_{\widetilde{\L}^{p}}^2}$ on the space $\V\cap\widetilde{\L}^p$. Furthermore, we have
	%$$
%	(\V^*+\widetilde{\L}^{p'})^*\cong	\V\cap\widetilde{\L}^p \  \text{and} \ (\V\cap\widetilde{\L}^p)^*\cong\V^*+\widetilde{\L}^{p'}.
	%$$
	Moreover, we have the continuous embeddings $\V\cap\widetilde{\L}^p\hookrightarrow\V\hookrightarrow\H\cong\H^*\hookrightarrow\V^*\hookrightarrow\V^*+\widetilde{\L}^{p'},$ where the embedding $\V\hookrightarrow\H$ is compact. 
	
%      Apart from these functional settings, we use the following function spaces while proving the comparison principle and showing the existence and uniqueness of a viscosity solution:
%
%		We denote by $\mathrm{C}^2(\H)$ (and $\mathrm{C}^2(\V)$), the space of all functions which are continuous on $\H$ (and $\V$) together with all their Fr\'echet derivatives up to order $2$. For a given $0\leq t<T$, we denote by $\mathrm{C}^{1,2}((t,T)\times\H)$, the space of all functions $\varphi:(t,T)\times\H\to\R$ for which $\varphi_t$ and $\D\varphi, \D^2\varphi$ (the Fr\'echet derivatives of $\varphi$ with respect to $\x\in\H$) exist and are uniformly continuous on closed and bounded subsets of $(t,T)\times\H$.

	\subsection{Linear operator}\label{linope}
%	Let $\mathcal{P}_p: \L^p(\mathbb{T}^d) \to\wi\L^p,$ $p\in[1,\infty)$ be the Helmholtz-Hodge (or Leray) projection operator (cf.  \cite{DFHM}, etc.).	Note that $\mathcal{P}_p$ is a bounded linear operator and for $p=2$,  $\mathcal{P}:=\mathcal{P}_2$ is an orthogonal projection (see \cite[Section 2.1]{JCR}). 
	We define the Stokes operator 
	\begin{equation*}
		\left\{
		\begin{aligned}
			\A\u&:=-\mathcal{P}\Delta\u=-\Delta\u,\;\u\in\mathcal{D}(\A),\\
			\mathcal{D}(\A)&:=\V\cap{\H}^{2}_\mathrm{p}(\mathbb{T}^d),
		\end{aligned}
		\right.
	\end{equation*}
	where $\mathcal{P}:\L^2(\mathbb{T}^d)\to\H$ is the Leray-Helmholtz orthogonal projection operator, which is bounded and self-adjoint (see \cite[Section 2.1]{JCR}). For the Fourier expansion $\u=\sum\limits_{k\in\mathbb{Z}^d} e^{2\pi i k\cdot \xi} \u_{k} ,$ by making the use of Parseval's identity and the definition of 
	$\|\cdot\|_{\H^2_\mathrm{p}}-$norm, one can show that 
	$\H^2_\mathrm{p}(\mathbb{T}^d)=\mathcal{D}(\I+\A):=\V_2$.
%	\begin{align*}
%		\|\u\|_{\H}^2=\sum\limits_{k\in\mathbb{Z}^d} |\u_{k}|^2 \  \text{and} \ \|\A\u\|_{\H}^2=(2\pi)^4\sum_{k\in\mathbb{Z}^d}|k|^{4}|\u_{k}|^2.
%	\end{align*}
%	Therefore, we have 
%	\begin{align*}
%		\|\u\|_{\H^2_\mathrm{p}}^2=\sum_{k\in\mathbb{Z}^d}|\u_{k}|^2+\sum_{k\in\mathbb{Z}^d}|k|^{4}|\u_{k}|^2= \|\u\|_{\H}^2+\frac{1}{(2\pi)^4}\|\A\u\|_{\H}^2\leq\|\u\|_{\H}^2+\|\A\u\|_{\H}^2.
%	\end{align*}
%	Moreover, by the definition of $\|\cdot\|_{\H^2_\mathrm{p}}$, we have $	\|\u\|_{\H^2_\mathrm{p}}^2\geq\|\u\|_{\H}^2+\|\A\u\|_{\H}^2$ and hence it is immediate that both the norms are equivalent and  $\mathcal{D}(\I+\A)=\H^2_\mathrm{p}(\mathbb{T}^d)=:\V_2$. 

	\begin{remark}\label{rg3L3r}
   1.) For $d=2$, by using the Sobolev embedding, $\H_{\mathrm{p}}^1(\mathbb{T}^d)\hookrightarrow\L^p(\mathbb{T}^d)$, for all $p\in[2,\infty)$, we find 
		\begin{align*}
			\|\u\|^{r+1}_{\L^{p(r+1)}(\mathbb{T}^d)}&=\||\u|^{\frac{r+1}{2}}\|_{\L^{2p}(\mathbb{T}^d)}^2\leq\mathpzc{C}_e \||\u|^{\frac{r+1}{2}}\|_{\H_{\mathrm{p}}^1(\mathbb{T}^d)}^2\nonumber\\&\leq\mathpzc{C}
			\bigg(\int_{\mathbb{T}^d}|\nabla\u(\xi)|^2|\u(\xi)|^{r-1}\d \xi+ \int_{\mathbb{T}^d}|\u(\xi)|^{r+1}\d \xi\bigg),
		\end{align*}
		for all $\u\in\V_2$ and for any $p\in[2,+\infty)$. 
		
		3.) Similarly, for $d=3$, by the Sobolev embedding $\H_{\mathrm{p}}^1(\mathbb{T}^d)\hookrightarrow\L^6(\mathbb{T}^d)$, we find
		\begin{align}\label{371}
			\|\u\|^{r+1}_{\L^{3(r+1)}(\mathbb{T}^d)}&=\||\u|^{\frac{r+1}{2}}\|_{\L^{6}(\mathbb{T}^d)}^2\leq \mathpzc{C}\||\u|^{\frac{r+1}{2}}\|_{\H_{\mathrm{p}}^1(\mathbb{T}^d)}^2
			\nonumber\\&\leq C\bigg(\int_{\mathbb{T}^d}|\nabla\u(\xi)|^2|\u(\xi)|^{r-1}\d \xi+ \int_{\mathbb{T}^d}|\u(\xi)|^{r+1}\d \xi\bigg), 
		\end{align}
		for all $\u\in\V_2$. 
	\end{remark}

	\subsection{Bilinear operator}
		Let $b(\cdot,\cdot,\cdot):\V\times\V\times\V\to\R$ be a continuous trilinear form defined by
		\begin{align*}
			b(\u,\v,\w)=\int_{\mathbb{T}^d}(\u(\xi)\cdot\nabla)\v(\xi)\cdot\w(\xi)\d \xi.
		\end{align*} 
		By the Riesz representation theorem, we can define $\mathcal{B}(\cdot,\cdot):\V\times\V\to\R$ a continuous bilinear operator such that $\langle\mathcal{B}(\u,\v),\w\rangle=b(\u,\v,\w)$ for all $\u,\v,\w\in\V$, which also satisfies (see \cite {Te})
		\begin{align}\label{syymB}
			\langle\mathcal{B}(\u,\v),\w\rangle=-\langle\mathcal{B}(\u,\w),\v\rangle \ \text{ and } \
			\langle\mathcal{B}(\u,\v),\v\rangle=0,
		\end{align}
		for any $\u,\v,\w\in\V$. We also denote $\mathcal{B}(\u) = \mathcal{B}(\u, \u)$. 
%		Note that if $\u,\v\in\H$ are such that $(\u\cdot\nabla)\v=\sum\limits_{j=1}^d u_j\frac{\partial v_j}{\partial \xi_j}\in\L^2(\mathbb{T}^d)$, then $\mathcal{B}(\u,\v)=\mathcal{P}[(\u\cdot\nabla)\v]$.

		\begin{remark}\label{BLrem}\cite[Theorem 2.5]{SMTM}
			1.) In view of \eqref{syymB}, along with H\"older's and Young's inequalities, we calculate
			\begin{align}\label{syymB1}
				|\langle\mathcal{B}(\u)-\mathcal{B}(\v),\u-\v\rangle|\leq
				\frac{\mu }{2}\|\nabla(\u-\v)\|_{\H}^2+\frac{1}{2\mu }\||\v|(\u-\v)\|_{\H}^2.
			\end{align} 
			Using H\"older's and Young's inequalities, we estimate the term $\||\v|(\u-\v)\|_{\H}^2$ as
			\begin{align}\label{syymB2}
				\int_{\mathbb{T}^d}|\v(\xi)|^2|\u(\xi)-\v(\xi)|^2\d \xi &=\int_{\mathbb{T}^d}|\v(\xi)|^2|\u(\xi)-\v(\xi)|^{\frac{4}{r-1}}|\u(\xi)-\v(\xi)|^{\frac{2(r-3)}{r-1}}\d \xi\nonumber\\&\leq\frac{\beta\mu }{2}\||\v|^{\frac{r-1}{2}}(\u-\v)\|_{\H}^2+\frac{r-3}{r-1}\left[\frac{4}{\beta\mu (r-1)}\right]^{\frac{2}{r-3}}\|\u-\v\|_{\H}^2,
			\end{align}
			for $r>3$. Using \eqref{syymB2} in \eqref{syymB1}, we find 
			\begin{align}\label{3.4}
				|\langle\mathcal{B}(\u)-\mathcal{B}(\v),\u-\v\rangle|\leq
				\frac{\mu }{2}\|\nabla(\u-\v)\|_{\H}^2 +\frac{\beta}{4}\||\v|^{\frac{r-1}{2}}(\u-\v)\|_{\H}^2 +\varrho\|\u-\v\|_{\H}^2,
			\end{align}
			where \begin{align}\label{eqn-varrho}
				\varrho:=\frac{r-3}{2\mu(r-1)}\left[\frac{4}{\beta\mu (r-1)}\right]^{\frac{2}{r-3}}.
				\end{align}
			
			2.) In a similar way, one can establish the following inequality:
			\begin{align}\label{syymB3}
				|(\B(\u),\A\u)|\leq\frac{\mu}{2}\|\A\u\|_{\H}^2+\frac{\beta}{4} \||\u|^{\frac{r-1}{2}}\nabla\u\|_{\H}^2+\varrho\|\nabla\u\|_{\H}^2.
			\end{align}
		\end{remark}

	 \subsection{Nonlinear operator}\label{C1prop}
		Let us define the operator $$\mathcal{C}(\u):=\mathcal{P}(|\u|^{r-1}\u)\ \text{ for }\ \u\in\V\cap\L^{r+1}.$$   Since the projection operator $\mathcal{P}$ is bounded from $\H^1$ into itself (cf. \cite[Remark 1.6]{Te}), the operator $\mathcal{C}(\cdot):\V\cap\widetilde{\L}^{r+1}\to\V^*+\widetilde{\L}^{\frac{r+1}{r}}$ is well-defined and we have $\langle\mathcal{C}(\u),\u\rangle =\|\u\|_{\widetilde{\L}^{r+1}}^{r+1}.$ 
		
		We now state a lemma on the monotonicity properties of the nonlinear operator $\mathcal{C}(\cdot)$, which plays a recurring role in our analysis.
		\begin{lemma}\label{monopropC}\cite[Subsection 2.4]{SMTM}
			For every $r\geq1$ and for all $\u,\v,\w\in\widetilde{\L}^{r+1}$, the nonlinear operator $\mathcal{C}(\cdot)$ satisfies following estimates:
			\begin{align}\label{monoC1}
				\langle\mathcal{C}(\u)-\mathcal{C}(\v),\w\rangle\leq
				r\left(\|\u\|_{\widetilde{\L}^{r+1}}+\|\v\|_{\widetilde{\L}^{r+1}}\right)^{r-1}\|\u-\v\|_{\widetilde{\L}^{r+1}}\|\w\|_{\widetilde{\L}^{r+1}}
			\end{align}
			and 
			\begin{align}\label{monoC2}
				\langle\mathcal{C}(\u)-\mathcal{C}(\v),\u-\v\rangle\geq \frac{1}{2}\||\u|^{\frac{r-1}{2}}(\u-\v)\|_{\H}^2+\frac{1}{2}\||\v|^{\frac{r-1}{2}}(\u-\v)\|_{\H}^2\geq\frac{1}{2^{r-1}}\|\u-\v\|_{\wi\L^{r+1}}^{r+1}.
			\end{align}
	\end{lemma} 

	\begin{lemma}\label{monoest}\cite[Theorem 2.5]{SMTM}
	Let $d\in\{2,3\}$. For $r>3$, we have the following monotonicity estimate:
	\begin{align}\label{monoest1}
	&\langle\mu\A\u+\mathcal{B}(\u)+\beta\mathcal{C}(\u)-\mu\A\v-\mathcal{B}(\v)-\beta\mathcal{C}(\v), \u-\v\rangle+\varrho\|\u-\v\|_{\H}^2
	\nonumber\\&\geq 
	\frac{1}{2}\||\u|^{\frac{r-1}{2}}(\u-\v)\|_{\H}^2+
	\frac{\mu }{2}\|\nabla(\u-\v)\|_{\H}^2,
	\end{align}
	where $\varrho$ is the constant given in \eqref{eqn-varrho}. Moreover, for $r=3$ with $2\beta\mu\geq1$, we have following global monotonicity estimate:
	\begin{align*}
		&\langle\mu\A\u+\mathcal{B}(\u)+\beta\mathcal{C}(\u)-\mu\A\v-\mathcal{B}(\v)-\beta\mathcal{C}(\v), \u-\v\rangle
		\nonumber\\&\geq 
		\frac{1}{2}\left(\beta-\frac{1}{2\mu }\right)\|\v(\u-\v)\|_{\H}^2.
	\end{align*}
	\end{lemma}
	\begin{remark}
		Note that for $r>3$ in both $d=2,3$, the presence of the linear damping term $\alpha\u$ strengthens the local monotonicity estimate \eqref{monoest}, yielding the following global monotonicity estimate:
		\begin{align*}
			&\langle\mu\A\u+\mathcal{B}(\u)+\alpha\u+\beta\mathcal{C}(\u)-\mu\A\v-\mathcal{B}(\v)-\alpha\u-\beta\mathcal{C}(\v), \u-\v\rangle
			\nonumber\\&\geq 
			\frac{1}{2}\||\u|^{\frac{r-1}{2}}(\u-\v)\|_{\H}^2+
			\frac{\mu }{2}\|\nabla(\u-\v)\|_{\H}^2+(\alpha-\varrho)\|\u-\v\|_{\H}^2,
		\end{align*}
		provided $\alpha\geq\varrho$.
	\end{remark}
	
	\begin{remark}[{\cite[Lemma 2.1]{KWH}}]
		We shall frequently use the following identity on $\mathbb{T}^d$:
		\begin{align}\label{torusequ}
		(\mathcal{C}(\u),\A\u)=\||\u|^{\frac{r-1}{2}}\nabla\u\|_{\H}^{2} +4\left[\frac{r-1}{(r+1)^2}\right]\|\nabla|\u|^{\frac{r+1}{2}}\|_{\H}^{2}.
%		\nonumber\\&\leq
%		r\||\u|^{\frac{r-1}{2}}\nabla\u\|_{\H}^{2}
		\end{align}
	\end{remark}

	\subsection{Some useful functional inequalities}
	The following inequalities and definitions are frequently used throughout the paper: 
	
	1.) \emph{Interpolation inequality:} Let $0\leq s_1\leq s\leq s_2\leq\infty$ and $0\leq\theta\leq1$ be such that $\frac{1}{s}=\frac{\theta}{s_1}+\frac{1-\theta}{s_2}$. Then for $\u\in\L^{s_2}(\mathbb{T}^d)$, we have
	\begin{align*}
		\|\u\|_{\L^s}\leq\|\u\|_{\L^{s_1}}^{\theta}\|\u\|_{\L^{s_2}}^{1-\theta}.
	\end{align*} 
	
	5.) \emph{Agmon's inequality:} For all $\u\in\H^2_{\mathrm{p}}(\mathbb{T}^d)$, $d\in\{2,3\}$, we have 
	\begin{align*}
		\|\u\|_{\L^{\infty}(\mathbb{T}^d)}\leq C\|\u\|_{\H}^{1-\frac{d}{4}}
		\|\u\|_{\H^2_{\mathrm{p}}(\mathbb{T}^d)}^{\frac{d}{4}}
		=\|\u\|_{\H}^{1-\frac{d}{4}}
		\|(\I+\A)\u\|_{\H}^{\frac{d}{4}}.
	\end{align*}

  	\section{Stochastic preliminaries}\label{stchpreLDP} \setcounter{equation}{0}
     Let $(\Omega,\mathscr{F},\P)$ be a complete probability space equipped with the filtration $\{\mathscr{F}_s^t\}_{s\geq t}$ such that 
    \begin{itemize}
    	\item $\{\mathscr{F}_s^t\}_{s\geq t}$ is \emph{right-continuous}, that is, for all $s\geq t$, we have $\mathscr{F}_{s+}^t:=\bigcup\limits_{\mathpzc{r}>s} \mathscr{F}_{\mathpzc{r}}^t =\mathscr{F}_s^t$.
    	\item $\mathscr{F}_t^t$ contains all $\P$-null sets of $\mathscr{F}$.
    \end{itemize}
    Let $\Q$ be a bounded, self-adjoint, and non-negative linear operator on $\H$ such that $\Tr(\Q)<\infty$. Let $\mathbf{W}$ denote an $\H$-valued $\Q$-Wiener process satisfying $\mathbf{W}(t)=0,$ $\P-$a.s.
    We shall work under the following standing assumption throughout the article:
    \begin{hypothesis}\label{trQ1}
    	$\A^{\frac12}\Q^{\frac12}$ is a Hilbert-Schimdt operator.
    \end{hypothesis}
    
  Let us write $\Q_1:=\A^{\frac12}\Q\A^{\frac12}$. Since both $\A^{\frac12}$ and $\Q^{\frac12}$ are self-adjoint, Hypothesis \ref{trQ1} directly implies that 
    \begin{align*}
    	\Tr(\Q_1)=\Tr(\A^{\frac12}\Q\A^{\frac12})=
    	\Tr(\A^{\frac12}\Q^{\frac12}\Q^{\frac12}\A^{\frac12})
    	=\Tr\big(\A^{\frac12}\Q^{\frac12}\big(\A^{\frac12}\Q^{\frac12}\big)^{*}\big)
    	=\|\A^{\frac12}\Q^{\frac12}\|_{\mathscr{L}_2(\H)}^2<\infty,
    \end{align*}
   where $\mathscr{L}_2(\H)$ is the Hilbert space of all Hilbert-Schmidt operators on $\H$.
   
%    where Hypothesis \ref{trQ1} also says that $\Q_1:=\A^{\frac12}\Q\A^{\frac12}$ is densely defined and it can be extended to a bounded linear operator, still denoted by $\Q_1$ which is of trace-class. 
    
    The $5$-tuple $$\nu:=(\Omega,\mathscr{F},\{\mathscr{F}_s^t\}_{s\geq t},\P,\mathbf{W}),$$ introduced above, is referred to as a \emph{generalized reference probability space} (see \cite[Definition 1.100, Chapter 1, pp. 35]{GFAS}). 
    %\emph{We assume that the generalized reference probability space $\nu$ is fixed.} 
    We denote by $M^2_\nu(t,T;\H)$ (a subset of $\mathrm{L}^2((t,T)\times\Omega;\H)$), the space of all $\H$-valued progressively measurable processes $\Y(\cdot)$ such that
    \begin{align*}
    	\|\Y(\cdot)\|_{M^2_\nu(t,T;\H)}:=\left(\E\left[\int_t^T \|\Y(s)\|_{\H}^2\d s\right]\right)^{\frac12}<+\infty.
    \end{align*}
    The notation $M^2_\nu(t,T;\H)$ emphasizes the dependence on the generalized reference probability space $\nu$. Processes in $M^2_\nu(t,T;\H)$ are identified if they are equal $\P\otimes\d t$-a.e.
    
     Let us now provide some preliminaries on large deviation theory which is used throughout this work and has been taken from \cite{ABPD,JFTGK}.
    \subsection{Large deviation principle}\label{mainLDP}
    Let $(\mathcal{E},\mathpzc{d})$ be a Polish space, which is a complete and separable metric space, and denote by $\mathscr{B}(\mathcal{E}),$ the $\sigma-$field of Borel sets in $\mathcal{E}$.
    Let $\{X_n\}_{n\in\N}$ be a sequence of $\mathcal{E}-$valued random variables defined on a probability space $(\Omega,\mathscr{F},\P)$. The theory of large deviations is the study of probabilities $\P_n(\B):=\P\{X_n\in\B\}, \ \B\in\mathscr{B}(\mathcal{E})$, which converges to zero exponentially fast as $n\to+\infty$. The exponential decay of these probabilities is characterised by a \emph{rate function}, defined as follows: 
    \begin{definition}
    	A function $\mathcal{I}:\mathcal{E}\to[0,+\infty]$ is called a \emph{rate function} if $\mathcal{I}$ is lower semicontinuous function. In addition, if for each $r<+\infty$, the level set $\{x\in\mathcal{E}:\mathcal{I}(x)\leq r\}$ is a compact subset of $\mathcal{E}$, then we say that the rate function $\mathcal{I}$ is \emph{good}.
    \end{definition}
    %     Since a function having compact level sets is automatically lower semicontinuous and it attains its infimum on any nonempty closed set, a rate function satisfies these properties.
    Next, we introduce the notion of  large deviation principle (LDP). For any measurable set $\B\subset\mathcal{E}$, we denote $\inf\limits_{x\in\B}\mathcal{I}(x):=\mathcal{I}(\B)$. 
    %    Moreover, for the sequence of random variables $\{X_n\}_{n\in\N}$, we write 
    %    \begin{align*}
    	%    	\P_n(\B):=\P\{\omega\in\Omega: X_n(\omega)\in\B\}, \ \text{ for } \ \B\subset\mathcal{S},
    	%    \end{align*}
    \begin{definition}
    	Let $\mathcal{I}$ be a rate function on $\mathcal{E}$. The sequence of $\mathcal{E}-$valued random variables $\{X_n\}_{n\in\N}$ is said to satisfy the \emph{large deviation principle} on $\mathcal{E}$ with rate function $\mathcal{I}$ if the following two conditions hold:
    	\begin{itemize}
    		\item \textbf{Large deviation upper bound:} For each closed set $F\subset\mathcal{E}$, we have
    		\begin{align*}
    			\limsup\limits_{n\to+\infty}\frac{1}{n}\log\P_n(F)\leq-\mathcal{I}(F).
    		\end{align*}
    		\item \textbf{Large deviation lower bound:} For each open set $G\subset\mathcal{E}$, we have
    		\begin{align*}
    			\liminf\limits_{n\to+\infty}\frac{1}{n}\log\P_n(F)\geq-\mathcal{I}(G).
    		\end{align*}
    	\end{itemize}
    \end{definition}
  In proving large deviation principles, it is useful to introduce the notion of exponential tightness, which plays a role analogous to tightness in the theory of weak convergence.
    \begin{definition}
    	A sequence of $\mathcal{E}-$valued random variables $\{X_n\}_{n\in\N}$  is said to be \emph{exponentially tight} if for each $\eps>0$, there exists a compact set $K=K(\eps)\subset\mathcal{E}$ such that
    	\begin{align*}
    		\limsup\limits_{n\to+\infty} \frac{1}{n}\P\{X_n\notin K^c\}\leq-\eps.
    	\end{align*}
    \end{definition}
    The following framework of large deviation principle, which we adopt in this work, is taken from \cite{JFTGK}. For reader's convenience, we present here the key aspects of this framework:
    \subsection{Feng and Kurtz framework for LDP({\cite{JFTGK}})}\label{FKLDP}
    Let us denote by $\D([0,+\infty);\mathcal{E})$, the set of $\mathcal{E}-$valued functions defined on $[0,+\infty)$, which are right continuous and have left limits at every $t\in[0,+\infty)$. We endow the space $\D([0,+\infty);\mathcal{E})$ with the Skorohod $\J-$topology (see \cite{Ajakub,AJMM} for definition and its properties). For notational convienience, we call $\D([0,+\infty);\mathcal{E})$, the \emph{path space}. We write $\C_b(\mathcal{E})$ for the set of all bounded continuous functions defined on $\mathcal{E}$. 
    The following definition is used in the sequel, which is required for the exponential tightness in path space and the existence of a Laplace limit (see Theorem \ref{LDPsemg6}).
    \begin{definition}[{\cite[Definition 3.18, Chapter 3]{JFTGK}}]\label{isopoi}
    	A collection of functions $\mathscr{A}\subset\C_b(\mathcal{E})$ is said to \emph{isolates points in $\mathcal{E}$} if, for every $x\in\mathcal{E}$, every $\eps>0$, and every compact set $\mathcal{K}\subset\mathcal{E}$, there exists $g\in\mathscr{A}$ such that
    	\begin{itemize}
    		\item $g(x)$ is small in magnitude: $|g(x)|<\eps$,
    		\item $g$ is non-positive on $\mathcal{K}$: 
    		$\sup\limits_{y\in\mathcal{K}} g(y)\leq0,$
    		\item $g$ attains significantly negative values outside an $\eps-$neighbourhood of $x$:
    		\begin{align*}
    			\sup\limits_{y\in \mathcal{K}\cap B_{\eps}^c(x)} g(y)<-\frac{1}{\eps},
    		\end{align*}
    		where $B_{\eps}^c(x)$ is the complement of $\eps-$open ball with centre at $x$.
    	\end{itemize}
    	If $g$ satisfies these conditions, we say that $g$ isolates $x$ relative to $\eps$ and $\mathcal{K}$. Moreover, if 
    	\begin{align*}
    		\sup\limits_{g\in\mathscr{A}}\sup\limits_{y} g(y)<+\infty,
    	\end{align*}
    	then the set $\mathscr{A}$ is said to be bounded above.
    \end{definition}

    \subsubsection{Verifying the exponential tightness in path space $\D([0,+\infty);\mathcal{E})$.} The following result demonstrates the exponential tightness in the space $\D([0,+\infty);\mathcal{E})$. It allows us to reduce the problem of verifying the exponential tightness of the $\mathcal{E}-$valued process $\{X_n(\cdot)\}_{n\geq1}$ to that of real valued process  $\{g(X_n(\cdot))\}_{n\geq1}$, where $g:\mathcal{E}\to\R$ is some real-valued function.

    \begin{theorem}\cite[Theorem 4.4, Chapter 4]{JFTGK}\label{LDPsemg6}
    	Let $(\mathcal{E},\mathpzc{d})$ be a complete separable metric space. The sequence of $\mathcal{E}-$valued stochastic processes $\{X_n(\cdot)\}_{n\geq1}$ is \emph{exponentially tight} in $\D([0,+\infty);\mathcal{E})$ if and only if the following two conditions are satisfied:
    	\begin{itemize}
    		\item \textbf{Exponential compact containment.} For every $T>0$ and $M>0$, there exists a compact set $\mathcal{K}_{M,T}\subset\mathcal{E}$ such that
    		\begin{align*}
    			\limsup\limits_{n\to+\infty}
    			\frac{1}{n}\log\P\big(\{\text{there exists } \ 0\leq t\leq T \ \text{ such that } \ X_n(t)\notin\mathcal{K}_{M,T}\}\big)\leq-M.
    		\end{align*}
    		\item \textbf{Weak exponential tightness.} There exists a family $\mathscr{A}\subset\C(\mathcal{E})$ of continuous functions which is closed under addition and isolates points in $\mathcal{E}$ such that for every $g\in\mathscr{A}$, the sequence of real valued processes $\{g(X_n)\}_{n\geq1}$ is exponentially tight in $\D([0,+\infty);\R)$.
    	\end{itemize}
    \end{theorem}
    \subsubsection{Establishing  the Laplace limit}
    To demonstrate the existence of the Laplace limit, we follow the approach as stated in \cite{AS2}. Specifically, we first establish the Laplace limit at a fixed time by identifying it with the limit of the viscosity solution to the associated second-order HJB equation (see Section \ref{LaplaceLDP} for the detailed procedure). Then, one may identify the Laplace limit at multiple times from the Markov property of the solution and Proposition \ref{vfLDP3}.

    \subsubsection{Proving the large deviation principle in the path space $\D([0,+\infty);\mathcal{E})$}
    Once the exponential tightness and the existence of Laplace limit are established, 
    the large deviation principle in the space $\D([0,+\infty);\mathcal{E})$ follows from the following result:
    \begin{proposition}[{\cite[Corollary 4.29, Chapter 4]{JFTGK}}]\label{exptightD}
    	Let $\mathscr{A}\subset\C_b(\mathcal{E})$ be a bounded above set of functions and isolates points. Assume that the following conditions hold:
    	\begin{itemize}
    		\item \textbf{Exponential tightness.} The sequence of $\mathcal{E}-$valued stochastic processes $\{X_n\}_{n\geq1}$ is exponentially tight in $\D([0,+\infty);\mathcal{E})$.
    		\item \textbf{Existence of a Laplace limit.} For every finite collection of times 
    		$0\leq t_1\leq\ldots\leq t_m$ and functions $g_1,\ldots,g_m\in\mathscr{A}$, the limit 
    		\begin{align}\label{expLmt}
    			\lim\limits_{n\to+\infty}\frac{1}{n}\log\E
    			\big[e^{n(g_1(X_n(t_1))+\ldots+g_m(X_n(t_m)))}\big]
    		\end{align}
    		exists.
    	\end{itemize}  
    Let $x\in \D([0,+\infty);\mathcal{E})$, and let $\triangle_x$ denote the set of discontinuities of $x$. 	Then, $\{X_n\}_{n\geq1}$ satisfies the large deviation principle in 
    	$\D([0,+\infty);\mathcal{E})$ with good rate function given by 
    		\begin{align}\label{ratefun}
    			I(x)=\sup\limits_{m}\sup\limits_{\{t_1,\ldots,t_m\}\subset\triangle_x^c}
    			\sup\limits_{g_1,\ldots,g_m\in\mathscr{A}}
    			\bigg\{&g(x(t_1))+\cdots+g_m(x(t_m))
    			\nonumber\\&\quad-
    			\lim\limits_{n\to+\infty}\frac{1}{n}\log\E
    			\big[e^{n(g_1(X_n(t_1))+\cdots+g_m(X_n(t_m)))}\big]\bigg\}.
    		\end{align}
    \end{proposition}
    
    %    \begin{remark}
    	%    	It is important to note that \cite[Corollary 4.29, Chapter 4]{JFTGK} provides the general formula for the rate function, contingent upon the existence of the limits mentioned in \eqref{expLmt}.
    	%    \end{remark}

    \subsubsection{Large deviation principle in the compact uniform topology}
    As outlined in \cite[Section 4, Chapter 4]{JFTGK}, this follows from the following two key ingredients:
    \begin{itemize}
    	\item[(i)] The large deviation principle in the path space $\D([0,+\infty);\mathcal{E})$;
    	\item[(ii)] $\C-$exponential tightness of the process $\{X_n(\cdot)\}_{n\geq1}$. 
    \end{itemize} 
    Let us first define the notion of $\C-$exponential tightness, which plays a crucial role in transferring the large deviation principle from Skorohod space to continuous space.
    
    \begin{definition}\cite[Definition 4.12, Chapter 4]{JFTGK}\label{cextght}
    	Let $(\mathcal{E},\mathpzc{d})$ be a complete separable metric space. Consider a  sequence of $\mathcal{E}-$valued stochastic processes $\{X_n(\cdot)\}_{n\geq1}$ which is exponentially tight in $\D([0,+\infty);\mathcal{E})$. We say that the sequence is \emph{$\C-$exponentially tight} if for every $\eta>0$ and $T>0$, the following condition holds:
    	\begin{align}\label{cexptght}
    		\limsup\limits_{n\to+\infty}
    		\frac{1}{n}\log\P\left(\sup\limits_{s\leq T} \mathpzc{d}(X_n(s),X_n(s-))\geq\eta\right)=
    		-\infty,
    	\end{align}
    	where $\mathpzc{d}(X_n(s),X_n(s-))$ measures the jump size of the process $X_n(\cdot)$ at $s$ and $X_n(s-)$ is the left limit of the process $X_n(\cdot)$ at $s$.
    \end{definition}
    %       The condition \eqref{cexptght} indicates that the probability of the process $\{X_n(\cdot)\}_{n\geq1}$ has large jumps of size at least $\eta$, in metric $d$, at time $s\leq T$, decays faster than any exponential rate as $n\to+\infty$.
    We then utilize the following result, together with the $\C-$exponential tightness, which gives the large deviation principle in the continuous space.
    \begin{theorem}[{\cite[Theorem 4.14, Section 4.4]{JFTGK}}]\label{refthm}
    	Let $\{\mathbb{P}_n\}_{n\in\N}$ be a sequence of probability measures on $\D([0,+\infty);\mathcal{E})$. Suppose that
    	\begin{itemize}
    		\item $\{\mathbb{P}_n\}_{n\in\N}$ is $\C-$exponentially tight, and
    		\item The large deviation principle holds for $\{\mathbb{P}_n\}_{n\in\N}$ (in the Skorohod topology) with rate function $I$.
    	\end{itemize}  
    	Then, the large deviation principle holds in the compact uniform topology with the same rate function $I$.
    \end{theorem}

   \section{Auxiliary results}\label{abscon} \setcounter{equation}{0}
   In this section, we discuss some supporting results for SCBF equations \eqref{stap}. We then derive the energy estimates and show that the exponential moment bounds hold.
   The following assumption on $\f$ is taken in this article:
   \begin{hypothesis}\label{fhyp}
   	The function $\f:[0,T]\to\V$ is continuous and there is $R>0$ such that 
   	\begin{align*}
   		\|\f(t)\|_{\V}\leq R, \  \text{ for all } \ t\in[0,+\infty).   
   	\end{align*}
   \end{hypothesis} 
   
   \subsection{Well posedness of the SCBF equations} \eqref{stap} 
  We now recall some existence, uniqueness, and continuous dependence results for both variational and strong solutions of the SCBF equations \eqref{stap}. The proofs of the following well-posedness results can be found in \cite[Propositions 4.3 and 4.4]{smtm1} and \cite[Theorems 3.10 and 4.2]{MTM8} with some minor modifications.
   
   \begin{proposition}\label{weLLp}
   	Assume that the Hypotheses \ref{trQ1}-\ref{fhyp} be satisfied. Let us fix $p\geq2$. For every $\y\in\H$, there exists a unique variational solution  $\Y_n(\cdot)=\Y_n(\cdot;t,\y)$ of \eqref{stap} having $\P-$a.s. continuous trajectories in $\H$ and satisfies the following energy estimates for all $s\in[t,T]$:
   	\begin{align}\label{eqn-conv-1}
   		&\E\left[\|\Y_n(s)\|_{\H}^p\right]+p\mu\E\bigg[\int_t^{s}\|\nabla\Y_n(\tau)\|_{\H}^2\|\Y_n(\tau)\|_{\H}^{p-2}\d\tau\bigg]+ p\beta\E\bigg[\int_t^{s}
   		\|\Y_n(\tau)\|_{\widetilde{\L}^{r+1}}^{r+1}
   		\|\Y_n(\tau)\|_{\H}^{p-2}\d\tau\bigg]\nonumber\\&\leq \|\y\|_{\H}^p+C(p,r,R,\Q,\alpha)(s-t),
   	\end{align}
   	and 
   	\begin{align}\label{vsee1}
   		&\E\bigg[\sup\limits_{s\in[t,T]}\|\Y_n(s)\|_{\H}^p\bigg]+\E\bigg[\int_t^T \|\nabla\Y_n(\tau)\|_{\H}^2 \|\Y_n(\tau)\|_{\H}^{p-2}\d\tau\bigg]+\E\bigg[\int_t^T \|\Y_n(\tau)\|_{\wi\L^{r+1}}^{r+1} \|\Y_n(\tau)\|_{\H}^{p-2}\d\tau\bigg]
   		\nonumber\\&\leq
   		\|\y\|_{\H}^p+C(p,r,R,\Q,\mu,\alpha,\beta,T). 		
   	\end{align}
   	%   	\begin{align}\label{eqn-conv-2}
   		%   		&\E\left[\|\nabla\Y_n(s)\|_{\H}^2\right]+\mu \E\bigg[\int_t^{s}\|\A\Y_n(\tau)\|_{\H}^2\d\tau\bigg]
   		%   		+\alpha\E\bigg[\int_t^{s}\|\nabla\Y_n(\tau)\|_{\H}^2\d\tau\bigg] 	\nonumber\\&\quad+
   		%   		\beta\E\bigg[\int_t^{s}
   		%   		\||\Y_n(\tau)|^{\frac{r-1}{2}}\nabla\Y_n(\tau)\|_{\H}^{2}\d\tau\bigg]
   		%   		\nonumber\\&\leq \big(\|\nabla\y\|_{\H}^2+\left(\frac{R^2}{2\alpha}+
   		%   		\frac{1}{n}\Tr(\Q_1)\right)(s-t)\big) e^{2\varrho (s-t)}
   		%   	\end{align}
   	%   	and 
   	%   	\begin{align}\label{ssee1}
   		%   		&\E\left[\sup\limits_{s\in[t,T]}\|\Y_n(s)\|_{\V}^2\right]+
   		%   		\E\bigg[\int_t^T \|\A\Y_n(\tau)\|_{\H}^2\d\tau\bigg]+\E\bigg[\int_t^T \||\Y_n(\tau)|^{\frac{r-1}{2}}\nabla\Y_n(\tau)\|_{\H}^2\d\tau\bigg]
   		%   		\nonumber\\&\leq
   		%   		\|\nabla\y\|_{\H}^2+C(R,\Q_1,\mu,\beta,n,T).  		
   		%   	\end{align}
   	Furthermore, if $\y\in\V$, then the {variational solution} $\Y_n(\cdot)=\Y_n(\cdot;t,\y)$ is a  \emph{strong solution} with continuous trajectories in $\V$. Moreover, for all $s\in[t,T]$, we have following energy estimates:
   	\begin{align}\label{eqn-conv-2}
   		&\E\left[\|\nabla\Y_n(s)\|_{\H}^p\right]+\frac{p\mu}{2} \E\bigg[\int_t^{s}\|\A\Y_n(\tau)\|_{\H}^2\|\nabla\Y_n(\tau)\|_{\H}^{p-2}\d\tau\bigg]\nonumber\\&\quad+\frac{3p\beta}{4} \E\bigg[\int_t^{s}
   		\||\Y_n(\tau)|^{\frac{r-1}{2}}\nabla\Y_n(\tau)\|_{\H}^{2} 
   		\|\nabla\Y_n(\tau)\|_{\H}^{p-2}\d\tau\bigg]\nonumber\\&\leq \big(\|\nabla\y\|_{\H}^p+C(p,\mathrm{Tr}(\Q_1),R,\alpha)(s-t)\big) e^{p\varrho (s-t)},
   	\end{align}
   	and 
   	\begin{align}\label{ssee1}
   		&\E\left[\sup\limits_{s\in[t,T]}\|\nabla\Y_n(s)\|_{\H}^p\right]+
   		\E\bigg[\int_t^T \|\A\Y_n(\tau)\|_{\H}^2 \|\nabla\Y_n(\tau)\|_{\H}^{p-2}\d\tau\bigg]\nonumber\\&\quad+\E\bigg[\int_t^T \||\Y_n(\tau)|^{\frac{r-1}{2}}\nabla\Y_n(\tau)\|_{\H}^2 \|\nabla\Y_n(\tau)\|_{\H}^{p-2}\d\tau\bigg]
   		\nonumber\\&\leq
   		\|\nabla\y\|_{\H}^p+C(p,r,R,\Q_1,\mu,\alpha,\beta,T).  		
   	\end{align}
   \end{proposition}

   \begin{proposition}[Continuous dependence of solutions]\label{cts-dep-soln}
   	
   	(i) There exists a constant $C$, independent of $t$, such that for all $s\in[t,T]$, $\P-$a.s., we have
   	\begin{align}\label{ctsdep0}
   		&\|\Y_{1,n}(s)-\Y_{2,n}(s)\|_{\H}^2+\int_t^s \|\nabla(\Y_{1,n}-\Y_{2,n})(\tau)\|_{\H}^2\d\tau+\int_t^s \|\Y_{1,n}(\tau)-\Y_{2,n}(\tau)\|_{\wi\L^{r+1}}^{r+1}\d\tau\nonumber\\&\leq
   		\|\y_1-\y_2\|_{\H}^2 e^{C(s-t)},
   	\end{align} 
   	where $\Y_{1,n}(\cdot)=\Y_{1,n}(\cdot;t,\y_1)$ and $\Y_{2,n}(\cdot)=\Y_{2,n}(\cdot;t,\y_2)$ are two strong  solutions of \eqref{stap} with initial conditions $\Y_{1,n}(t)=\y_1$ in $\H$ and $\Y_{2,n}(t)=\y_2$ in $\H$. 
   	
   	(ii) If $\|\y\|_{\V}\leq R_1$, where $R_1$ is arbitrary, then there exists a constant $$C=C(\mu,\alpha,\beta,T,R,R_1,\mathrm{Tr}(\Q),\mathrm{Tr}(\Q_1))$$ such that for all $s\in[t,T]$, 
   	\begin{align}\label{ctsdep0.1}
   		\E\left[\|\Y_n(s)-\y\|_{\H}^2\right]\leq C(\mu,\alpha,\beta,T,R,R_1, \mathrm{Tr}(\Q),\mathrm{Tr}(\Q_1))(s-t), 
   	\end{align} 
   	where $\Y_n(\cdot)=\Y_n(\cdot;t,\y)$.
   	
   	(iii) For every initial condition $\y\in\V,$ there exists a modulus $\omega$ such that
   	\begin{align}\label{ctsdep0.2}
   		\E\left[\|\Y_n(s)-\y\|_{\V}^p\right]\leq\omega_{\y}(s-t), \ \text{ for all } \ s\in[t,T],
   	\end{align}
   	where $\Y_n(\cdot)=\Y_n(\cdot;t,\y)$.
   \end{proposition}
   
   \begin{proof}
   For  reader's convenience, we provide the proof of part $(ii)$ only. For the proof of other parts one can refer \cite[Proposition 4.4]{smtm1}. Let us take $\Z(\cdot):=\Y(\cdot)-\y$. Then, we rewrite from \eqref{stap}, for $s\in[t,T]$, $\P-$a.s., 
   \begin{align*}
   	\Z(s)=&\int_t^s\big[-\mu\A\Y(\tau)-\alpha\Y(\tau)-\mathcal{B}(\Y(\tau))-\beta\mathcal{C}(\Y(\tau))+\f(\tau)\big]\d\tau
   	\nonumber\\&+
   	\frac{1}{\sqrt{n}}\int_t^s\Q^{\frac12}\d\mathbf{W}(\tau).
   \end{align*}
   On applying the infinite-dimensional It\^o formula to the function $\|\cdot\|_{\H}^2$ and to the process $\Z(\cdot)$ and then taking expectation, we get
   \begin{align}\label{ctsdep4}
   	\E\left[\|\Z(s)\|_{\H}^2\right]&=2\E\bigg[\int_t^s \big(-\mu\A\Y(\tau)-\alpha\Y(\tau)- \mathcal{B}(\Y(\tau))-\beta\mathcal{C}(\Y(\tau))+\f(\tau),\Z(\tau)\big)
   	\d\tau\bigg]
   	\nonumber\\&\quad+
   	\mathrm{Tr}(\Q)(s-t).
   \end{align}
   By using the Cauchy Schwarz inequality  and Young's inequality, \eqref{ctsdep4} reduces to
   \begin{align}\label{ctsdep5}
   	\E\left[\|\Z(s)\|_{\H}^2\right]\leq&-\mu\E\bigg[\int_t^s \|\nabla\Y(\tau)\|_{\H}^2\d\tau\bigg]
   	-\alpha\E\bigg[\int_t^s \|\Y(\tau)\|_{\H}^2\d\tau\bigg]
   	-2\beta\E\bigg[\int_t^s
   	\|\Y(\tau)\|_{\wi\L^{r+1}}^{r+1}\d\tau\bigg]
   	\nonumber\\&+2\mu\E\bigg[\int_t^s
   	\big(\A\Y(\tau),\y\big)\d\tau\bigg]+
   	2\alpha\E\bigg[\int_t^s
   	\big(\Y(\tau),\y\big)\d\tau\bigg]
   	\nonumber\\&+
   	2\E\bigg[\int_t^s
   	\big(\f(\tau),\Z(\tau)\big)\d\tau\bigg]+
   	2\E\bigg[\int_t^s
   	\big(\mathcal{B}(\Y(\tau)),\y\big)\d\tau\bigg]
   	\nonumber\\&+
   	2\beta\E\bigg[\int_t^s
   	\big(\mathcal{C}(\Y(\tau)),\y\big)\d\tau\bigg]+
   	\mathrm{Tr}(\Q)(s-t).
   \end{align}
   By using the  Cauchy Schwarz inequality and Sobolev's embedding, we estimate
   	\begin{align}
   		\bigg|\E\left[\int_{t}^{s}(\mathcal{B}(\Y(s)),\y) 
   		\d s\right]\bigg|&\leq\mathpzc{C}
   		\E\left[\int_{t}^{s}
   		\|\Y(s)\|_{\V}^2\|\nabla\y\|_{\H}\d s\right]
   		\nonumber\\&\leq
   		\mathpzc{C}\|\y\|_{\V}\left(\E\left[\sup\limits_{t\leq s \leq T} \|\Y(s)\|_{\V}^2\right]\right)(s-t).\label{ctsdep6}\\
   		\bigg|\E\left[\int_{t}^{s}(\mathcal{C}(\Y(s)),\y) 
   		\d s\right]\bigg|
   		&\leq\E\left[\int_{t}^{s}
   		\|\Y\|_{\wi\L^{r+1}}^{r-1}\|\y\|_{\wi\L^{r+1}}\d s\right]
   		\nonumber\\&\leq
   		\mathpzc{C}\|\y\|_{\V}\left(\E\left[\sup\limits_{t\leq s \leq T} \|\Y(s)\|_{\V}^{r-1}\right]\right)(s-t).\label{ctsdep7}
   \end{align}	
   Plugging \eqref{ctsdep6}-\eqref{ctsdep7} into \eqref{ctsdep5} yields that
   \begin{align}\label{ctsdep55}
   	\E\left[\|\Z(s)\|_{\H}^2\right]\leq&-\frac{\mu}{2}\E\bigg[\int_t^s \|\nabla\Y(\tau)\|_{\H}^2\d\tau\bigg]
   	-\frac{\alpha}{2}\E\bigg[\int_t^s \|\Y(\tau)\|_{\H}^2\d\tau\bigg]-
   	2\beta\E\bigg[\int_t^s
   	\|\Y(\tau)\|_{\wi\L^{r+1}}^{r+1}\d\tau\bigg]
   	\nonumber\\&+
   	2(\mu+\alpha)(s-t)\|\y\|_{\V}^2
   	+\mathrm{Tr}(\Q)(s-t)+
   	\E\int_t^s \|\f(\tau)\|_{\H}^2\d\tau
   	\nonumber\\&
   	+\mathpzc{C}\|\y\|_{\V}\left(\E\left[\sup\limits_{t\leq s \leq T} \|\Y(s)\|_{\V}^2\right]\right)(s-t)
   	+\mathpzc{C}\|\y\|_{\V}\left(\E\left[\sup\limits_{t\leq s \leq T} \|\Y(s)\|_{\V}^{r-1}\right]\right)(s-t)
   	\nonumber\\&
   	+\E\int_t^s\|\Z(\tau)\|_{\H}^2\d\tau,
   \end{align}
   for all $s\in[t,T]$. 
   	Using Hypothesis \ref{fhyp}, Fubini's theorem, \eqref{vsee1} and \eqref{ssee1}, we further simplify \eqref{ctsdep55} as follows:
   	\begin{align*}
   		\E\left[\|\Z(s)\|_{\H}^2\right]\leq&
   		\big(2(\mu+\alpha)\|\y\|_{\V}^2+R^2+
   		\mathrm{Tr}(\Q)\big)(s-t)+
   		\mathpzc{C}\|\y\|_{\V}\left(\E\left[\sup\limits_{t\leq s \leq T} \|\Y(s)\|_{\V}^2\right]\right)(s-t)
   		\nonumber\\&
   		+\mathpzc{C}\|\y\|_{\V}\left(\E\left[\sup\limits_{t\leq s \leq T} \|\Y(s)\|_{\V}^{r-1}\right]\right)(s-t)
   		+\int_t^s\E\|\Z(\tau)\|_{\H}^2\d\tau.
   	\end{align*}  
   	By applying Gr\"onwall's inequality and making use of \eqref{eqn-conv-2}, we deduce that
   	\begin{align}\label{ctsdep8}
   		\E\left[\|\Z(s)\|_{\H}^2\right]\leq\mathpzc{C}(s-t),
   \end{align}
   where $\mathpzc{C}=\mathpzc{C}(\mu,\alpha,\beta,R,\|\y\|_{\V},\Tr(\Q),\Tr(\Q_1))$.
   This completes the proof of \eqref{ctsdep0.1}.
   \end{proof}
   
   \subsection{Exponential moment estimates}
   A key step in establishing the LDP is to derive exponential moment bounds for the solutions of the SCBF equations \eqref{stap}. For reader's convenience, we present a detailed proof of this result below. 
   \begin{proposition}\label{propexpest}
   	For 
   	\begin{align}\label{eqn-con}
   \mbox{$\mu\geq \frac{1}{\alpha^{\frac{r-3}{r-1}}}\left(\frac{(r-3)}{2(r-1)}\right)^{\frac{r-3}{r-1}}\left(\frac{1}{\beta(r-1)}\right)^{\frac{2}{r-1}}$\ \text{ or }\ $\mu\geq \frac{1}{4}\max\left\{\frac{1}{\alpha},\frac{1}{\beta}\right\}$,}
   \end{align}
  there exist constants $\mathpzc{c}_i>0$, $i=1,2,3$, such that if $\y\in\V,$ then
   	\begin{align}\label{expmoest}
   		\E\left[\sup_{s\in[t,T]}e^{n\mathpzc{c}_1\|\Y_n(s)\|_{\V}^2}
   		\right]\leq\mathpzc{c}_2e^{n\mathpzc{c}_3},
   	\end{align}
   	where 
   	\begin{align*}
   		0<\mathpzc{c}_1&<\mathpzc{c}_1^0:=\frac{\alpha-\varrho^*}{\Tr(\Q_1)} \ \text{ or } \\
   		0<\mathpzc{c}_1&<\mathpzc{c}_2^0:=\frac{\alpha-\frac{1}{4\mu}}
   		{\Tr(\Q_1)}, \\
   		\mathpzc{c}_2&=\mathpzc{c}_2\big(\|\y\|_{\V},\mathpzc{c}_1,
   		\mathpzc{c}_1^0,\mathrm{Tr}(\Q_1),R\big) \ \text{ and }\\
   		\mathpzc{c}_3&=2\mathpzc{c}_1\left(\frac{R^2}{2\alpha}+
   		\Tr(\Q_1)\right)T.
   	\end{align*}
   \end{proposition}
 
   \begin{proof}
   	Let us set $Z(\cdot):=\mathpzc{a}n\|\Y_n(\cdot)\|_{\V}^2$, for some $\mathpzc{a}>0$ to be chosen later.
   	Let us now consider the function $\mathpzc{G}(\mathpzc{w})=
   	e^{\mathpzc{a}n\|\A^{\frac12}\mathpzc{w}\|_{\H}^2},
   	\ \mathpzc{w}\in\mathcal{D}(\A^{\frac12})$. Then, we have
   	\begin{align*}
   		\D_{\mathpzc{w}}\mathpzc{G}(\mathpzc{w})
   		\mathpzc{z}&=
   		2\mathpzc{a}n
   		e^{\mathpzc{a}n\|\A^{\frac12}\mathpzc{w}\|_{\H}^2}
   		(\A^{\frac12}\mathpzc{w},\A^{\frac12}\mathpzc{z}) 
   		\ \text{ for } \ \mathpzc{z}\in\mathcal{D}(\A^{\frac12})
   	\end{align*}
   	and
   	\begin{align*}
   		 \D^2_{\mathpzc{w}}\mathpzc{G}(\mathpzc{w})
   		(\mathpzc{z}_1,\mathpzc{z}_2)&=
   		4(\mathpzc{a}n)^2 e^{\mathpzc{a}n\|\A^{\frac12}\mathpzc{w}\|_{\H}^2}
   		\underbrace{(\A^{\frac12}\mathpzc{w},
   			\A^{\frac12}\mathpzc{z}_1)
   			(\A^{\frac12}\mathpzc{w},\A^{\frac12} \mathpzc{z}_2)}_{= (\A^{\frac12}\mathpzc{w}\otimes\A^{\frac12}\mathpzc{w})
   			(\A^{\frac12}\mathpzc{z}_1,\A^{\frac12}\mathpzc{z}_2)}
   		\nonumber\\&\quad+2\mathpzc{a}n
   		e^{\mathpzc{a}n\|\A^{\frac12}\mathpzc{w}\|_{\H}^2}
   		(\A^{\frac12}\mathpzc{z}_1,\A^{\frac12}\mathpzc{z}_2)
   		\ \text{ for } \ \mathpzc{z}_1,\mathpzc{z}_2 \in\mathcal{D}(\A^{\frac12}).
   	\end{align*}
   	Moreover, we find
   	\begin{align*}
   		\Tr(\Q\D^2_{\mathpzc{w}}\mathpzc{G}(\mathpzc{w})) =4(\mathpzc{a}n)^2  \Tr(\Q_1)
   		e^{\mathpzc{a}n\|\A^{\frac12}\mathpzc{w}\|_{\H}^2} \|\A^{\frac12}\mathpzc{w}\|_{\H}^2+
   		2\mathpzc{a}n\Tr(\Q_1).
   	\end{align*}
   	Let us now operate by $\A^{\frac12}$ in \eqref{stap} to obtain the following stochastic differential satisfied the stochastic process $\A^{\frac12}\Y(\cdot)$:
   	\begin{align*}
   		&\d\A^{\frac12}\Y_n(s)+\A^{\frac12}[\mu\A\Y_n(s)+\mathcal{B}(\Y_n(s))
   		+\alpha\Y_n(s)+\beta\mathcal{C}(\Y_n(s))]\d s \nonumber\\&=
   		\A^{\frac12}\f(s)\d s+ \frac{1}{\sqrt{n}}\A^{\frac12}\Q^{\frac12}\d\mathbf{W}(s),
   	\end{align*}
   	for a.e. $s\in[t,T]$. 
   	We define a sequence of  stopping times
   	\begin{align*}
   		\wi\theta_N=\inf\{s\geq t:\|\nabla\Y(s)\|_{\H}>N\}, \ \text{ for any } \ N>0.
   	\end{align*}
   	On applying It\^o's formula (see \cite[Theorem A.1, pp. 294]{ZBSP}, \cite[Theorem 4.3, pp. 1809]{ZBSP1}) to the function $e^{\mathpzc{a}n\|\cdot\|_{\H}^2}$ and to the process $\A^{\frac12}\Y(\cdot)$, we find for all $s\in[t,T]$, $\P-\text{a.s.},$
   	\begin{align}\label{vse7.1}
   		&e^{\mathpzc{a}n\|\nabla\Y_n(s\land\wi\theta_N)\|_{\H}^2}
   		+2\mathpzc{a}n\int_t^{s\land\wi\theta_N} e^{\mathpzc{a}n\|\nabla\Y_n(\tau)\|_{\H}^2}
   		\big(\mu\|\A\Y_n(\tau)\|_{\H}^2+
   		\alpha\|\nabla\Y_n(\tau)\|_{\H}^2\big)\d\tau
   		\nonumber\\&\quad+
   		2\mathpzc{a}n\beta\int_t^{s\land\wi\theta_N}
   		e^{\mathpzc{a}n\|\nabla\Y_n(\tau)\|_{\H}^2}
   		\big(\mathcal{C}(\Y_n(\tau)),\A\Y_n(\tau)\big)\d\tau
   		\nonumber\\&=
   		e^{\mathpzc{a}n\|\nabla\y\|_{\H}^2}+
   		2\mathpzc{a}n\int_t^{s\land\wi\theta_N}
   		e^{\mathpzc{a}n\|\nabla\Y_n(\tau)\|_{\H}^2} \big(\A^{\frac12}\f(\tau),\A^{\frac12}\Y_n(\tau)\big)\d\tau
   		\nonumber\\&\quad-
   		2\mathpzc{a}n\int_t^{s\land\wi\theta_N}
   		e^{\mathpzc{a}n\|\nabla\Y_n(\tau)\|_{\H}^2}
   		\big(\mathcal{B}(\Y_n(\tau)),\A\Y_n(\tau)\big)\d\tau
   		\nonumber\\&\quad
   		+2\mathpzc{a}\sqrt{n}M_{s\land\wi\theta_N}
   		+\mathpzc{a}\Tr(\Q_1)
   		\int_t^{s\land\wi\theta_N}
   		e^{\mathpzc{a}n\|\nabla\Y_n(\tau)\|_{\H}^2}\d\tau
   		\nonumber\\&\quad+
   		2\mathpzc{a}^2n\Tr(\Q_1)
   		\int_t^{s\land\wi\theta_N}
   		e^{\mathpzc{a}n\|\nabla\Y_n(\tau)\|_{\H}^2}
   		\|\nabla\Y_n(\tau)\|_{\H}^2\d\tau, 
   	\end{align}
   	where $${M}_{s\land\wi\theta_N}=
   	\int_t^{s\land\wi\theta_N} 
   	e^{\mathpzc{a}\|\nabla\Y_n(\tau)\|_{\H}^2}
   	\big(\A^{\frac12}\Y(\tau), \A^{\frac12}\Q^{\frac12}\d\W(\tau)\big)$$ is a martingale.
   	Using Hypothesis \ref{fhyp}, Young's inequality, equality \eqref{torusequ} and the estimate \eqref{syymB3}, we calculate following:
   	\begin{align}
   		\big(\A^{\frac12}\f,\A^{\frac12}\Y_n\big) &\leq\frac{R^2}{2(\alpha-\varrho^*-
   			\mathpzc{a}\Tr(\Q_1))}+\frac{\alpha-\varrho^*-
   			\mathpzc{a}\Tr(\Q_1)}{2} \|\nabla\Y_n\|_{\H}^2,\label{vse8}\\
   		\big(\mathcal{C}(\Y_n),\A\Y_n\big)&\geq
   		\||\Y_n|^{\frac{r-1}{2}}\nabla\Y_n\|_{\H}^{2},
   		\label{vse8.2}\\
   		(\mathcal{B}(\Y_n),\A\Y_n)|&\leq
   		\mu\|\A\Y_n\|_{\H}^2
   		+\frac{\beta}{2}\||\Y_n|^{\frac{r-1}{2}}\nabla\Y_n\|_{\H}^2 +\varrho^*\|\nabla\Y_n\|_{\H}^2,\label{vse8.3}
   	\end{align}
   	where $
   		\varrho^*:=\frac{r-3}{2\mu(r-1)}\left[\frac{1}{\beta\mu (r-1)}\right]^{\frac{2}{r-3}}.$ Utilizing the inequalities \eqref{vse8}-\eqref{vse8.3} in \eqref{vse7.1}, we obtain for all $s\in[t,T]$, $\P-\text{a.s.},$
   	\begin{align}\label{vse11}
   		&e^{\mathpzc{a}n\|\nabla\Y_n(s\land\wi\theta_N)\|_{\H}^2}+
   		2\mathpzc{a}n\alpha\int_t^{s\land\wi\theta_N} e^{\mathpzc{a}n\|\nabla\Y_n(\tau)\|_{\H}^2}
   		\|\nabla\Y_n(\tau)\|_{\H}^2\d\tau 
   		\nonumber\\&\quad+
   		\mathpzc{a}n\beta\int_t^{s\land\wi\theta_N}
   		e^{\mathpzc{a}n\|\nabla\Y_n(\tau)\|_{\H}^2}
   		\||\Y_n(\tau)|^{\frac{r-1}{2}}\nabla\Y_n(\tau)\|_{\H}^{2}\d\tau \nonumber\\&\leq 
   		e^{\mathpzc{a}n\|\nabla\y\|_{\H}^2}+
   		\frac{\mathpzc{a}nR^2}{\alpha-\varrho^*-
   			\mathpzc{a}\Tr(\Q_1)}\int_t^{s\land\wi\theta_N}
   		e^{\mathpzc{a}n\|\nabla\Y_n(\tau)\|_{\H}^2}\d\tau+
   			2\mathpzc{a}n\varrho^*\int_t^{s\land\wi\theta_N}
   			e^{\mathpzc{a}n\|\nabla\Y_n(\tau)\|_{\H}^2}
   			\|\nabla\Y_n(\tau)\|_{\H}^2\d\tau
   		\nonumber\\&\quad
   		+2\mathpzc{a}\sqrt{n}M_{s\land\wi\theta_N}
   		+\mathpzc{a}\Tr(\Q_1)
   		\int_t^{s\land\wi\theta_N}
   		e^{\mathpzc{a}n\|\nabla\Y_n(\tau)\|_{\H}^2}\d\tau+
   		2\mathpzc{a}^2n\Tr(\Q_1)
   		\int_t^{s\land\wi\theta_N}
   		e^{\mathpzc{a}n\|\nabla\Y_n(\tau)\|_{\H}^2}
   		\|\nabla\Y_n(\tau)\|_{\H}^2\d\tau
   		\nonumber\\&\quad+
   		\mathpzc{a}n(\alpha-\varrho^*-
   		\mathpzc{a}\Tr(\Q_1))\int_t^{s\land\wi\theta_N} e^{\mathpzc{a}n\|\nabla\Y_n(\tau)\|_{\H}^2}
   		\|\nabla\Y_n(\tau)\|_{\H}^2\d\tau 
   		\nonumber\\&\leq
   		e^{\mathpzc{a}n\|\nabla\y\|_{\H}^2}+\mathpzc{a}n\underbrace{
   			\left(\frac{R^2}{\alpha-\varrho^*-
   				\mathpzc{a}\Tr(\Q_1)}+
   			\Tr(\Q_1)\right)}_{:=\mathpzc{k}}
   		\int_t^{s\land\wi\theta_N}
   		e^{\mathpzc{a}n\|\nabla\Y_n(\tau)\|_{\H}^2}\d\tau+
   		2\mathpzc{a}\sqrt{n}M_{s\land\wi\theta_N}
   		\nonumber\\&\quad+
   		\left(2\mathpzc{a}n\varrho^*+2\mathpzc{a}^2n\Tr(\Q_1)
   		\right)\int_t^{s\land\wi\theta_N}
   		e^{\mathpzc{a}n\|\nabla\Y_n(\tau)\|_{\H}^2}
   		\|\nabla\Y_n(\tau)\|_{\H}^2\d\tau
   			\nonumber\\&\quad+
   		\mathpzc{a}n(\alpha-\varrho^*-
   		\mathpzc{a}\Tr(\Q_1))\int_t^{s\land\wi\theta_N} e^{\mathpzc{a}n\|\nabla\Y_n(\tau)\|_{\H}^2}
   		\|\nabla\Y_n(\tau)\|_{\H}^2\d\tau 
   	\end{align}
%   	In view of the embedding $\mathcal{D}(\A)\hookrightarrow\mathcal{D}(\A^{\frac12})$, there is constant $\mathpzc{C}_e>0$ such that 
%   	\begin{align}\label{da12e}
%   		\|\nabla\Y_n\|_{\H}\leq\mathpzc{C}_e\|\A\Y_n\|_{\H}.
%   	\end{align}
   	On rearranging the terms in \eqref{vse11}, we obtain $\P-$a.s.,
   	\begin{align}\label{vse11dade}
   		&e^{\mathpzc{a}n\|\nabla\Y_n(s\land\wi\theta_N)\|_{\H}^2}+
   			\mathpzc{a}n\left(\alpha-\varrho^*-
   			\mathpzc{a}\Tr(\Q_1)\right)
   		\int_t^{s\land\wi\theta_N}
   		e^{\mathpzc{a}n\|\nabla\Y_n(\tau)\|_{\H}^2}
   		\|\nabla\Y_n(\tau)\|_{\H}^2\d\tau
   		%      		\nonumber\\&+
   		%      		{\color{Maroon}\mathpzc{a}n(\alpha-2\varrho)}\int_t^{s\land\wi\theta_N} e^{\mathpzc{a}n\|\nabla\Y_n(\tau)\|_{\H}^2}
   		%      		\|\nabla\Y_n(\tau)\|_{\H}^2\d\tau 
   		\nonumber\\&\quad
   		+
   		\mathpzc{a}n\beta\int_t^{s\land\wi\theta_N}
   		e^{\mathpzc{a}n\|\nabla\Y_n(\tau)\|_{\H}^2}
   		\||\Y_n(\tau)|^{\frac{r-1}{2}}\nabla\Y_n(\tau)\|_{\H}^{2}\d\tau \nonumber\\&\leq 
   		e^{\mathpzc{a}n\|\nabla\y\|_{\H}^2}+
   		\mathpzc{a}n\mathpzc{k}\int_t^{s\land\wi\theta_N}
   		e^{\mathpzc{a}n\|\nabla\Y_n(\tau)\|_{\H}^2}\d\tau+
   		2\mathpzc{a}\sqrt{n}M_{s\land\wi\theta_N}
   	\end{align}
   	Now choose $\mathpzc{a}>0$ in such a way that 
   		\begin{align*}
   			\alpha-\varrho^*-\mathpzc{a}\Tr(\Q_1)>0 \ \text{ or }
   			\  \mathpzc{a}<\frac{\alpha-\varrho^*}{\Tr(\Q_1)}.
   	\end{align*}
%   	With this choice of $\mathpzc{a}$, we obtain $\P-$a.s. the following inequality:
%   	\begin{align}\label{vse11dade1}
%   		&e^{\mathpzc{a}n\|\nabla\Y_n(s\land\wi\theta_N)\|_{\H}^2}+{\color{Blue}
%   		\mathpzc{a}n\left(\mu\mathpzc{C}_e^2+\alpha-2\varrho-
%   		\mathpzc{a}\Tr(\Q_1)\right)}\int_t^{s\land\wi\theta_N} e^{\mathpzc{a}n\|\nabla\Y_n(\tau)\|_{\H}^2}
%   		\|\nabla\Y_n(\tau)\|_{\H}^2\d\tau 
%   		\nonumber\\&\quad+
%   		\frac{3\mathpzc{a}n\beta}{2}\int_t^{s\land\wi\theta_N}
%   		e^{\mathpzc{a}n\|\nabla\Y_n(\tau)\|_{\H}^2}
%   		\||\Y_n(\tau)|^{\frac{r-1}{2}}\nabla\Y_n(\tau)\|_{\H}^{2}\d\tau \nonumber\\&\leq 
%   		e^{\mathpzc{a}n\|\nabla\y\|_{\H}^2}+
%   		\mathpzc{a}n\mathpzc{k}\int_t^{s\land\wi\theta_N}
%   		e^{\mathpzc{a}n\|\nabla\Y_n(\tau)\|_{\H}^2}\d\tau+
%   		2\mathpzc{a}\sqrt{n}M_{s\land\wi\theta_N},
%   	\end{align}
%   	for all $s\in[t,T]$. 
   	Then, on taking the expectation in \eqref{vse11dade}, we deduce for all $s\in[t,T]$
   	\begin{align}\label{vse12}
   		&\E\left[e^{\mathpzc{a}n\|\nabla\Y_n(s\land\wi\theta_N)\|_{\H}^2}
   		\right]+\mathpzc{a}n\left(\alpha-\varrho^*-
   		\mathpzc{a}\Tr(\Q_1)\right)
   		\E\bigg[\int_t^{s\land\wi\theta_N} e^{\mathpzc{a}n\|\nabla\Y_n(\tau)\|_{\H}^2}
   		\|\nabla\Y_n(\tau)\|_{\H}^2\d\tau\bigg]
   		\nonumber\\&\quad+
   		\mathpzc{a}n\beta\E\bigg[\int_t^{s\land\wi\theta_N}
   		e^{\mathpzc{a}n\|\nabla\Y_n(\tau)\|_{\H}^2}
   		\||\Y_n(\tau)|^{\frac{r-1}{2}}\nabla\Y_n(\tau)\|_{\H}^{2}
   		\d\tau\bigg] 
   		\nonumber\\&\leq\|\nabla\y\|_{\H}^2+
   		\mathpzc{a}n\mathpzc{k}\E\bigg[\int_t^{s\land\wi\theta_N}
   		e^{\mathpzc{a}n\|\nabla\Y_n(\tau)\|_{\H}^2}\d\tau\bigg],
   	\end{align}
   	for all $s\in[t,T]$. By an application of Fubini's theorem, we write from \eqref{vse12} that
   	\begin{align}\label{vse13}
   		\E\left[e^{\mathpzc{a}n\|\nabla\Y_n(s\land\wi\theta_N)\|_{\H}^2}
   		\right]\leq
   		\|\nabla\y\|_{\H}^2+
   		\mathpzc{a}n\mathpzc{k}\int_t^{s}\E\left[e^{\mathpzc{a}n \|\nabla\Y_n(\tau\land\wi\theta_N)\|_{\H}^2}\right]\d\tau,
   	\end{align}
   	for all $s\in[t,T]$. On employing Gr\"onwall's inequality, we conclude from \eqref{vse13} that for all $s\in[t,T]$
   	\begin{align}\label{vse14}
   		\E\left[e^{\mathpzc{a}n\|\nabla\Y_n(s\land\wi\theta_N)\|_{\H}^2}
   		\right]\leq\mathpzc{C},
   	\end{align}
   	where $\mathpzc{C}=\mathpzc{C} (\mathpzc{k},T,\|\nabla\xi\|_{\H}^2)$. 
%   	Note that for the indicator function $\mathds{1}$, we write 
%   	$$\E\left[\mathds{1}_{\{\wi\theta_N<s\}}\right]=\mathbb{P}\Big\{\omega\in\Omega:\wi\theta_N(\omega)<s\Big\}.$$
%   	{\color{Blue}
%   		By an application of Markov's inequality, we obtain 
%   		\begin{align}\label{vse14.1}
%   			\E\left[e^{\mathpzc{a}n\|\nabla\Y_n(s\land\wi\theta_N)\|_{\H}^2}
%   			\right]&=
%   			\E\left[e^{\mathpzc{a}n\|\nabla\Y_n(s\land\wi\theta_N)\|_{\H}^2}\mathds{1}_{\{\wi\theta_N<s\}}\right]
%   			+\E\left[e^{\mathpzc{a}n\|\nabla\Y_n(s\land\wi\theta_N)\|_{\H}^2}
%   			\mathds{1}_{\{\wi\theta_N\geq s\}}\right]
%   			\nonumber\\&\geq
%   			\E\left[e^{\mathpzc{a}n\|\nabla\Y_n(s\land\wi\theta_N)\|_{\H}^2}\mathds{1}_{\{\wi\theta_N<s\}}\right]\geq e^{\mathpzc{a}nN^2}
%   			\mathbb{P}\Big\{\omega\in\Omega:\wi\theta_N<s\Big\}.
%   		\end{align}
%   		Utilizing the energy estimate \eqref{vse14} in \eqref{vse14.1}, we find 
%   		\begin{align*}
%   			\mathbb{P}\Big\{\omega\in\Omega:\wi\theta_N<s\Big\}\leq
%   			\frac{1}{e^{\mathpzc{a}nN^2}}
%   			\E\left[e^{\mathpzc{a}n\|\nabla\Y_n(s\land\wi\theta_N)\|_{\H}^2}
%   			\right]\leq
%   			\frac{\mathpzc{C}}{e^{\mathpzc{a}nN^2}}.
%   		\end{align*}
%   		Hence, we have
%   		\begin{align*}
%   			\lim_{n\to+\infty}\mathbb{P}
%   			\Big\{\omega\in\Omega:\wi\theta_N<s\Big\}=0, \ \text{ for all }\ s\in [t,T],
%   		\end{align*}
%   		and $s\land\wi\theta_N\to s$ as $n\to+\infty$, $\P$-a.s.}
   	Using the energy estimate \eqref{eqn-conv-2}, taking limit $N\to+\infty$ in \eqref{vse12} (\cite[Proposition 3.5]{MTM8}) and using the \emph{monotone convergence theorem} together with \eqref{vse14}, we finally obtain  for all $s\in[t,T]$
   		\begin{align}\label{vsedade1}
   			&\E\left[e^{\mathpzc{a}n\|\nabla\Y_n(s)\|_{\H}^2}
   			\right]+\mathpzc{a}n\left(\alpha-\varrho^*-
   				\mathpzc{a}\Tr(\Q_1)\right)
   				\E\bigg[\int_t^{s} e^{\mathpzc{a}n\|\nabla\Y_n(\tau)\|_{\H}^2}
   			\|\nabla\Y_n(\tau)\|_{\H}^2\d\tau\bigg]
   			\nonumber\\&\quad+
   			\mathpzc{a}n\beta\E\bigg[\int_t^{s}
   			e^{\mathpzc{a}n\|\nabla\Y_n(\tau)\|_{\H}^2}
   			\||\Y_n(\tau)|^{\frac{r-1}{2}}\nabla\Y_n(\tau)\|_{\H}^{2}
   			\d\tau\bigg]\leq\mathpzc{C}.
   	\end{align}
   	
   	For $\alpha>\varrho^*$, we now take supremum over $t$ to $T\land\wi\theta_N$ in \eqref{vse11dade} followed by expectation, we find 
   	\begin{align}\label{vse15}
   		\E\bigg[\sup\limits_{s\in[t,T\land\wi\theta_N]}
   		e^{\mathpzc{a}n\|\nabla\Y_n(s)\|_{\H}^2}\bigg]
   		%      		+\mathpzc{a}n(\alpha-2\varrho)\E\bigg[\int_t^{T\land\theta_N}
   		%      		e^{\mathpzc{a}n\|\nabla\Y_n(\tau)\|_{\H}^2}
   		%      		\|\nabla\Y_n(\tau)\|_{\H}^2\d\tau\bigg]
   		%      		\nonumber\\&\quad+ 	\frac{3\mathpzc{a}n\beta}{2}\E\bigg[\int_t^{T\land\theta_N}
   		%      		e^{\mathpzc{a}n\|\nabla\Y_n(\tau)\|_{\H}^2}
   		%      		\||\Y_n(\tau)|^{\frac{r-1}{2}}\nabla\Y_n(\tau)\|_{\H}^{2}
   		%      		\d\tau\bigg]\nonumber\\
   		&\leq e^{\mathpzc{a}n\|\nabla\y\|_{\H}^2}+
   			\mathpzc{a}n\mathpzc{k}
   			\E\bigg[\int_t^{T\land\wi\theta_N}
   			e^{\mathpzc{a}n\|\nabla\Y_n(\tau)\|_{\H}^2}\d\tau\bigg]
   		\nonumber\\&\quad+ 
   		2\mathpzc{a}\sqrt{n}\E\bigg[\sup\limits_{s\in[t,T\land\wi\theta_N]} \bigg|\int_t^{s}
   		e^{\mathpzc{a}n\|\nabla\Y_n(\tau)\|_{\H}^2}
   		\big(\A^{\frac12}\Y_n(\tau), \A^{\frac12}\Q^{\frac12}\d\W(\tau)\big)\bigg|\bigg].
   	\end{align} 
   	By an application of the Burkholder-Davis-Gundy inequality (see \cite[Theorem 1.1]{CMMR}),  H\"older's and Young's inequalities, we calculate
   	\begin{align}\label{vse15.1}
   		&2\mathpzc{a}\sqrt{n}
   		\E\bigg[\sup\limits_{s\in[t,T\land\wi\theta_N]}\bigg|\int_t^s
   		e^{\mathpzc{a}n\|\nabla\Y_n(\tau)\|_{\H}^2}
   		\big(\A^{\frac12}\Y_n(\tau),\A^{\frac12}\Q^{\frac12} \d\W(\tau)\big)\bigg|\bigg]
   		\nonumber\\&\leq
   		2\mathpzc{a}\sqrt{n}(\mathrm{Tr}(\Q_1))^{\frac12}
   		\E\bigg[\int_t^{T\land\wi\theta_N}
   		e^{2\mathpzc{a}n\|\nabla\Y_n(\tau)\|_{\H}^2}
   		\|\nabla\Y_n(\tau)\|_{\H}^{2}\d\tau
   		\bigg]^{\frac12}
   		\nonumber\\&\leq
   		2\mathpzc{a}\sqrt{n}(\mathrm{Tr}(\Q_1))^{\frac12}
   		\E\bigg[\left(\sup\limits_{s\in[t,T\land\wi\theta_N]}
   		e^{\mathpzc{a}n\|\nabla\Y_n(s)\|_{\H}^2}\right)^{\frac{1}{2}}
   		\left(\int_t^{T\land\wi\theta_N}
   		e^{\mathpzc{a}n\|\nabla\Y_n(s)\|_{\H}^2}\|\nabla\Y_n(s)\|_{\H}^2
   		\d\tau\right)^{\frac12}\bigg]
   		\nonumber\\&\leq
   		\frac12\E\bigg[\sup\limits_{s\in[t,T\land\wi\theta_N]}
   		e^{\mathpzc{a}n\|\nabla\Y_n(s)\|_{\H}^2}\bigg]+ 2\mathpzc{a}^2n\mathrm{Tr}(\Q_1)
   		\E\bigg[\int_t^{T\land\wi\theta_N}
   		e^{\mathpzc{a}n\|\nabla\Y_n(s)\|_{\H}^2}\|\nabla\Y_n(s)\|_{\H}^2\d\tau
   		\bigg].
   	\end{align}
   	On substituting \eqref{vse15.1} into \eqref{vse15}, we obtain
   	\begin{align*}
   		\frac12\E\bigg[\sup\limits_{s\in[t,T\land\wi\theta_N]}
   		e^{\mathpzc{a}n\|\nabla\Y_n(s)\|_{\H}^2}\bigg]
   		&\leq e^{\mathpzc{a}n\|\nabla\y\|_{\H}^2}+\mathpzc{a}n
   		\mathpzc{k}\E\bigg[\int_t^{T\land\wi\theta_N}
   		e^{\mathpzc{a}n\|\nabla\Y_n(\tau)\|_{\H}^2}\d\tau\bigg]
   		\nonumber\\&\quad+
   		2\mathpzc{a}^2n\mathrm{Tr}(\Q_1)\underbrace{
   			\E\bigg[\int_t^{T\land\wi\theta_N}
   			e^{\mathpzc{a}n\|\nabla\Y_n(s)\|_{\H}^2}\|\nabla\Y_n(s)\|_{\H}^2\d\tau
   			\bigg]}_{\text{ bounded from \eqref{vsedade1} for $\alpha>\varrho^*$}}.
   	\end{align*}
   	On utilizing \eqref{vsedade1} and applying Fubini's theorem, we rewrite above as
   		\begin{align*}
   			\E\bigg[\sup\limits_{s\in[t,T\land\wi\theta_N]}
   			e^{\mathpzc{a}n\|\nabla\Y_n(s)\|_{\H}^2}\bigg]
   			&\leq 2e^{\mathpzc{a}n\|\nabla\y\|_{\H}^2}+
   			2\mathpzc{a}n\mathpzc{k}\int_t^{T}\E\bigg[
   			\sup\limits_{\tau\in[t,s\land\wi\theta_N]}
   			e^{\mathpzc{a}n\|\nabla\Y_n(\tau)\|_{\H}^2}\bigg]\d\tau
   			\nonumber\\&\quad+
   			\frac{2\mathpzc{a}\mathpzc{C}}
   			{\mathpzc{c}_1^0-\mathpzc{a}},
   	\end{align*}
   	provided $\mathpzc{a}<\mathpzc{c}_1^0$, where $\mathpzc{c}_1^0:=\frac{\alpha-\varrho^*}{\Tr(\Q_1)}$. By using  Gr\"onwall's inequality, and then passing to the limit $n\to+\infty$ together with monotone convergence theorem, we finally obtain for all $s\in[t,T]$
   	\begin{align*}
   		\E\bigg[\sup\limits_{s\in[t,T]}
   		e^{\mathpzc{a}n\|\nabla\Y_n(s)\|_{\H}^2}\bigg] 
   		\leq
   		\bigg(2e^{\mathpzc{a}n\|\nabla\y\|_{\H}^2}+\frac{2\mathpzc{a}\mathpzc{C}}
   		{\mathpzc{c}_1^0-\mathpzc{a}}
   		\bigg)e^{2\mathpzc{a}n\mathpzc{k}T}.
   	\end{align*} 
   	
   	One can estimate $(\mathcal{B}(\Y_n),\A\Y_n)$ in the following way also: 
   	\begin{align}\label{BA}
   		|(\mathcal{B}(\Y_n),\A\Y_n)|&\leq\mu\|\A\Y_n\|_{\H}^2+
   		\frac{1}{4\mu}\|\mathcal{B}(\Y_n)\|_{\H}^2
   		\nonumber\\&=
   		\mu\|\A\Y_n\|_{\H}^2+\frac{1}{4\mu}\int_{\mathbb{T}^d}|\Y_n(\xi)|^2
   		|\nabla\Y_n(\xi)|^2\d\xi
   		\nonumber\\&=
     	\mu\|\A\Y_n\|_{\H}^2+\frac{1}{4\mu}
   		\int_{\mathbb{T}^d}|\nabla\Y_n(\xi)|^2\left(|\Y_n(\xi)|^{r-1}+1\right) \underbrace{\frac{|\Y_n(\xi)|^2}{|\Y_n(\xi)|^{r-1}+1}}_{<1 \ \text{for } r\geq3}\d\xi
   		\nonumber\\&<
   			\mu\|\A\Y_n\|_{\H}^2+
   			\frac{1}{4\mu}\int_{\mathbb{T}^d}|\nabla\Y_n(\xi)|^2|\Y_n(\xi)|^{r-1}\d\xi
   			+\frac{1}{4\mu}\int_{\mathbb{T}^d}|\nabla\Y_n(\xi)|^2\d\xi.
   	\end{align}
   	 Moreover, we estimate $\big(\A^{\frac12}\f,\A^{\frac12}\Y_n\big) $ as 
   	 \begin{align}\label{vse81}
   	 		\big(\A^{\frac12}\f,\A^{\frac12}\Y_n\big) &\leq\frac{R^2}{2(\alpha-\frac{1}{4\mu}-
   	 			\mathpzc{a}\Tr(\Q_1))}+\frac{\alpha-\frac{1}{4\mu}-
   	 			\mathpzc{a}\Tr(\Q_1)}{2} \|\nabla\Y_n\|_{\H}^2.
   	 \end{align}
   	 Utilizing \eqref{BA}-\eqref{vse81} together with \eqref{vse8}-\eqref{vse8.2} into \eqref{vse7.1} and employing the Sobolev embedding $\D(\A)\hookrightarrow\D(\A^{\frac12})$, we obtain for all $s\in[t,T]$, $\P-$a.s.,
   	 \begin{align*}
   	 	&e^{\mathpzc{a}n\|\nabla\Y_n(s\land\wi\theta_N)\|_{\H}^2}+
   	 	2\mathpzc{a}n\left(\alpha-\frac{1}{4\mu}-
   	 	\mathpzc{a}\Tr(\Q_1)
   	 	\right)\int_t^{s\land\wi\theta_N}
   	 	e^{\mathpzc{a}n\|\nabla\Y_n(\tau)\|_{\H}^2}
   	 	\|\nabla\Y_n(\tau)\|_{\H}^2\d\tau
   	 	\nonumber\\&\quad+
   	 	\mathpzc{a}n\left(2\beta-\frac{1}{2\mu}\right)\int_t^{s\land\wi\theta_N}
   	 	e^{\mathpzc{a}n\|\nabla\Y_n(\tau)\|_{\H}^2}
   	 	\||\Y_n(\tau)|^{\frac{r-1}{2}}\nabla\Y_n(\tau)\|_{\H}^{2}\d\tau
   	 	\nonumber\\&\leq
   	 	e^{\mathpzc{a}n\|\nabla\y\|_{\H}^2}+\mathpzc{a}n\mathpzc{k}
   	 	\int_t^{s\land\wi\theta_N}
   	 	e^{\mathpzc{a}n\|\nabla\Y_n(\tau)\|_{\H}^2}\d\tau+
   	 	2\mathpzc{a}\sqrt{n}M_{s\land\wi\theta_N}.
   	 \end{align*}
   	 Proceeding in a similar way as we performed above, for $2\alpha>\frac{1}{2\mu}$ and $4\beta\mu\geq1$, we finally obtain
   	\begin{align*}
   	\E\bigg[\sup\limits_{s\in[t,T]}
   e^{\mathpzc{a}n\|\nabla\Y_n(s)\|_{\H}^2}\bigg] 
   \leq
   \bigg(2e^{\mathpzc{a}n\|\nabla\y\|_{\H}^2}+\frac{2\mathpzc{a}\mathpzc{C}}
   {\mathpzc{c}_2^0-\mathpzc{a}}
   \bigg)e^{2\mathpzc{a}n\mathpzc{k}T},
   \end{align*}
   	provided $\mathpzc{a}<\mathpzc{c}_2^0$, where $\mathpzc{c}_2^0:=\frac{\alpha-\frac{1}{4\mu}}
   	{\Tr(\Q_1)}$. This completes the proof of \eqref{expmoest}.
   \end{proof}
   
   \begin{remark}\label{rem-2D} 
   	1.) On $\mathbb{T}^2$, we have $(\mathcal{B}(\Y_n),\A\Y_n)=0$. As a result, the additional term  $$2\mathpzc{a}n\varrho\int_t^{s\land\wi\theta_N}
   		e^{\mathpzc{a}n\|\nabla\Y_n(\tau)\|_{\H}^2}
   		\|\nabla\Y_n(\tau)\|_{\H}^2\d\tau,$$ which is due to calculation \eqref{vse8.3}, will not appear in \eqref{vse11}. Consequently, the assumption $\alpha>\varrho^*$ or $\mu>\frac14\max\left\{\frac{1}{\alpha},\frac{1}{\beta} \right\}$, is no longer necessary in Proposition \ref{propexpest}. This allows us to obtain the bound \eqref{expmoest} without any restriction on $\mu$, $\alpha$ and $\beta$.
   		
   		2.) In Sections \ref{LaplaceLDP} and \ref{LaDePr}, we shall work under condition \eqref{eqn-con} to obtain further results and analysis.
   \end{remark}

       \section{Viscosity solution and comparison principle} \setcounter{equation}{0}\label{detstchjb}
      As mentioned in the introduction, the viscosity solution framework and the comparison principle for the associated HJB equation play a central role in establishing the Laplace limit. The Laplace integral of the solution $\Y_n(\cdot)$ of SCBF equations \eqref{stap} at a single time turns out to satisfy a nonlinear second order HJB equation (see \eqref{LDP1}). However, due to the logarithmic-exponential structure of the Laplace limit expression, the arguments require certain refinements compared to those in \cite{smtm1}. In particular, the choice of test functions need to be modified slightly from those used in \cite{smtm1} to accommodate the specific structure of our problem. Due to this modification in the test functions, the monotonicity of the operator
      $\mu\A+\mathcal{B}(\cdot)+\beta\mathcal{C}$ (Lemma \ref{monoest}) becomes crucial in obtaining the desired contradiction in the proof of the comparison principle (see \eqref{viscdef4}-\eqref{viscdef5} in Theorem \ref{comparison}). It constitutes one of the main advantages of our approach to SCBF equations \eqref{stap} compared to \cite{AS2}, where such monotonicity is not available and the authors instead rely on quantisation (or truncation) techniques to recover a monotonicity estimate.  Moreover, the proof of existence of viscosity solutions in this setting does not rely on stochastic control techniques, unlike in \cite{smtm1}. 
       
        Let us first define a test function, adapted from \cite{AS2} (also see \cite{FGSSA1}).
     \begin{definition}\label{testD}
     	A function $\uppsi:(0,T)\times\V\to\R$ is called a \emph{test function} if $$\uppsi(\cdot,\cdot)=\upvarphi(\cdot,\cdot)\pm\mathfrak{h}(\|\cdot\|_{\V}),$$ where
     	\begin{itemize}
     		\item $\upvarphi\in\mathrm{C}^{1,2}((0,T)\times\H)$ and is such that $\upvarphi,\upvarphi_t,\D\upvarphi$ and $\D^2\upvarphi$ are uniformly continuous on $[\eps,T-\eps]\times\H$ for every $\eps>0$;
     		\item $\mathfrak{h}\in\mathrm{C}^2([0,+\infty))$ and is such that $\mathfrak{h}'(0)=0, \ \mathfrak{h}''(0)>0, \ \mathfrak{h}'(\theta)>0$ for $\theta\in(0,+\infty)$.
     	\end{itemize}
     \end{definition}
     
     \begin{remark}\label{reFrede}
      We remark that even though $\|\cdot\|_{\V}$ is not differentiable at $0$, the function $\mathfrak{h}(\|\cdot\|_{\V})\in\C^2(\V)$. Therefore, the terms involving $\D\mathfrak{h}$ and $\D^2\mathfrak{h}$ have to be understood in a proper way. Following \cite{FGSSA}, we define
     	\begin{align*}
     		\D\mathfrak{h}(\y)&:= \frac{\mathfrak{h}'(\|\y\|_{\V})}{\|\y\|_{\V}}\A\y,\\
     		\D^2\mathfrak{h}(\y)&:=\frac{\mathfrak{h}'(\|\y\|_{\V})}{\|\y\|_{\V}}
     		\left(\A-\frac{1}{\|\y\|_{\V}^2}(\A\y\otimes\A\y)\right)+
     		\frac{\mathfrak{h}''(\|\y\|_{\V})}{\|\y\|_{\V}^2}(\A\y\otimes\A\y)
     	\end{align*}
     	and we write with slight abuse of notation $\D\uppsi:=\D\upvarphi\pm\D\mathfrak{h}$ and $\D^2\uppsi:=\D^2\upvarphi\pm\D^2\mathfrak{h}$ for a function
     	$\uppsi=\upvarphi\pm\mathfrak{h}$.
     \end{remark}
     
     We now define the viscosity solution of the following terminal value problem for HJB equation:
     	\begin{equation}\label{thjb}
     	\left\{
     \begin{aligned}
     	&u_t+\frac{k}{2}\mathrm{Tr}(\Q\D^2u)-(\mu\A\y+\mathcal{B}(\y)+\alpha\y+\beta\mathcal{C}(\y),\D u)+F(t,\y,\D u)=0, \ \text{ in } \ (0,T)\times\V, \\
     	&u(T,\y)=g(\y),
     \end{aligned}
     \right.
     \end{equation}
     where $F:[0, T]\times\V\times\H\to\R$ and $k\geq0$.
     \begin{definition}\label{viscsoLndef}
     	A weakly sequentially upper-semicontinuous (respectively, lower-semicontinuous) function $u:(0,T)\times\V\to\R$ is called \emph{a viscosity subsolution} (respectively, \emph{supersolution}) of \eqref{thjb} if whenever $u-\uppsi$ has a local maximum (respectively, $u+\uppsi$ has a local minimum) at a point $(t,\y)\in(0,T)\times\V$ for every test function $\uppsi$, then $\y\in\V_2$ and 
     	\begin{align*}
     	&\uppsi_t(t,\y)+\frac{k}{2}\mathrm{Tr}(\Q\D^2\uppsi(t,\y))\nonumber\\&\quad-
     	(\mu\A\y+ \mathcal{B}(\y)+\alpha\y+\beta\mathcal{C}(\y),\D\uppsi(t,\y))
     	+F(t,\y,\D\uppsi(t,\y))\geq0
     \end{align*}
     (respectively, 
     \begin{align*}
     	-&\uppsi_t(t,\y)-\frac{k}{2}\mathrm{Tr}(\Q\D^2\uppsi(t,\y))\nonumber\\&+
     	(\mu\A\y+\mathcal{B}(\y)+\alpha\y+\beta\mathcal{C}(\y),\D\uppsi(t,\y))+F(t,\y,-\D\uppsi(t,\y))\leq0.)
     \end{align*}
     	A \emph{viscosity solution} of \eqref{thjb} is a function which is both \emph{viscosity subsolution and viscosity supersolution}.
     \end{definition}
     
%     \begin{hypothesis}\label{trQ1f}
%     	Assume that $\Tr(\Q_1)<\infty$ and $\f:[0,T]\to\V$ is bounded and continuous.
%     \end{hypothesis}
     
     \begin{remark}
     	The point of maxima and minima in  Definition \ref{viscsoLndef} can be assumed to be strict and global. Further, for bounded sub and supersolutions,  $\upvarphi,\upvarphi_t,\D\upvarphi$ and $\D^2\upvarphi$ can be assumed to be uniformly continuous on $(0,T)\times\H$.
     \end{remark}
	
	\subsubsection{Comparison Principle}\label{comparisonP} 
	Let us now prove the comparison principle of viscosity solutions for the following terminal value problem, which is inspired by the techniques in  \cite{FGSSA,FGRA,AS2}.
	\begin{equation}\label{thjbcomp}
	\left\{
	\begin{aligned}
	&(u_n)_t+\frac{1}{2n}\mathrm{Tr}(\Q\D^2u_n)-
	\frac12\|\Q^{\frac12}\D u_n\|_{\H}^2\\ &\quad+ (-\mu\A\y-\mathcal{B}(\y)-\alpha\y-\beta\mathcal{C}(\y)+\f(t),\D u_n) =0, \ \text{ in } \ (0,T)\times\V, \\
	&	u_n(T,\y)=g(\y),
	\end{aligned}
	\right.
\end{equation}
   where $n\geq1$ or $n=+\infty$.
	\begin{theorem}\label{comparison}
	Assume that Hypotheses \ref{trQ1} and \ref{fhyp} hold and $g\in\mathrm{Lip}_b(\H)$. Let $u$ be a viscosity subsolution of \eqref{thjbcomp} and $v$ be a viscosity supersolution of \eqref{thjbcomp} such that
		\begin{equation}\label{bdd}
			\left\{
			\begin{aligned}
				u(t,\y), -v(t,\y)&\leq\mathpzc{C}, \  \text{ for some } \ \mathpzc{C}>0,\\
			\lim\limits_{t\to T} \{\big(u(t,\y)-g(\y)\big)_{+}+\big(v(t,\y)-g(\y)\big)_{-}\}
			&=0, \ \text{ uniformly on bounded sets of } \ \V,
			\end{aligned}
			\right.
		\end{equation}
      	{where for any real-valued function $f$, $f_{+}=\max\{f,0\}$ and $f_{-}=\max\{-f,0\}$.} Moreover, for $\psi=u$ or $\psi=v$, we assume that 
      	\begin{align}\label{uvL}
      		|\psi(t,\y)-\psi(t,\x)|\leq\mathfrak{L}\|\x-\y\|_{\H}, \ \text{ for some } \ \mathfrak{L}\geq0,
      	\end{align}
      	for all $t\in(0,T)$ and $\y,\x\in\V$. Then, for $r>3$ and $r=3$ with $2\beta\mu\geq1$, we have
      	$$u\leq v\ \text{ on }\ (0,T)\times\V.$$
	\end{theorem}
	
	 \begin{proof}
     We only consider the case when $n<+\infty$ in \eqref{thjbcomp}. The case $n=+\infty$ in \eqref{thjbcomp} can be proven in a similar way as it reduces to the first order HJB. The proof is divided into the following  steps:
     \vskip 0.2cm
     \noindent
     \textbf{Step-1:} Define
		\begin{align}\label{ugvg}
			u_{\upgamma}(t,\y):=u(t,\y)-\frac{\upgamma}{t}\  \text{ and } \ v_{\upgamma}(t,\y):=v(t,\y)+\frac{\upgamma}{t},
		\end{align}
        for some $\gamma>0$. It is sufficient to prove that
        $u_{\upgamma}\leq v_{\upgamma}$, for all $(t,\y)\in(0,T)\times\V$ and all $\upgamma>0$. Then, we can obtain $u\leq v$ by letting $\gamma\to0$. 
        
	We assume that $u_{\upgamma}\not\leq v_{\upgamma}$ on $(0,T)\times\V$. Then,  there is a $\kappa>0$ such that for sufficiently small $\upgamma>0$, we have (\cite[Theorem 3.50]{GFAS})
	\begin{align*}
		0<\mathfrak{m}:=\lim\limits_{R\to+\infty}\lim\limits_{q\to0}\lim\limits_{\upnu\to0}\sup\bigg\{& u_{\upgamma}(t,\y)-v_{\upgamma}(s,\x): \|\y-\x\|_{\H}<q, \ \|\y\|_{\V},\|\x\|_{\V}\leq R, \\& |t-s|<\upnu, \ \kappa<t,s\leq T\bigg\}.
	\end{align*}
	We also define
	\begin{align*}
		\mathfrak{m}_{\eps}:=\lim\limits_{q\to0}\lim\limits_{\upnu\to0}\sup\bigg\{&u_{\upgamma}(t,\y)-v_{\upgamma}(s,\x)-\frac{\|\y-\x\|_{\H}^2}{2\eps}:\\& \|\y\|_{\V},\|\x\|_{\V}\leq R, \ \y,\x\in\V, \ \ |t-s|<\upnu, \ \kappa<t,s\leq T\bigg\},\\
		\mathfrak{m}_{\eps,\delta}:=\lim\limits_{\upnu\to0}\sup\bigg\{&u_{\upgamma}(t,\y)-v_{\upgamma}(s,\x)-\delta  \|\y\|_{\V}^2-\delta  \|\x\|_{\V}^2-\frac{\|\y-\x\|_{\H}^2}{2\eps}: \\&\y,\x\in\V, \ \ |t-s|<\upnu, \ \kappa<t,s\leq T\bigg\},\\
		\mathfrak{m}_{\delta,\eps,\eta}:=\sup\bigg\{&u_{\upgamma}(t,\y)-v_{\upgamma}(s,\x)-\delta \|\y\|_{\V}^2-\delta \|\x\|_{\V}^2-\frac{\|\y-\x\|_{\H}^2}{2\eps}\\&-\frac{(t-s)^2}{2\eta}: \ \y,\x\in\V, \ \kappa<t,s\leq T\bigg\}.
	\end{align*}
	From the above definitions, we have the following convergences:
	\begin{align}\label{comconv}
		\mathfrak{m}\leq\lim\limits_{\eps\to0}\mathfrak{m}_\eps, \ \  \mathfrak{m}_\eps=
		\lim\limits_{\delta\to0}\mathfrak{m}_{\eps,\delta}, \ \ \text{ and } \ 
		\mathfrak{m}_{\eps,\delta}=\lim\limits_{\eta\to0}\mathfrak{m}_{\eps,\delta,\eta}.
	\end{align}
		For $\eps,\delta,\eta>0$, we define the function $\Phi$ on $(0,T]\times\H$ by
		\begin{equation*}
			\Phi(t,s,\y,\x)=\left\{
			\begin{aligned}
				&u_{\upgamma}(t,\y)-v_{\upgamma}(s,\x)-
				\frac{\|\y-\x\|_{\H}^2}{2\eps}-\delta\|\y\|_{\V}^2 \\&-\delta\|\x\|_{\V}^2-\frac{(t-s)^2}{2\eta}, \  &&\text{ if } \ \y,\x\in\V,\\
				&-\infty, \  &&\text{ if } \ \y,\x\notin\V.
			\end{aligned}
			\right.
		\end{equation*}
		Clearly, $\Phi\to-\infty$ as $\max\{\|\y\|_{\V},\|\x\|_{\V}\}\to+\infty$.
	
		\vskip 0.2cm
		\noindent
		\textbf{Step-II:} \textbf{\emph{$\Phi$ is weakly sequentially upper-semicontinuous on $(0,T]\times(0,T]\times\H\times\H$.} } 
		By the properties of norms, the functions $\y\mapsto\|\y\|_{\V}^2$, $\x\mapsto\|\x\|_{\V}^2$, and $(\y,\x)\mapsto\|\y-\x\|_{\H}^2$  are weakly sequentially lower-semicontinuous  in  $\H$, $\H$ and $\H\times\H$, respectively. Moreover, $\u_{\upgamma}$ is a weakly sequentially upper-semicontinuous function in $(0,T)\times\V$. We will now show that
			$u_{\upgamma}(t,\y)-\delta\|\y\|_{\V}^2 $ \text{is weakly sequentially upper-semicontinuous on} $ (0,T)\times\H.$
	
		Suppose this is not true. Then, there exists a sequence $(t_n)_{n\geq1}$ in $(0,T)$ with $t_n\to t\in(0,T)$ and a sequence $(\y_n)_{n\geq1}$ in $\H$ with $\y_n\rightharpoonup\y\in\H$ such that
		\begin{align}\label{contra}
			\limsup\limits_{n\to+\infty}\left(u_{\upgamma}(t_n,\y_n)-
			\delta\|\y_n\|_{\V}^2\right)>
			u_{\upgamma}(t,\y)-\delta\|\y\|_{\V}^2.
		\end{align}
		Now, if $\liminf\limits_{n\to+\infty}\|\y_n\|_{\V}=+\infty$, then \eqref{contra} is impossible because of the assumption \eqref{bdd} on $\u$. Therefore, $\liminf\limits_{n\to+\infty}\|\y_n\|_{\V}<+\infty$ and by the properties of limit inferior, there exists a subsequence (still denoted by $(t_n,\y_n)$) such that $\limsup\limits_{n\to+\infty}\|\y_n\|_{\V}<+\infty$. By an application of the Banach-Alaoglu theorem, we then have $\y_n\rightharpoonup\y$ in $\V$ (along a subsequence), which implies $\|\y\|_{\V}\leq\liminf\limits_{n\to+\infty}\|\y_n\|_{\V}$ and therefore, from \eqref{contra}, we further have
		\begin{align*}
	\limsup\limits_{n\to+\infty}u_{\upgamma}(t_n,\y_n)>u_{\upgamma}(t,\y),
		\end{align*}
		 which yields a contradiction to the fact that $u_{\upgamma}$ is weakly sequentially upper-semicontinuous.  In the same way, one shows that $v_{\upgamma}(s,\x)-\delta\|\x\|_{\V}^2$ is weakly sequentially lower-semicontinuous on $(0,T)\times\H$. This establishes the desired claim. Consequently, by the definition of viscosity solution, $\Phi$ attains a global maximum over $(0,T]\times(0,T]\times\H\times\H$ at some point $(\bar{t},\bar{s},\bar{\y},\bar{\x})\in(0,T]\times(0,T]\times\V\times\V$. 
%		 {\color{Maroon}Moreover, from \eqref{comconv}, there exists $\mathpzc{k}>0$ such that
%		 \begin{align}\label{phtc}
%		 \Phi(\bar{t},\bar{s},\bar{\y},\bar{\x})\geq\mathpzc{k}>0.
%		 \end{align}}
		\vskip 0.1 cm 
	\textbf{Claim:} \emph{$\|\bar{\y}\|_{\V},\|\bar{\x}\|_{\V}$ are bounded independently of $\eps$, for a fixed $\delta>0$.} 
		 Indeed, for any $\y\in\V$, we have 
		 \begin{align*}
		  \Phi(\bar{t},\bar{s},\y,\y)\leq\Phi(\bar{t},\bar{s},\bar{\y},\bar{\x}).
		 \end{align*}
		 Using the definition of $\Phi$ and \eqref{bdd}, we obtain
		 \begin{align*}
		 \delta\|\bar{\y}\|_{\V}^2+\delta\|\bar{\x}\|_{\V}^2&\leq
		 u_{\upgamma}(\bar{t},\bar{\y})-u_{\upgamma}(\bar{t},\y)+
		 v_{\upgamma}(\bar{s},\y)-v_{\upgamma}(\bar{s},\bar{\x})
		 +2\delta\|\y\|_{\V}^2
		 \nonumber\\&\leq\mathpzc{C}+2\delta\|\y\|_{\V}^2,
		 \end{align*}
		 for all $\y\in\V$ and for all $0<\bar{t},\bar{s}<T$. In particular for $\y=\boldsymbol{0}$, we find 
		 \begin{align*}
		 \delta\|\bar{\y}\|_{\V}^2+\delta\|\bar{\x}\|_{\V}^2
		 \leq\mathpzc{C},
		 \end{align*}
		 Thus, for a fixed $\delta$, we conclude that $\bar{\x}$ and $\bar{\y}$ are bounded independently of $\eps$ in $\V$.
                    
          By using \eqref{comconv}, we have the following:
          \begin{align}\label{copm1}
          	\limsup\limits_{\eta\to0}\frac{(\bar{t}-\bar{s})^2}{2\eta}=0 \ \text{ for fixed } \ \delta>0, \eps>0,
          \end{align}
          and
          \begin{align}\label{copm2}
          	\limsup\limits_{\delta\to0}\limsup\limits_{\eta\to0} 
          	\delta(\|\bar{\y}\|_{\V}^2+\|\bar{\x}\|_{\V}^2)=0 \ \text{ for fixed } \ \eps>0.
          \end{align}
          We can assume this maximum point to be strict (for instance, see \cite[Lemma 3.37, Chapter 3]{GFAS}) and it follows by the definition of viscosity solution that $\bar{\y},\bar{\x}\in\V_2$. Consequently, from \eqref{copm1}-\eqref{copm2}, the fact that $g\in\mathrm{Lip}_b(\H)$ and 
          from \eqref{bdd}, it follows that for sufficiently small $\upgamma$ and $\delta$, we must have $$\bar{t},\bar{s}<T,$$ provided that $\eta$ and $\eps$ are chosen small enough.
        
        Moreover, by using the definition of $\Phi$ and the fact that $\Phi(\bar{t},\bar{s},\bar{\y},\bar{\y})\leq\Phi(\bar{t},\bar{s},\bar{\y},\bar{\x})$ yield
          \begin{align*}
          \frac{\|\bar{\y}-\bar{\x}\|_{\H}^2}{2\eps}\leq\mathfrak{L}\|\bar{\y}-\bar{\x}\|_{\H}+\delta\|\bar{\y}\|_{\H}^2,
          \end{align*}
          which in view of \eqref{copm2}, implies that
          \begin{align}\label{copm3}
          	\limsup\limits_{\delta\to0}\limsup\limits_{\eta\to0}
          	\frac{\|\bar{\y}-\bar{\x}\|_{\H}}{\eps}\leq2\mathfrak{L}.
          \end{align}
       
       \textbf{Step-IV:} \emph{Reduction to finite-dimensional space.} 
       Let $\H_1\subset\H_2\subset\ldots$ be finite-dimensional subspaces of $\H$, spanned by the eigenfunctions of $\A$, with $\overline{\bigcup_{N=1}^\infty}\H_N=\H$. For $N>1$, let $\mathbf{P}_N$ denote the orthogonal projection onto $\H_N$. It is clear that $\mathbf{P}_N$  is a bounded linear operator on $\H$ and so is $\mathbf{Q}_N:=\I-\mathbf{P}_N$. Let us denote  $\H_N^{\perp}=\mathbf{Q}_N\H$. We then have an orthogonal decomposition $\H=\H_N\times\H_N^{\perp}$. For any $\y\in\H$, we write $\y_N:=\mathbf{P}_N\y\in\H_N$ and $\y_N^{\perp}:=\mathbf{Q}_N\y\in\H_N^{\perp}$ so that $\y=(\mathbf{P}_N\y,\mathbf{Q}_N\y)=(\y_N,\y_N^{\perp})$. 
       
       Let us now fix $N\in\N$. By using the properties of $\mathbf{P}_N$ and $\mathbf{Q}_N$, we have the following straightforward identities:
       \begin{align*}
       	\|\y-\x\|_{\H}^2&=\|\mathbf{P}_N(\y-\x)\|_{\H}^2
       	+\|\mathbf{Q}_N(\y-\x)\|_{\H}^2,\label{fnPQ1}\\
       	\|\mathbf{Q}_N(\y-\x)\|_{\H}^2
       	&\leq2(\mathbf{Q}_N(\bar{\y}-\bar{\x}),\y-\x)-
       	\|\mathbf{Q}_N(\bar{\y}-\bar{\x})\|_{\H}^2+2\|\mathbf{Q}_N(\y-\bar{\y})\|_{\H}^2\nonumber\\&\quad
       	+2\|\mathbf{Q}_N(\x-\bar{\x})\|_{\H}^2\nonumber,
       \end{align*}
       with the equality in the second inequality if $\y=\bar{\y}$ and $\x=\bar{\x}$.
       Let us now define 
       \begin{equation*}
       	\wi u(t,\y):=
       	\left\{
       	\begin{aligned}
       		&u_{\upgamma}(t,\y)-\frac{(\y,\mathbf{Q}_N(\bar{\y}-\bar{\x}))}{\eps}+\frac{\|\mathbf{Q}_N(\bar{\y}-\bar{\x})\|_{\H}^2}{2\eps}-
       		\frac{\|\mathbf{Q}_N(\y-\bar{\y})\|_{\H}^2}{\eps}\nonumber\\&\quad-
       		\delta \|\y\|_{\V}^2, \ \text{ when } \ \y,\x\in\V,\\
       		&-\infty, \ \text{ when } \ \y,\x\notin\V,
       	\end{aligned}
       	\right.
       \end{equation*}
       and 
       \begin{equation*}
       	\wi v(s,\x):=
       	\left\{
       	\begin{aligned}
       		&v_{\upgamma}(s,\x)-\frac{(\x,\mathbf{Q}_N(\bar{\y}-\bar{\x}))}{\eps}+
       		\frac{\|\mathbf{Q}_N(\x-\bar{\x})\|_{\H}^2}{\eps}+
       		\delta \|\x\|_{\V}^2, \ \text{ when } \ \y,\x\in\V,\\
       		&+\infty, \ \text{ when } \ \y,\x\notin\V.
       	\end{aligned}
       	\right.
       \end{equation*}
       We emphasize that such extended $\wi{u}$ and $\wi{v}$ are weakly sequentially upper-semicontinuous and lower-semicontinuous on $(0,T]\times\H$, respectively. It follows that the function
       \begin{align*}
       	\wi\Phi(t,s,\y,\x):=\wi u(t,\y)-\wi v(s,\x)-\frac{\|\mathbf{P}_N(\y-\x) \|_{\H}^2}{2\eps}-\frac{(t-s)^2}{2\eta},
       \end{align*}
       always satisfies $\wi\Phi\leq\Phi$ whenever $\y,\x\in\V$ and it attains a strict global maximum over $(0,T]\times(0,T]\times\H\times\H$ at $(\bar{t},\bar{s},\bar{\y},\bar{\x})$ where $\wi\Phi(\bar{t},\bar{s},\bar{\y},\bar{\x}) =\Phi(\bar{t},\bar{s},\bar{\y},\bar{\x})$.

       \vskip 0.2cm
       \noindent	
       \textbf{Step-V:} \emph{Finite-dimensional maximum principle.} 
       We now define, for $\y_N,\x_N\in\H_N$, the functions 
       \begin{align}\label{supNminN}
       	\wi{u}_1(t,\y_N):=\sup\limits_{\y_N^{\perp}\in\H_N^{\perp}} \wi{u}(t,\y_N,\y_N^{\perp}) \ \text{ and } \ 
       	\wi{v}_1(s,\x_N):=\inf\limits_{\x_N^{\perp}\in\H_N^{\perp}} \wi{v}(s,\x_N,\x_N^{\perp}). 
       \end{align}
       Since $\wi{u}$ and $-\wi{v}$ are weakly sequentially upper-semicontinuous on $(0,T]\times\H$, the functions $\wi{u}_1$ and $-\wi{v}_1$ are upper-semicontinuous on $(0,T]\times\H_N$. 
       Note that
       	\begin{align*}
       		\wi\Phi(\bar{t},\bar{s},\bar{\y},\bar{\x})&=
       		\wi u(\bar{t},\bar{\y})-\wi v(\bar{s},\bar{\x})- \frac{\|\mathbf{P}_N(\bar{\y}-\bar{\x}) \|_{\H}^2}{2\eps}-\frac{(\bar{t}-\bar{s})^2}{2\eta}
       		\nonumber\\&=
       		\wi u(\bar{t},\mathbf{P}_N\bar{\y},\mathbf{Q}_N\bar{\y})-\wi v(\bar{s},\mathbf{P}_N\bar{\x},\mathbf{Q}_N\bar{\x})- \frac{\|\mathbf{P}_N(\bar{\y}-\bar{\x}) \|_{\H}^2}{2\eps}-\frac{(\bar{t}-\bar{s})^2}{2\eta}
       		\nonumber\\&\leq
       		\wi u_1(\bar{t},\mathbf{P}_N\bar{\y})-\wi v_1 (\bar{s},\mathbf{P}_N\bar{\x})- \frac{\|\mathbf{P}_N(\bar{\y}-\bar{\x}) \|_{\H}^2}{2\eps}-\frac{(\bar{t}-\bar{s})^2}{2\eta}.
       	\end{align*}
       	Since, $\wi\Phi$ has global strict maxima at $(\bar{t},\bar{s},\bar{\y},\bar{\x}),$ where $\wi\Phi(\bar{t},\bar{s},\bar{\y},\bar{\x}) =\Phi(\bar{t},\bar{s},\bar{\y},\bar{\x})$, therefore, by the definition of $\wi{u}$ and $\wi{v}$, it follows that
       	\begin{align}\label{fndmax}
       		\wi{u}_1(\bar{t},\mathbf{P}_N\bar{\y})=\wi{u}(\bar{t},\bar{\y})  \ \text{ and } \ 
       		\wi{v}_1(\bar{s},\mathbf{P}_N\bar{\x})=\wi{v}(\bar{s},\bar{\x}).
       \end{align}
       Let us now define a map $\Phi_N:(0,T]\times(0,T]\times\H_N\times\H_N\to\R$ by
       \begin{align*}
       	\Phi_N(t,s,\y_N,\x_N)&=\wi{u}_1(t,\y_N)-\wi{v}_1(s,\x_N) -\frac{\|\y_N-\x_N\|_{\H}^2}{2\eps}-\frac{(t-s)^2}{2\eta}
       	\nonumber\\&=
       	\sup\limits_{\y_N^{\perp},\x_N^{\perp}\in\H_N^{\perp}}
       	\wi\Phi\left(t,s,(\y_N,\y_N^{\perp}),(\x_N,\x_N^{\perp})\right).
       \end{align*}
       In view of \eqref{fndmax} and the fact that $\wi\Phi$ has a global strict maxima at $(\bar{t},\bar{s},\bar{\y},\bar{\x})$, it is easy to check that $\Phi_N$ attains a strict global maximum over $(0,T]\times(0,T]\times\H_N\times\H_N$ at $(\bar{t},\bar{s},\bar{\y}_N,\bar{\x}_N)=(\bar{t},\bar{s},\mathbf{P}_N\bar{\y},\mathbf{P}_N\bar{\x})$. By an application of the finite-dimensional maximum principle \cite[Theorem 8.3]{MGHL} (also see \cite[Theorem E.11, pp. 871]{GFAS}), for every $n\in\N$ there exist points $t^n,s^n\in(0,T)$ and $\y_{N}^n,\x_N^n\in\H_N$ such that
       \begin{equation}\label{limits1}
       	\left\{
       	\begin{aligned}
       		t^n\to\bar{t}, \ s^n\to\bar{s}, \ \y_N^n\to\bar{\y}_N, \  \x_N^n\to\bar{\x}_N, \ \text{ as } \ n\to+\infty,\\
       		\wi{u}_1(t^n,\y_N^n)\to\wi{u}_1(\bar{t},\bar{\y}_N), \ \wi{v}_1(s^n,\x_N^n)\to\wi{v}_1(\bar{s},\bar{\x}_N),  \ \text{ as } \ n\to+\infty.
       	\end{aligned}
       	\right.
       \end{equation}
      Moreover, there exist functions $\varphi_n,\psi_n\in\C^{1,2}((0,T)\times\H_N)$ with uniformly continuous derivatives such that $\wi{u}_1-\varphi_n$ and $-\wi{v}_1+\psi_n$ have strict, global maxima at $(t^n,\y_N^n)$ and $(s^n,\x_N^n)$, respectively, and 
       \begin{equation}\label{limits11}
       	\left\{
       	\begin{aligned}
       		(\varphi_n)_t(t^n,\y_N^n)\to\frac{\bar{t}-\bar{s}}{\eta}, \ \ &\D\varphi_n(t^n,\y_N^n)\to\frac{1}{\eps}(\bar{\y}_N-\bar{\x}_N), \ \text{ as } \ n\to+\infty,\\
       		(\psi_n)_t(s^n,\x_N^n)\to\frac{\bar{t}-\bar{s}}{\eta}, \ \ &\D\psi_n(s^n,\x_N^n)\to\frac{1}{\eps}(\bar{\y}_N-\bar{\x}_N), \ \text{ as } \ n\to+\infty,\\
       		\D^2\varphi_n(t^n,\y_N^n)\to Y_N, \ \ &\D^2\psi_n(s^n,\y_N^n)\to X_N, \ \text{ as } \ n\to+\infty,
       	\end{aligned}
       	\right.
       \end{equation}
       where $X_N=\mathbf{P}_N X_N\mathbf{P}_N$, $Y_N=\mathbf{P}_N Y_N\mathbf{P}_N$ are bounded independently of $n$ and satisfy 
       \begin{align}\label{xnyn}
       	-\frac{3}{\eps}\begin{pmatrix} \mathbf{P}_N \ \ \ 0\\0 \  \  \ \mathbf{P}_N \end{pmatrix}\leq\begin{pmatrix} Y_N \ \ \ 0 \\0  \  \ \ X_N\end{pmatrix}\leq\frac{3}{\eps}
       	\begin{pmatrix} \mathbf{P}_N \ \ \ -\mathbf{P}_N \\-\mathbf{P}_N  \ \ \  \mathbf{P}_N\end{pmatrix}.
       \end{align}
       It is clear from \eqref{xnyn} that $X_N,Y_N$ satisfy $Y_N\leq X_N$, that is, $Y_N\xi\cdot\xi\leq X_N\xi\cdot\xi$ for $\xi\in\R^N$, where `$\cdot$' indicates the Euclidean product on $\R^N$.
       
       \vskip 0.2cm
       \noindent
       \textbf{Step-VI:} \emph{Back to infinite-dimensional space.} Let us consider a map $\Phi_N^n:(0,T)\times(0,T)\times\H\times\H\to\R$ given by 
       \begin{align*}
       	\Phi_N^n(t,s,\y,\x)=\wi{u}(t,\y)-\wi{v}(s,\x)-\varphi_n(t,\mathbf{P}_N\y)+\psi_n(s,\mathbf{P}_N\x).
       \end{align*}
       This map has the variable split and by the definition of $\wi{u}$ and $\wi{v}$, it attains global maximum (which can be assumed strict) over $(0,T]\times(0,T]\times\H\times\H$ at some point $(\hat{t}^n,\hat{s}^n,\hat{\y}^n,\hat{\x}^n)\in(0,T]\times(0,T]\times\V\times\V$.
        Setting 
       \begin{align*}
       	\hat{\y}^n=(\mathbf{P}_N\hat{\y}^n,\mathbf{Q}_N\hat{\y}^n), \ 
       	\hat{\x}^n=(\mathbf{P}_N\hat{\x}^n,\mathbf{Q}_N\hat{\x}^n), \ \text{ for every } \ \y_N^{\perp},\x_N^{\perp}\in\H_N^{\perp},
       \end{align*}
       then, from \eqref{supNminN}, we have
       \begin{align}\label{PPqqnn}
       	&\wi{u}_1(\hat{t}^n,\mathbf{P}_N\hat{\y}^n)-\wi{v}_1(\hat{s}^n,\mathbf{P}_N\hat{\x}^n)-
       	\varphi_n(\hat{t}^n,\mathbf{P}_N\hat{\y}^n)+\psi_n(\hat{s}^n,\mathbf{P}_N\hat{\x}^n)
       	\nonumber\\&
       	\geq\wi{u}(\hat{t}^n,\mathbf{P}_N\hat{\y}^n,\mathbf{Q}_N\hat{\y}^n)-
       	\wi{v}(\hat{s}^n,\mathbf{P}_N\hat{\x}^n,\mathbf{Q}_N\hat{\x}^n)-
       	\varphi_n(\hat{t}^n,\mathbf{P}_N\hat{\y}^n)+\psi_n(\hat{s}^n,\mathbf{P}_N\hat{\x}^n)
       	\nonumber\\&
       	\geq\wi{u}(t^n,\y_N^n,\y_N^{\perp})-\wi{v}(s^n,\x_N^n,\x_N^{\perp})-
       	\varphi_n(t^n,\y_N^n)+\psi_n(s^n,\x_N^n),
       \end{align}
       where in the last inequality we have used the fact that $\Phi_N^n$ has a strict global maxima at $(\hat{t}^n,\hat{s}^n,\hat{\y}^n,\hat{\x}^n)$. On taking suprema over $\y_N^{\perp}$ and $\x_N^{\perp}$ in the above inequality, we obtain
       \begin{align}\label{maxNminN}
       	&\wi{u}_1(\hat{t}^n,\mathbf{P}_N\hat{\y}^n)-\wi{v}_1(\hat{s}^n,\mathbf{P}_N\hat{\x}^n)-
       	\varphi_n(\hat{t}^n,\mathbf{P}_N\hat{\y}^n)+\psi_n(\hat{s}^n,\mathbf{P}_N\hat{\x}^n)
       	\nonumber\\&
       	\geq\wi{u}_1(t^n,\y_N^n)-\wi{v}_1(s^n,\x_N^n)-
       	\varphi_n(t^n,\y_N^n)+\psi_n(s^n,\x_N^n).
       \end{align}
       Since $\wi{u}_1-\varphi_n$ and $-\wi{v}_1+\psi_n$ have unique strict, global maxima at $(t^n,\y_N^n)$ and $(s^n,\x_N^n)$, respectively, therefore
       	\eqref{maxNminN} implies
       	\begin{align}\label{limits22}
       		\hat{t}^n=t^n, \ \hat{s}^n=s^n, \  \mathbf{P}_N\hat{\y}^n=\y_N^n \ \text{ and } \  
       		\mathbf{P}_N\hat{\x}^n=\x_N^n,
       	\end{align}  
       	and the fact that $\Phi_N^n$ has a global maximum at $(\hat{t}^n,\hat{s}^n,\hat{\y}^n,\hat{\x}^n)$, the second inequality in \eqref{PPqqnn} gives 
       	\begin{align}\label{limits2}
       		\wi{u}(\hat{t}^n,\hat{\y}^n)=\wi{u}_1(t^n,\y_N^n)  \ \text{ and } \ 
       		\wi{v}(\hat{s}^n,\hat{\x}^n)=\wi{v}_1(s^n,\x_N^n).
       	\end{align}
       	In view of \eqref{limits1} and \eqref{fndmax}, we further conclude 
       	\begin{equation}\label{limits23}
       		\left\{
       		\begin{aligned}
       			\wi{u}(\hat{t}^n,\hat{\y}^n)=\wi{u}_1(t^n,\y_N^n)\to
       			\wi{u}_1(\bar{t},\bar{\y}_N)=\wi{u}(\bar{t},\bar{\y}), \ \text{ as } \ n\to+\infty,\\ \wi{v}(\hat{s}^n,\hat{\x}^n)=\wi{v}_1(s^n,\x_N^n)\to
       				\wi{v}_1(\bar{s},\bar{\x}_N)=\wi{v}(\bar{s},\bar{\x}),  \ \text{ as } \ n\to+\infty.
       		\end{aligned}
       		\right.
       \end{equation}
       From \eqref{limits1}, \eqref{limits22}-\eqref{limits23} together with weak upper-semicontinuity of $\wi{u}$ and weak lower-semicontinuity of $\wi{v}$, we have 
       \begin{align*}
       	\hat{\y}^n\to\bar{\y} \ \text{ and } \ \hat{\x}^n\to\bar{\x} \ \text{ as } 
       	\ n\to+\infty \ \text{ in } \H.
       \end{align*}
      
      Furthermore, by the definition of $\wi{u}$ and $\wi{v}$ and the fact that $\Phi_N^n$ has a strict global maxima at $(\hat{t}^n,\hat{s}^n,\hat{\y}^n,\hat{\x}^n)$, we must have 
      \begin{align}\label{conv3}
      	\|\hat{\y}^n\|_{\V}\leq \mathpzc{C} \ \text{ and } \ \|\hat{\x}^n\|_{\V}\leq \mathpzc{C},
      \end{align}
      for some constant $ \mathpzc{C}$ independent of $n$. Therefore, by an application of the Banach-Alaoglu theorem, we have the following weak convergences (along a subsequence still denoted by the same symbol) as $n\to+\infty$:
\begin{align*}
	\hat{\y}^n\rightharpoonup\wi\y \ \text{ and } \ \hat{\x}^n\rightharpoonup\wi\x, \ \text{ in } \ \V.
\end{align*}
By the uniqueness of weak limits, we further have $\wi\x=\bar{\x}$ and $\wi\y=\bar{\y}$.
Note that weak convergent sequences are bounded. 	The above weak convergence together with \eqref{conv3}, and the fact that  $(\bar{t},\bar{s},\bar{\y},\bar{\x})$ is a point of global maxima for $\Phi$, imply 
\begin{align*}
	\|\hat{\y}^n\|_{\V}\to\|\bar{\y}\|_{\V},  \ \|\hat{\x}^n\|_{\V}\to\|\bar{\x}\|_{\V},  \ \text{ as } \ n\to+\infty. 
\end{align*}
which in turn gives the following strong convergence as $n\to+\infty$ (Radon-Riesz property):
\begin{align}\label{eqn-v-norm-conv}
	\hat{\y}^n\to\bar{\y}\ \text{ and } \ \hat{\x}^n\to\bar{\x} \ \text{ in } \ \V. 
\end{align}

			\noindent
		\vskip 0.2cm
		\textbf{Step-VII:} \emph{Applying definition of viscosity solution.} 
		We now use that $u_{\upgamma}$ and $v_{\upgamma}$ are viscosity subsolution and supersolution, respectively. Let us define a test function 
		\begin{align*}
			\varphi(t,\y):=&\varphi_n(t,\y)+\frac{(\y,\mathbf{Q}_N(\bar{\y}-\bar{\x}))}{\eps}+\frac{\|\mathbf{Q}_N(\y-\bar{\y})\|_{\H}^2}{\eps}-\frac{\|\mathbf{Q}_N(\bar{\y}-\bar{\x})\|_{\H}^2}{2\eps}+
			\delta\|\y\|_{\V}^2.
		\end{align*}
		Since $u_\upgamma$ is a viscosity subsolution of \eqref{thjbcomp} over $(0,T)\times\V$, then by the definition of viscosity subsolution with the test function $\varphi$, we have $\hat{\y}^n\in\V_2$ (even $\hat{\x}^n\in\V_2$ also, in the case of viscosity supersolution $v_\upgamma$), and it satisfies    
		{\small {\begin{align}\label{supsoLdef}
				&(\varphi_n)_t(t_n,\hat{\y}^n)+\frac{1}{2n}\mathrm{Tr}(\Q\D^2\varphi_n(t_n,\hat{\y}^n))
					\nonumber\\&\quad+
					\frac{1}{2n}\mathrm{Tr}\bigg(\frac{1}{\eps}\Q\D^2(\hat{\y}^n,\mathbf{Q}_N(\bar{\y}-\bar{\x}))+\frac{1}{\eps}\Q\D^2\|\mathbf{Q}_N(\hat{\y}^n-\bar{\y})\|_{\H}^2-\frac{1}{\eps}\Q\D^2\|\mathbf{Q}_N(\bar{\y}-\bar{\x})\|_{\H}^2+2\delta\Q(\A+\I)\bigg)
					\nonumber\\&\quad-\frac12 
					\bigg\|\Q^{\frac12}\bigg(\D\varphi_n(t_n,\hat{\y}^n)
					+\frac{1}{\eps}\mathbf{Q}_N(\bar{\y}-\bar{\x})+\frac{2}{\eps}\mathbf{Q}_N(\hat{\y}^n-\bar{\y})
					+2\delta(\A+\I)\hat{\y}^n\bigg)\bigg\|_{\H}^2
				     \nonumber\\&\quad-
					\bigg(\mu\A\hat{\y}^n,\D\varphi_n(t_n,\hat{\y}^n)+\frac{1}{\eps}\mathbf{Q}_N(\bar{\y}-\bar{\x})+\frac{2}{\eps}\mathbf{Q}_N(\hat{\y}^n-\bar{\y})
					+2\delta(\A+\I)\hat{\y}^n\bigg)
					\nonumber\\&\quad-
					\bigg(\alpha\hat{\y}^n,\D\varphi_n(t_n,\hat{\y}^n)+\frac{1}{\eps}\mathbf{Q}_N(\bar{\y}-\bar{\x})+\frac{2}{\eps}\mathbf{Q}_N(\hat{\y}^n-\bar{\y})
					+2\delta(\A+\I)\hat{\y}^n\bigg)
					\nonumber\\&\quad-
					\bigg(\mathcal{B}(\hat{\y}^n),\D\varphi_n(t_n,\hat{\y}^n)+\frac{1}{\eps}\mathbf{Q}_N(\bar{\y}-\bar{\x})+\frac{2}{\eps}\mathbf{Q}_N(\hat{\y}^n-\bar{\y})
					+2\delta(\A+\I)\hat{\y}^n\bigg)
					\nonumber\\&\quad-
					\beta\bigg(\mathcal{C}(\hat{\y}^n),\D\varphi_n(t_n,\hat{\y}^n)+\frac{1}{\eps}\mathbf{Q}_N(\bar{\y}-\bar{\x})+\frac{2}{\eps}\mathbf{Q}_N(\hat{\y}^n-\bar{\y})+2\delta
					(\A+\I)\hat{\y}^n\bigg)
					\nonumber\\&\quad+
					\bigg(\f(t_n),\D\varphi_n(t_n,\hat{\y}^n)+\frac{1}{\eps}\mathbf{Q}_N(\bar{\y}-\bar{\x})+\frac{2}{\eps}\mathbf{Q}_N(\hat{\y}^n-\bar{\y})
					+2\delta
					(\A+\I)\hat{\y}^n\bigg)
					\nonumber\\&\geq\frac{\upgamma}{T^2}.
		\end{align}}}
		The following \emph{derivatives} are immediate 
		\begin{align}\label{derQ}
			\D^2(\hat{\y}^n,\mathbf{Q}_N(\bar{\y}-\bar{\x}))=0, \  \D^2\|\mathbf{Q}_N(\hat{\y}^n-\bar{\y})\|_{\H}^2=2\mathbf{Q}_N \ \text{ and } \ \D^2\|\mathbf{Q}_N(\bar{\y}-\bar{\x})\|_{\H}^2=0.
		\end{align}
		On utilizing \eqref{derQ} in \eqref{supsoLdef} and rearranging the terms, we obtain 
		\begin{align}\label{viscdef1}
			&(\varphi_n)_t(t_n,\hat{\y}^n)+ \frac{1}{2n}\mathrm{Tr}\left(\Q\D^2\varphi_n(t_n,\hat{\y}^n)+\frac{2}{\eps}\Q\mathbf{Q}_N+2\delta(\Q+\Q_1)\right)
			\nonumber\\&\quad-\frac12 
			\bigg\|\Q^{\frac12}\bigg(\mathfrak{T}_n
			+2\delta(\A+\I)\hat{\y}^n\bigg)\bigg\|_{\H}^2-
		\bigg(\mu(\A+\I)\hat{\y}^n+\alpha\hat{\y}^n+\mathcal{B}(\hat{\y}^n)+
		\beta\mathcal{C}(\hat{\y}^n)-\f(t_n),\mathfrak{T}_n\bigg)
			\nonumber\\&\quad+\mu(\hat{\y}^n,\mathfrak{T}_n)-2\delta
			\underbrace{\bigg(\mu\A\hat{\y}^n+\alpha\hat{\y}^n+
				\mathcal{B}(\hat{\y}^n)+
				\beta\mathcal{C}(\hat{\y}^n)-\f(t_n),
				(\A+\I)\hat{\y}^n\bigg)}_{\J}
			\nonumber\\&\geq\frac{\upgamma}{T^2},
		\end{align}
		where $$\mathfrak{T}_n:=\D\varphi_n(t_n,\hat{\y}^n)
		+\frac{1}{\eps}\mathbf{Q}_N(\bar{\y}-\bar{\x})+\frac{2}{\eps}\mathbf{Q}_N(\hat{\y}^n-\bar{\y}).$$

	    Let us now estimate $\J$. We only consider the case when $r>3$ in $d\in\{2,3\}$. The case, when $r=3$ with $2\beta\mu\geq1$ and $d\in\{2,3\}$ can be treated in a similar way.
		From \eqref{syymB3}, \eqref{torusequ} and Cauchy-Schwarz inequality, we calculate
		\begin{align}\label{rg3c}
			&\J=-2\mu\|(\A+\I)\hat{\y}^n\|_{\H}^2-2\alpha\|\hat{\y}^n\|_{\V}^2-
			2\beta\|\hat{\y}^n\|_{\wi\L^{r+1}}^{r+1}+2\mu\|\hat{\y}^n\|_{\V}^2
			\nonumber\\&\quad-
			2(\mathcal{B}(\hat{\y}^n),(\A+\I)\hat{\y}^n)- 2\beta(\mathcal{C}(\hat{\y}^n),\A\hat{\y}^n)+
			2(\f(t),(\A+\I)\hat{\y}^n)
			\nonumber\\&\leq-
			\mu\|(\A+\I)\hat{\y}^n\|_{\H}^2-2\alpha\|\hat{\y}^n\|_{\V}^2-
			2\beta\|\hat{\y}^n\|_{\wi\L^{r+1}}^{r+1}+2\mu\|\hat{\y}^n\|_{\V}^2
			\nonumber\\&\quad-\frac{3\beta}{2}
			\||\hat{\y}^n|^{\frac{r-1}{2}}\nabla\hat{\y}^n\|_{\H}^2+ 2\varrho\|\nabla\hat{\y}^n\|_{\H}^2+\|\f\|_{\V}^2
			+\|\hat{\y}^n\|_{\V}^2.
		\end{align}
			On substituting \eqref{rg3c} into \eqref{viscdef1}, we obtain
			\begin{align}\label{viscdef11}
			&(\varphi_n)_t(t_n,\hat{\y}^n)+ \frac{1}{2n}\mathrm{Tr}\left(\Q\D^2\varphi_n(t_n,\hat{\y}^n)+\frac{2}{\eps}\Q\mathbf{Q}_N+2\delta(\Q+\Q_1)\right)
			\nonumber\\&\quad-
			\bigg(\mu(\A+\I)\hat{\y}^n+\alpha\hat{\y}^n+\mathcal{B}(\hat{\y}^n)+
			\beta\mathcal{C}(\hat{\y}^n)-\f(t_n),\mathfrak{T}_n\bigg)
			\nonumber\\&\quad+\mu(\hat{\y}^n,\mathfrak{T}_n)+2\varrho\|\nabla\hat{\y}^n\|_{\H}^2
			+\|\f\|_{\V}^2+(2\mu+1)\|\hat{\y}^n\|_{\V}^2
			\nonumber\\&\geq\frac{\upgamma}{T^2}+\mu\delta\|(\A+\I)\hat{\y}^n\|_{\H}^2+2\alpha\delta\|\hat{\y}^n\|_{\V}^2+
			2\beta\delta\|\hat{\y}^n\|_{\wi\L^{r+1}}^{r+1}+\frac{3\beta\delta}{2}
			\||\hat{\y}^n|^{\frac{r-1}{2}}\nabla\hat{\y}^n\|_{\H}^2
			\nonumber\\&\quad+
			\frac12 
			\bigg\|\Q^{\frac12}\bigg(\mathfrak{T}_n
			+2\delta(\A+\I)\hat{\y}^n\bigg)\bigg\|_{\H}^2.
			\end{align}
			By using the Cauchy-Schwarz and Young's inequalities, we estimate the following:
		\begin{align}
			|(\mathcal{C}(\hat{\y}^n),\mathfrak{T}_n)|&\leq C\left(\||\hat{\y}^n|^{\frac{r-1}{2}}\nabla\hat{\y}^n\|_{\H}^{\frac{2r}{r+1}}+\|\hat{\y}^n\|_{\wi\L^{r+1}}^r\right)
			\|\mathfrak{T}_n\|_{\H}\nonumber\\&\leq
			\frac{\delta}{4}\||\hat{\y}^n|^{\frac{r-1}{2}}
			\nabla\hat{\y}^n\|_{\H}^2+
			C\|\mathfrak{T}_n\|_{\H}^{r+1}+
			\delta\|\hat{\y}^n\|_{\wi\L^{r+1}}^{r+1},
			\label{appdac1}\\
			|(\mathcal{B}(\hat{\y}^n),\mathfrak{T}_n)|&\leq \frac12\|\mathfrak{T}_n\|_{\H}^2
			+\frac{\beta\delta}{4} \||\hat{\y}^n|^{\frac{r-1}{2}}\nabla\hat{\y}^n\|_{\H}^2
			+\widetilde{\varrho}\|\nabla\hat{\y}^n\|_{\H}^2,\label{appdac2}\\
			|((\A+\I)\hat{\y}^n,\mathfrak{T}_n)|&\leq \frac{\delta}{4}\|(\A+\I)\hat{\y}^n\|_{\H}^2
			+C\|\mathfrak{T}_n\|_{\H}^2,\label{appdac3}\\
			|(\f(t),\mathfrak{T}_n)|&\leq \frac{\delta}{4}\|(\A+\I)\hat{\y}^n\|_{\H}^2+C\|\mathfrak{T}_n\|_{\H}^2\label{appdac4},\\
			|(\hat{\y}^n,\mathfrak{T}_n)|&\leq \frac{1}{2}\|\hat{\y}^n\|_{\H}^2+\frac{1}{2}\|\mathfrak{T}_n\|_{\H}^2,\label{appdac5}
		\end{align}
		where $\widetilde{\varrho}=\frac{r-3}{\delta(r-1)}\left[\frac{4}{\beta\delta (r-1)}\right]^{\frac{2}{r-3}}$. On combining \eqref{appdac1}-\eqref{appdac5} and substituting  into \eqref{viscdef11}, we obtain
			\begin{align}\label{viscdef111}
			&(\varphi_n)_t(t_n,\hat{\y}^n)+ \frac{1}{2n}\mathrm{Tr}\left(\Q\D^2\varphi_n(t_n,\hat{\y}^n)+\frac{2}{\eps}\Q\mathbf{Q}_N+2\delta(\Q+\Q_1)\right)
			\nonumber\\&\quad
			+ \mathpzc{C}\|\hat{\y}^n\|_{\V}^2+ \mathpzc{C}\|\mathfrak{T}_n\|_{\H}^{r+1}+
			 \mathpzc{C}\|\mathfrak{T}_n\|_{\H}^2
			\nonumber\\&\geq\frac{\upgamma}{T^2}+
			\frac{\mu\delta}{2}\|(\A+\I)\hat{\y}^n\|_{\H}^2
			+2\alpha\delta\|\hat{\y}^n\|_{\V}^2+
			\beta\delta\||\hat{\y}^n|^{\frac{r-1}{2}}\nabla\hat{\y}^n\|_{\H}^2+
			\beta\delta\|\hat{\y}^n\|_{\wi\L^{r+1}}^{r+1}
			\nonumber\\&\quad+\frac12 
			\bigg\|\Q^{\frac12}\bigg(\mathfrak{T}_n
			+2\delta(\A+\I)\hat{\y}^n\bigg)\bigg\|_{\H}^2.
		\end{align}
		By utilizing the fact $\mathbf{Q}_N=\I-\mathbf{P}_N$ together with the convergences \eqref{limits1}-\eqref{limits11} and \eqref{eqn-v-norm-conv}, we conclude the following bound:
		\begin{align}\label{ot1}
			\|\mathfrak{T}_n\|_{\H}&\leq
%			\bigg\|\D\varphi_n(t_n,\hat{\y}^n)
%			+\frac{1}{\eps}\mathbf{Q}_N(\bar{\y}-\bar{\x})+\frac{2}{\eps}\mathbf{Q}_N(\hat{\y}^n-\bar{\y})\bigg\|_{\H}
%			\nonumber\\&=
%			\bigg\|\D\varphi_n(t_n,\hat{\y}^n)-\frac{1}{\eps}\mathbf{P}_N(\bar{\y}-\bar{\x})+\frac{1}{\eps}(\bar{\y}-\bar{\x})+\frac{2}{\eps}\mathbf{Q}_N(\hat{\y}^n-\bar{\y})\bigg\|_{\H}
%			\nonumber\\&\leq 
			\bigg\|\D\varphi_n(t_n,\hat{\y}^n)-\frac{1}{\eps}\mathbf{P}_N(\bar{\y}-\bar{\x})\bigg\|_{\H}+\bigg\|\frac{1}{\eps}(\bar{\y}-\bar{\x})+\frac{2}{\eps}\mathbf{Q}_N(\hat{\y}^n-\bar{\y})\bigg\|_{\H}
			\nonumber\\&\leq  \mathpzc{C}(\eps,\|\bar{\y}\|_{\V},\|\bar{\x}\|_{\V}).
		\end{align} 
		Finally, from the convergences \eqref{limits1}-\eqref{limits11}, \eqref{eqn-v-norm-conv} and \eqref{ot1}, the left hand side of \eqref{viscdef111} is bounded independent of $n$. Therefore, on employing the Banach-Alaoglu theorem and the fact that weak limits are unique, we have 
		\begin{align}\label{weakconv1}
			\A\hat{\y}^n\rightharpoonup\A\bar{\y} \  \text{ in } \ \mathcal{D}(\A) \ \text{ as } \ n\to+\infty.
		\end{align}

		On taking the limit supremum as $n\to+\infty$ in \eqref{viscdef1}, we deduce 
			\begin{align}\label{viscdef2}
			&\frac{\bar{t}-\bar{s}}{\eta}+
			\frac{1}{2n}\mathrm{Tr}\left(\Q Y_N+\frac{2}{\eps} \Q\mathbf{Q}_N+2\delta(\Q+\Q_1)\right)
			\nonumber\\&-\frac12 
			\bigg\|\Q^{\frac12}\bigg(\frac{1}{\eps}(\bar{\y}-\bar{\x})
			+2\delta(\A+\I)\bar{\y}\bigg)\bigg\|_{\H}^2-
			\bigg(\mu\A\bar{\y}+\alpha\bar{\y}+\mathcal{B}(\bar{\y})+
			\beta\mathcal{C}(\bar{\y}),\frac{1}{\eps}(\bar{\y}-\bar{\x})\bigg)
			\nonumber\\&+\bigg(\f(\bar{t}),\frac{1}{\eps}(\bar{\y}-\bar{\x})
			+2\delta(\A+\I)\bar{\y}\bigg)\geq\frac{\upgamma}{T^2}.
		\end{align}
	See \cite[Theorem 5.11]{smtm1} for the detailed procedure of how the limit supremum is taken in the above step as $n\to+\infty$. On employing a similar method as above for the supersolution $v_\upgamma$, we arrive at
			\begin{align}\label{viscdef3}
			&\frac{\bar{t}-\bar{s}}{\eta}+
			\frac{1}{2n}\mathrm{Tr}\left(\Q X_N-\frac{2}{\eps} \Q\mathbf{Q}_N-2\delta(\Q+\Q_1)\right)
			\nonumber\\&-\frac12 
			\bigg\|\Q^{\frac12}\bigg(\frac{1}{\eps}(\bar{\y}-\bar{\x})
			-2\delta(\A+\I)\bar{\x}\bigg)\bigg\|_{\H}^2-
			\bigg(\mu\A\bar{\x}+\alpha\bar{\x}+\mathcal{B}(\bar{\x})+
			\beta\mathcal{C}(\bar{\x}),\frac{1}{\eps}(\bar{\y}-\bar{\x})\bigg)
			\nonumber\\&+\bigg(\f(\bar{s}),\frac{1}{\eps}(\bar{\y}-\bar{\x})
			-2\delta(\A+\I)\bar{\x}\bigg)\leq-\frac{\upgamma}{T^2}.
		\end{align}
		The convergence  \eqref{copm2} and the fact that $\Q^{\frac12}\A^{\frac12}$ is bounded on $\H$, yield the following estimate:
		\begin{align}\label{copm23}
		&2\delta\|\Q^{\frac12}(\A+\I)\bar{\y}\|_{\H}+
		2\delta\|\Q^{\frac12}(\A+\I)\bar{\x}\|_{\H}
		\nonumber\\&\leq 2\delta
		\|\Q^{\frac12}(\A+\I)^{\frac12} \big((\A+\I)^{\frac12}\bar{\y}\big)\|_{\H}
		+2\delta
		\|\Q^{\frac12}(\A+\I)^{\frac12} \big((\A+\I)^{\frac12}\bar{\x}\big)\|_{\H}
		\nonumber\\&\leq
		\mathpzc{D}_1\delta(\|\bar{\y}\|_{\V}+\|\bar{\x}\|_{\V})
		\nonumber\\&\leq
		\mathpzc{D}_1\sqrt{\delta},
		\end{align}
	where $\mathpzc{D}_1$ is some positive constant. By using \eqref{copm23} and \eqref{copm3}, we estimate following:
		\begin{align}\label{copm2233}
		&\left|\frac12\bigg\|\Q^{\frac12}\bigg(\frac{1}{\eps}(\bar{\y}-\bar{\x})
		+2\delta(\A+\I)\bar{\y}\bigg)\bigg\|_{\H}^2-
		\frac12\bigg\|\Q^{\frac12}\bigg(\frac{1}{\eps}(\bar{\y}-\bar{\x})
		-2\delta(\A+\I)\bar{\x}\bigg)\bigg\|_{\H}^2\right|
		\nonumber\\&=
		\frac{2\delta}{\eps}\left(\Q^{\frac12}(\bar{\y}-\bar{\x}), \Q^{\frac12}(\A+\I)(\bar{\y}+\bar{\x})\right)+
		2\delta^2\big(\|\Q^{\frac12}(\A+\I)\bar{\y}\|_{\H}^2-
		\|\Q^{\frac12}(\A+\I)\bar{\x}\|_{\H}^2\big)
		\nonumber\\&=
		\frac{2\delta}{\eps}\left(\Q^{\frac12}(\bar{\y}-\bar{\x}), \Q^{\frac12}(\A+\I)(\bar{\y}+\bar{\x})\right)+
		2\delta^2\big(\Q^{\frac12}(\A+\I)(\bar{\y}-\bar{\x}),
		\Q^{\frac12}(\A+\I)(\bar{\y}+\bar{\x})\big)
		\nonumber\\&\leq
		\frac{2\delta}{\eps}\|\Q^{\frac12}(\bar{\y}-\bar{\x})\|_{\H}
		\|\Q^{\frac12}(\A+\I)(\bar{\y}+\bar{\x})\|_{\H}+
		2\delta^2
		\|\Q^{\frac12}(\A+\I)(\bar{\y}-\bar{\x})\|_{\H}
		\|\Q^{\frac12}(\A+\I)(\bar{\y}+\bar{\x})\|_{\H}
		\nonumber\\&\leq
		\mathpzc{D}_2\sqrt{\delta},
	   \end{align}
		where $\mathpzc{D}_2$ is some positive constant. Moreover, from Hypothesis \ref{fhyp} and \eqref{copm2}, we find
		\begin{align}\label{fdeca}
		&2\delta\big|\big(\f(\bar{t}),(\A+\I)\bar{\y}\big)\big|+2\delta
		\big|\big(\f(\bar{s}),(\A+\I)\bar{\x}\big)\big|
		\nonumber\\&=
		2\delta\big|\big((\A+\I)^{\frac12}\f(\bar{t}),(\A+\I)^{\frac12}\bar{\y}
		\big)\big|+2\delta
		\big|\big((\A+\I)^{\frac12}\f(\bar{s}),(\A+\I)^{\frac12}\bar{\x}\big)\big|
		\nonumber\\&\leq
		2\delta\|(\A+\I)^{\frac12}\f(\bar{t})\|_{\H}\|\bar{\y}\|_{\V}+
		2\delta\|(\A+\I)^{\frac12}\f(\bar{s})\|_{\H}\|\bar{\x}\|_{\V}
		\nonumber\\&\leq
		\mathpzc{D}_3\sqrt{\delta},
		\end{align}
for some positive constant $\mathpzc{D}_3$. 
On combining \eqref{viscdef2} with \eqref{viscdef3} and utilizing the estimates \eqref{copm23}-\eqref{fdeca} and the fact that $Y_N\leq X_N$, we obtain
\begin{align}\label{viscdef4}
&-\frac{1}{n}\mathrm{Tr}\left(\frac{2}{\eps} \Q\mathbf{Q}_N+2\delta(\Q+\Q_1)\right)+
\frac{\alpha}{\eps}\|\bar{\y}-\bar{\x}\|_{\H}^2
\nonumber\\&\quad+
\bigg(\mu\A\bar{\y}+\mathcal{B}(\bar{\y})+
\beta\mathcal{C}(\bar{\y})-\mu\A\bar{\x}-\mathcal{B}(\bar{\x})-\beta\mathcal{C}(\bar{\x}),\frac{1}{\eps}(\bar{\y}-\bar{\x})\bigg)
\nonumber\\&\quad+
\bigg(\f(\bar{s})-\f(\bar{t}),\frac{1}{\eps}(\bar{\y}-\bar{\x})\bigg)-
\mathpzc{D}_4\sqrt{\delta}\leq-\frac{2\upgamma}{T^2},
\end{align}
where $\mathpzc{D}_4$ is some positive constant. Since, $\f:[0,T]\to\V$ is bounded and continuous, therefore using \eqref{copm1} and \eqref{copm3}, we find
\begin{align}\label{fdeca1}
\left|\bigg(\f(\bar{t})-\f(\bar{s}),\frac{1}{\eps}(\bar{\y}-\bar{\x})\bigg)\right|
&\leq
\frac{1}{\eps}\|\f(\bar{t})-\f(\bar{s})\|_{\H}\|\bar{\y}-\bar{\x}\|_{\H}
\nonumber\\&\leq
\frac{1}{\eps}\mathpzc{w}_{\f}(|\bar{t}-\bar{s}|)\|\bar{\y}-\bar{\x}\|_{\H}
\nonumber\\&\leq
\mathfrak{L}\mathpzc{w}_{\f}(o(\sqrt{\eta}))\to0 \ \text{ as } \ \eta\to0,
\end{align}
where $\mathpzc{w}_{\f}$ is the modulus of continuity of $\f$. Further, using the properties of the projection operator $\Q_N$ and the fact that $\Tr(\Q)<+\infty$, one can deduce that 
 \begin{align}\label{fdeca2}
 	\Tr(\Q\mathbf{Q}_N)\to0 \ \text{ as } \ n\to+\infty.
 \end{align}
 	Then \eqref{viscdef4} together with \eqref{fdeca1}-\eqref{fdeca2} and the monotonicity estimate \eqref{monoest1} yield
\begin{align}\label{viscdef5}
-\frac{2\delta}{n}\mathrm{Tr}(\Q+\Q_1)+
\frac{\alpha}{\eps}\|\bar{\y}-\bar{\x}\|_{\H}^2
-\frac{\varrho}{\eps}\|\bar{\y}-\bar{\x}\|_{\H}^2-
\mathpzc{D}_4\sqrt{\delta}-\mathpzc{w}_1(N)-\mathpzc{w}_2(\eta)
\leq-
\frac{2\upgamma}{T^2},
\end{align}
where $\mathpzc{w}_1(N)\to0$ as $n\to+\infty$ for $\eps, \delta, \eta$ fixed, and $\mathpzc{w}_2(\eta)\to0$ as $\eta\to0$ for $\eps,\delta$ fixed. 
Finally, in \eqref{viscdef5}  we first let $n\to+\infty$. Then, for fixed $\delta>0$ and $\eps>0$, we take the $\liminf$ as $\eta\to0$. Next, for fixed $\eps>0$, we take the $\liminf$ as $\delta\to0$, and finally we pass to the $\liminf$ as $\eps\to0$. This yields $\upgamma<0$, which is a contradiction. This completes the proof.
\end{proof}

\begin{remark}
	1.) Note that in \eqref{viscdef4}, we employ the monotonicity estimate for the operator $\mu\A+\mathcal{B}(\cdot)+\beta\mathcal{C}(\cdot)$ when $r>3$ (see Lemma \ref{monoest}). It constitutes the most crucial step of our analysis and highlights a key difference from \cite{AS2}.
	For the Navier-Stokes operator $\mu \mathcal{A} + \mathcal{B}(\cdot)$, a direct monotonicity estimate cannot be established due to the presence of the convective term alone. To overcome this difficulty, the authors in \cite{AS2} employed a quantization technique to derive suitable monotonicity estimates.

    2.) In contrast, in our setting, the presence of the absorption term, together with the diffusion term, allow us to control the convective term for $r>3$, as well as $r=3$ under the condition $2\beta\mu\geq1$ in both $d=2,3$. However, for $r<3$ and $r=3$ with $2\beta\mu<1$, the same difficulty persists as in the Navier-Stokes case.
\end{remark}

	\begin{theorem}\label{comparisonappr}
	Assume that Hypotheses \ref{trQ1} and \ref{fhyp} hold and let $g\in\mathrm{Lip}_b(\H)$. Let $u_n$ be a bounded viscosity solution of \eqref{thjbcomp} for $n<+\infty$ and $v$ be a viscosity solution of \eqref{thjbcomp} for $n=+\infty$ such that
	\begin{align}\label{bdd12}
		\lim\limits_{t\to T} \{\big|u_n(t,\y)-g(\y)\big|+\big|v(t,\y)-g(\y)\big|\}
		=0, \ \text{ uniformly on bounded sets of } \ \V,
	\end{align}
	and 
	\begin{align}\label{bddL}
		|v(t,\y)-v(t,\x)|\leq\mathfrak{L}\|\y-\x\|_{\H}, 
	\end{align}
for some $\mathfrak{L}\geq0$ and for all $t\in(0,T]$ and $\y,\x\in\V$. 
Then, there exits a constant $\mathpzc{C}$ independent of $n$ such that
\begin{align}\label{supestmain} 
	\|u_n-v\|_{\infty}\leq\frac{\mathpzc{C}}{\sqrt{n}}.
\end{align}
\end{theorem}

\begin{proof}
	Let $\mathpzc{C}_1>0$ be a constant such that
	\begin{align}\label{trc1}
		\frac12\Tr(\Q)\leq\mathpzc{C}_1.
	\end{align}
	We set 
	\begin{align}\label{vnvndef}
		v_n:=v+\frac{1}{\sqrt{n}}(T-t)\mathpzc{C}_1.
	\end{align}
	
	Then $v_n$ is a viscosity supersolution of 
	\begin{align}\label{viscsupvn}
		(v_n)_t-\frac12\|\Q^{\frac12}\D v_n\|_{\H}^2 + (-\mu\A\y-\mathcal{B}(\y)-\alpha\y-\beta\mathcal{C}(\y)+\f(t),\D v_n)=
	    -\frac{\mathpzc{C}_1}{\sqrt{n}}.
	\end{align}
To proceed further, we follow a similar approach as we performed  in the proof of Theorem \ref{comparison}. Let us fix $\eps=\frac{1}{\sqrt{n}}$. Now, assume  that $u_n\not\leq v_n$. Then, following the same reasoning as in the proof of Theorem \ref{comparison},  for sufficiently small parameters  $\upgamma>0,\delta>0,\eta>0$, the function 
\begin{align}\label{tfunct}
	(t,s,\y,\x)\mapsto u_{n,\upgamma}(t,\y)-v_{n,\upgamma}(s,\x)-
	\frac{\sqrt{n}}{2}\|\y-\x\|_{\H}^2-\frac{(t-s)^2}{2\eta}-
	\delta(\|\y\|_{\V}^2+\|\x\|_{\V}^2),
\end{align}
attains a strict global maxima over $(0,T]\times(0,T]\times\H\times\H$ at some points $(\bar{t},\bar{s},\bar{\y},\bar{\x})\in(0,T]\times(0,T]\times\V\times\V$, where $0<\bar{t},\bar{s}<T$. Here $u_{n,\upgamma}$ and $u_{n,\upgamma}$ are defined analogously to \eqref{ugvg}. Moreover, the limits \eqref{copm1}-\eqref{copm3} are satisfied with $\eps=\frac{1}{\sqrt{n}}$.
Then from \eqref{tfunct}, 
 the function 
\begin{align*}
	(t,\y)\mapsto u_{n,\upgamma}(t,\y)-v_{n,\upgamma}(\bar{s},\bar{\x})-
	\frac{\sqrt{n}}{2}\|\y-\bar{\x}\|_{\H}^2-\frac{(t-\bar{s})^2}{2\eta}-
	\delta(\|\y\|_{\V}^2+\|\bar{\x}\|_{\V}^2)
\end{align*} 
has a strict global maximum at $(\bar{t},\bar{\y})$ in $(0,T)\times\H$. Since, $u_n$ is the viscosity subsolution of \eqref{thjbcomp}, we obtain 
\begin{align}\label{viscdef6}
	&-\frac{\upgamma}{\bar{t}^2}+\frac{\bar{t}-\bar{s}}{\eta}+
	\frac{1}{2n}\mathrm{Tr}\left(\sqrt{n}\Q+2\delta(\Q+\Q_1)\right)
	\nonumber\\&-\frac12 
	\|\Q^{\frac12}\big(\sqrt{n}(\bar{\y}-\bar{\x})
	+2\delta(\A+\I)\bar{\y}\big)\|_{\H}^2+\big(\f(\bar{t}),\sqrt{n}
	(\bar{\y}-\bar{\x})+2\delta(\A+\I)\bar{\y}\big)
	\nonumber\\&+
	\big(-\mu\A\bar{\y}-\alpha\bar{\y}-\mathcal{B}(\bar{\y})-
	\beta\mathcal{C}(\bar{\y}),\sqrt{n}(\bar{\y}-\bar{\x})
	+2\delta(\A+\I)\bar{\y}\big)\geq0.
\end{align}
Now, in view of \eqref{syymB3} and \eqref{torusequ}, we find
\begin{align}\label{DA1}
	&\big(-\mu\A\bar{\y}-\alpha\bar{\y}-\mathcal{B}(\bar{\y})-
	\beta\mathcal{C}(\bar{\y}),(\A+\I)\bar{\y}\big)
	\nonumber\\&=
	-\mu\|\A\bar{\y}\|_{\H}^2-\alpha\|\bar{\y}\|_{\V}^2-
	(\mathcal{B}(\bar{\y}),\A\bar{\y})-
	\beta(\mathcal{C}(\bar{\y}),\A\bar{\y}) -\beta\|\bar{\y}\|_{\wi\L^{r+1}}^{r+1}
	\nonumber\\&\leq-
	\frac{\mu}{2}\|\A\bar{\y}\|_{\H}^2-\alpha\|\bar{\y}\|_{\V}^2-
	\frac{3\beta}{4} \||\bar{\y}|^{\frac{r-1}{2}}\nabla\bar{\y}\|_{\H}^2+
	\varrho\|\nabla\bar{\y}\|_{\H}^2-\beta\|\bar{\y}\|_{\wi\L^{r+1}}^{r+1}.
\end{align}
Substituting \eqref{DA1} into \eqref{viscdef6}, we obtain 
\begin{align}\label{viscdef7}
	&-\frac{\upgamma}{\bar{t}^2}+\frac{\bar{t}-\bar{s}}{\eta}+
    \frac{1}{2n}\mathrm{Tr}\left(\sqrt{n}\Q+2\delta(\Q+\Q_1)
	\right)
\nonumber\\&-\frac12 
\|\Q^{\frac12}\big(\sqrt{n}(\bar{\y}-\bar{\x})
+2\delta(\A+\I)\bar{\y}\big)\|_{\H}^2+\big(\f(\bar{t}),\sqrt{n}
(\bar{\y}-\bar{\x})+2\delta(\A+\I)\bar{\y}\big)
\nonumber\\&+
2\varrho\delta(\|\nabla\bar{\y}\|_{\H}^2+\|\nabla\bar{\x}\|_{\H}^2)+
\big(-\mu\A\bar{\y}-\alpha\bar{\y}-\mathcal{B}(\bar{\y})-
\beta\mathcal{C}(\bar{\y}),\sqrt{n}(\bar{\y}-\bar{\x})\big)
\nonumber\\&\geq
\mu\delta\|\A\bar{\y}\|_{\H}^2+2\alpha\delta\|\bar{\y}\|_{\V}^2+
\frac{3\beta\delta}{2} \||\bar{\y}|^{\frac{r-1}{2}}\nabla\bar{\y}\|_{\H}^2+ 2\beta\delta\|\bar{\y}\|_{\wi\L^{r+1}}^{r+1},
\end{align}
where we have used the fact that $\|\nabla\bar{\y}\|_{\H}^2\leq \|\nabla\bar{\y}\|_{\H}^2+\|\nabla\bar{\x}\|_{\H}^2$. Moreover, from \eqref{tfunct}, we infer that the function
\begin{align*}
	(s,\x)\mapsto v_{n,\upgamma}(s,\x)-u_{n,\upgamma}(\bar{t},\bar{\y})+
	\frac{\sqrt{n}}{2}\|\bar{\y}-\x\|_{\H}^2+\frac{(\bar{t}-s)^2}{2\eta}+
	\delta(\|\bar{\y}\|_{\V}^2+\|\x\|_{\V}^2),
\end{align*} 
has a strict global minimum at $(\bar{s},\bar{\x})$ in $(0,T)\times\H$. Therefore, by using the definition of 
viscosity supersolution  $v_n$ of \eqref{viscsupvn} and performing similar calculations as in \eqref{viscdef7}, we find 
\begin{align}\label{viscdef8}
	&\frac{\upgamma}{\bar{s}^2}+\frac{\bar{t}-\bar{s}}{\eta}
  -\frac12 \|\Q^{\frac12}\big(\sqrt{n}(\bar{\y}-\bar{\x})
	-2\delta(\A+\I)\bar{\x}\big)\|_{\H}^2
	\nonumber\\&\quad
	+\big(\f(\bar{s}),\sqrt{n}(\bar{\y}-\bar{\x})
	-2\delta(\A+\I)\bar{\x}\big)
	\nonumber\\&\quad
	+\big(-\mu\A\bar{\x}-\alpha\bar{\x}-\mathcal{B}(\bar{\x})-
	\beta\mathcal{C}(\bar{\x}),\sqrt{n}(\bar{\y}-\bar{\x})\big)
	\nonumber\\&\quad+
	\mu\delta\|\A\bar{\x}\|_{\H}^2+2\alpha\delta\|\bar{\x}\|_{\V}^2+
	\frac{3\beta\delta}{2} \||\bar{\x}|^{\frac{r-1}{2}}\nabla\bar{\x}\|_{\H}^2+ 2\beta\delta\|\bar{\x}\|_{\wi\L^{r+1}}^{r+1}
	\nonumber\\&
	\leq2\varrho\delta(\|\nabla\bar{\y}\|_{\H}^2+\|\nabla\bar{\x}\|_{\H}^2)-
	\frac{\mathpzc{C}_1}{\sqrt{n}}.
\end{align}
On combining \eqref{viscdef7} with \eqref{viscdef8} and utilizing \eqref{copm2233} with $\eps=\frac{1}{\sqrt{n}}$, \eqref{fdeca}, \eqref{fdeca1}, \eqref{trc1} and monotonicity estimate \eqref{monoest1}, we arrive at 
\begin{align}\label{supest} 
&\frac{2\upgamma}{T^2}
+\sqrt{n}\|\bar{\y}-\bar{\x}\|_{\H}^2+\varrho\|\bar{\y}-\bar{\x}\|_{\H}^2
\nonumber\\&
\leq4\varrho\delta(\|\nabla\bar{\y}\|_{\H}^2+\|\nabla\bar{\x}\|_{\H}^2)
+\frac{\delta}{n}\mathrm{Tr}(\Q+\Q_1)
+\mathpzc{D}_4\sqrt{\delta}+\mathpzc{w}_{\f}(\eta),
\end{align}
where $\mathpzc{w}_{\f}$ is the modulus of continuity of $\f$. In \eqref{supest}, we first take $\limsup$ as $\eta\to0$ for a fixed $\delta>0$, and then we pass to $\limsup$ as $\delta\to0$ along with \eqref{copm2}. It lead us to a contradiction  $\upgamma<0$. Thus, we obtain
\begin{align}\label{supest1}
	u_n\leq v+\frac{1}{\sqrt{n}}(T-t)\mathpzc{C}_1.
\end{align}
On applying similar procedure, as above, to functions $v_n=v-\frac{1}{\sqrt{n}}(T-t)\mathpzc{C}_1$ and $u_n$, we find
	\begin{align}\label{supest2}
			 v-\frac{1}{\sqrt{n}}(T-t)\mathpzc{C}_1\leq u_n.
	\end{align}
Together with \eqref{supest1} and \eqref{supest2}, we obtain \eqref{supestmain}.
\end{proof}

\begin{remark}
	1.) 
	Our choice of $v_n$ in \eqref{vnvndef} differs slightly from the one used in the work \cite{AS2}. Specifically, to obtain a contradiction in step \eqref{supest}, the authors in \cite{AS2} consider the following definition
	\begin{align*}
		v_n=v+\frac{1}{\sqrt{n}}(T-t)(4\mathpzc{C}(\mu,r)\mathfrak{L}^2+\mathpzc{C}_1)+\frac{2\mathfrak{L}^2}{\sqrt{n}}.
	\end{align*}
%	The extra constant $4\mathpzc{C}(\mu,r)\mathfrak{L}^2+\mathpzc{C}_1$  and $\frac{2\mathfrak{L}^2}{\sqrt{n}}$ instead of $\mathpzc{C}_1$ in the definition of $v_n$. 
	The additional constants, that is,  $4\mathpzc{C}(\mu,r)\mathfrak{L}^2+\mathpzc{C}_1$ and $\frac{2\mathfrak{L}^2}{\sqrt{n}}$, were required there because of the monotonicity estimate associated with the Navier-Stokes operator $\mu\A+\mathcal{B}^q(\cdot)$, where $\mathcal{B}^q(\cdot)$ is the quantised bilinear operator (see  \eqref{eqn-bq} and \cite{VBSS} for the properties of the quantised $\mathcal{B}^q(\cdot)$). In our setting, however, for $r>3$ and $r=3$ with $2\beta\mu\geq1$ in $d=2,3$, the nonlinear operator $\mathcal{C}(\cdot)$ plays a crucial role. Its inherent monotonicity (see Lemma \ref{monopropC} and \ref{monoest}) allows us to bypass the introduction of the quantised bilinear operator, leading to a simpler formulation of $v_n$.
	
	2.)
	Let us make some comments before proceeding. Since we are working on $\mathbb{T}^d$ with $d\in\{2,3\}$, so in \eqref{viscdef6} and \eqref{viscdef7}, we cannot utilize the property $(\mathcal{B}(\y),\A\y)=0$, which is true only in $\mathbb{T}^2$. Moreover, unlike in the case of NSE (as done in \cite{AS2}), we have to deal with the additional term $\beta\big(\mathcal{C}(\y),(\A+\I)\y\big)$.
	However, on $\mathbb{T}^d$, as mentioned in the Subsections \ref{advtg1}-\ref{advtg2}, we have the advantage of regularity estimate \eqref{torusequ}, which is crucial to handle this additional term. Moreover, the computation in \eqref{syymB3}, combined with the identity \eqref{torusequ}, is also useful for handling the term  $(\mathcal{B}(\y),\A\y)$ on $\mathbb{T}^3$. In this way, we can handle both the terms, $(\mathcal{B}(\y),\A\y)$ and $(\mathcal{C}(\y),\A\y)$ in both two and three-dimensions. Even in $\mathbb{T}^2$ also, this approach is helpful as we are not using the property $(\mathcal{B}(\y),\A\y)=0$.
\end{remark}

\section{Existence of a Laplace limit at a single time} \setcounter{equation}{0}\label{LaplaceLDP}
In this section, we aim to establish the existence of the Laplace limit for the processes $\Y_n(\cdot)$ at a single time, that is, for the family of processes $\{\Y_n(T)\}_{n\geq1}$, with values in $\H$. For the rest of the work, we assume that $r\in(3,5)$ when $d=3$, $r\in(3,+\infty)$ when $d=2$, and $r=3$ with $2\beta\mu\geq1$ in both $d=2$ and $d=3$.
%      Particularly, our goal is to show that \eqref{LapLacefun} converge as $n\to+\infty$.
\begin{proposition}\label{LapunLh}
	Assume that Hypotheses \ref{trQ1} and  \ref{fhyp} hold and let $g\in\mathrm{Lip}_b(\H)$. Consider the function
	\begin{align}\label{LapLacefun}
		\mathpzc{U}_n(t,\y)=-\frac{1}{n}\log\E[e^{-ng(\Y_n(T))}].
	\end{align}
	Then, $\mathpzc{U}_n$ is uniformly bounded in $n$. Moreover, there exist a constant $\mathpzc{C}_1$ and, for every $\widetilde{R}>0$, a constant $\mathpzc{C}_2=\mathpzc{C}_2(\widetilde{R})$ such that
	\begin{align}\label{LapLacefun1}
		|\mathpzc{U}_n(t,\y)-\mathpzc{U}_n(s,\z)|
		\leq\mathpzc{C}_1\|\y-\z\|_{\H}+\mathpzc{C}_2
		\big(\max\{\|\y\|_{\V},\|\z\|_{\V}\}\big)|t-s|^{\frac12}.
	\end{align}
\end{proposition}

\begin{proof}
	The uniform boundedness of $\mathpzc{U}_n$ is immediate from the fact that $g\in\mathrm{Lip}_b(\H)$. It remains to prove \eqref{LapLacefun1}. For the brevity of notations, we write $\Y_n^{t,\y}(\cdot)=\Y_n(\cdot;t,\y), \Y_n^{t,\z}(\cdot)=\Y_n(\cdot;t,\z), \Y_n^{s,\y}(\cdot)=\Y_n(\cdot;s,\y)$ and $\Y_n^{s,\z}(\cdot)=\Y_n(\cdot;s,\z)$. Since the expectation of the exponential of a random variable is positive, by the mean value theorem for logarithm function, we calculate  
	\begin{align}\label{udfyz}
		|\mathpzc{U}_n(t,\y)-\mathpzc{U}_n(t,\z)|&=
		\bigg|\frac{1}{n}\log\E[e^{-ng(\Y_n^{t,\y}(T))}]-\frac{1}{n}
		\log\E[e^{-ng(\Y_n^{t,\z}(T))}]\bigg|
		\nonumber\\&\leq
		\frac{1}{n}\frac{|\E[e^{-ng(\Y_n^{t,\y}(T))}-e^{-ng(\Y_n^{t,\z}(T))}]|}
		{\min\{\E[e^{-ng(\Y_n^{t,\y}(T))}],\E[e^{-ng(\Y_n^{t,\z}(T))}]\}}.
	\end{align}
	Since $g\in\mathrm{Lip}_b(\H)\subset \mathrm{C}_b(\H)$, we have the following bound:
	\begin{align*}
		e^{-n\|g\|_{\infty}}\leq e^{-ng(\Y_n^{t,\y}(\cdot))}, e^{-ng(\Y_n^{t,\z}(\cdot))}\leq e^{n\|g\|_{\infty}},
	\end{align*}
	which gives 
	\begin{align}\label{Lowbde}
		\min\{\E[e^{-ng(\Y_n^{t,\y}(\cdot))}],\E[e^{-ng(\Y_n^{t,\z}(\cdot))}]\}
		\geq 	e^{-n\|g\|_{\infty}}.
	\end{align}
	Further, by an application of the mean value theorem to the function $e^{-ng(\Y_n(\cdot))}$, we obtain
	\begin{align}\label{Lipgbd}
		|e^{-ng(\Y_n^{t,\y}(\cdot))}-e^{-ng(\Y_n^{t,\z}(\cdot))}|
		&\leq
		ne^{n\|g\|_{\infty}}|g(\Y_n^{t,\y}(\cdot))-g(\Y_n^{t,\z}(\cdot))|
		\nonumber\\&\leq 
		ne^{n\|g\|_{\infty}}\mathrm{Lip}(g) \|\Y_n^{t,\y}(\cdot)-\Y_n^{t,\z}(\cdot)\|_{\H}.
	\end{align}
	Plugging \eqref{Lowbde}-\eqref{Lipgbd} into \eqref{udfyz}, we obtain
	\begin{align}\label{udfyz1}
		|\mathpzc{U}_n(t,\y)-\mathpzc{U}_n(t,\z)|
		&\leq
		e^{2n\|g\|_{\infty}}\mathrm{Lip}(g)
		\E\big[\|\Y_n^{t,\y}(T)-\Y_n^{t,\z}(T)\|_{\H}\big]
		\nonumber\\&\leq
		e^{2n\|g\|_{\infty}}\mathrm{Lip}(g)\|\y-\z\|_{\H}.
	\end{align}
	Repeating the similar argument and using \eqref{ctsdep8}, we write
	\begin{align}\label{udfyz2}
		|\mathpzc{U}_n(t,\y)-\mathpzc{U}_n(s,\y)|
		&\leq
		e^{2n\|g\|_{\infty}}\mathrm{Lip}(g)
		\E\big[\|\Y_n^{t,\y}(T)-\Y_n^{s,\y}(T)\|_{\H}\big]
		\nonumber\\&\leq
		e^{2n\|g\|_{\infty}}\mathrm{Lip}(g)
		\big(\E\big[\|\Y_n^{t,\y}(T)-\y\|_{\H}]+\E\big[\|\Y_n^{s,\y}(T)-\y\|_{\H}\big]\big)
		\nonumber\\&\leq
		\mathpzc{C}e^{2n\|g\|_{\infty}}\mathrm{Lip}(g)|t-s|^{\frac12},
	\end{align}
	where $\mathpzc{C}=\mathpzc{C}(\mu,\alpha,\beta,R,\|\y\|_{\V},\Tr(\Q),\Tr(\Q_1))$.
	Similarly, one can obtain
	\begin{align}\label{udfyz3}
		|\mathpzc{U}_n(t,\z)-\mathpzc{U}_n(s,\z)|
		&\leq
		e^{2n\|g\|_{\infty}}\mathrm{Lip}(g)
		\E\big[\|\Y_n^{t,\z}(T)-\Y_n^{s,\z}(T)\|_{\H}\big]
		\nonumber\\&\leq
		e^{2n\|g\|_{\infty}}\mathrm{Lip}(g)
		\big(\E\big[\|\Y_n^{t,\z}(T)-\z\|_{\H}\big]+\E\big[\|\Y_n^{s,\z}(T)-\z\|_{\H}\big]\big)
		\nonumber\\&\leq
		\mathpzc{C}e^{2n\|g\|_{\infty}}\mathrm{Lip}(g)|t-s|^{\frac12},
	\end{align}
	where $\mathpzc{C}=\mathpzc{C}(\mu,\alpha,\beta,R,\|\z\|_{\V},\Tr(\Q),\Tr(\Q_1))$. Let us choose $\widetilde{R}=\max\{\|\y\|_{\V},\|\z\|_{\V}\}$. Then, on combining \eqref{udfyz1}-\eqref{udfyz3}, we finally conclude that
	\begin{align*}
		|\mathpzc{U}_n(t,\y)-\mathpzc{U}_n(s,\z)|
		\leq e^{2n\|g\|_{\infty}}\mathrm{Lip}(g)\|\y-\z\|_{\H}+\mathpzc{C}_3
		e^{2n\|g\|_{\infty}}\mathrm{Lip}(g)|t-s|^{\frac12},
	\end{align*}
where 	$\mathpzc{C}_3=\mathpzc{C}_3(\mu,\alpha,\beta,R,\widetilde{R},
	\Tr(\Q),\Tr(\Q_1))$.
\end{proof}

\begin{proposition}\label{Lapvscso}
	Under the assumptions of Proposition \ref{LapunLh}, the function $\mathpzc{U}_n$ is the unique bounded viscosity solution of 
	\begin{equation}\label{LDP1}
		\left\{
		\begin{aligned}
			&(\mathpzc{U}_n)_t+\frac{1}{2n}\mathrm{Tr}(\Q\D^2\mathpzc{U}_n)-
			\frac12\|\Q^{\frac12}\D \mathpzc{U}_n\|_{\H}^2\\&\quad+ (-\mu\A\y-\mathcal{B}(\y)-\alpha\y-\beta\mathcal{C}(\y)+\f(t),\D \mathpzc{U}_n) =0, \ \text{ in } \ (0,T)\times\V, \\
			&\mathpzc{U}_n(T,\y)=g(\y),
		\end{aligned}
		\right.
	\end{equation}
	satisfying \eqref{bdd12}.
\end{proposition}

\begin{proof} 
 The boundedness of the viscosity solution is immediate from the Proposition \ref{LapunLh}.	It suffices to verify that $\mathpzc{U}_n$ is a bounded viscosity subsolution of \eqref{LDP1}. The proof of the bounded viscosity supersolution proceeds in an analogous manner. Let $\psi=\varphi+\mathpzc{h}(\|\cdot\|_{\V})$ be a test function such that the function $\mathpzc{U}_n-\mathpzc{h}(\|\cdot\|_{\V})-\varphi$ attains a local maximum at the point $(t,\y)$. Without loss of generality, we may assume that this maximum is strict and global. Furthermore, due to the boundedness of $\mathpzc{U}_n$, we can also assume that $\varphi,\varphi_t,\D\varphi$ and $\D^2\varphi$ are bounded and uniformly continuous.
%	that $\mathpzc{h}(r)=r^2$ for sufficiently large $r$. 
	\vskip 0.2cm
	\noindent
	\textbf{The points of local maxima is in $\mathcal{D}(\I+\A)$.} Since $\mathpzc{U}_n-\mathpzc{h}(\|\cdot\|_{\V})-\varphi$ has a local maximum at the point $(t,\y)$, we have
	\begin{align}\label{Ladp}
		&\mathpzc{U}_n(t+\eps,\Y_n(t+\eps))-
		\mathpzc{h}(\|\Y_n(t+\eps)\|_{\V})-
		\varphi(t+\eps,\Y_n(t+\eps))
		\nonumber\\&\leq
		\mathpzc{U}_n(t,\y)-\mathpzc{h}(\|\y\|_{\V})-\varphi(t,\y).
	\end{align}
	Let us define $\mathpzc{V}_n:=e^{-n\mathpzc{U}_n}$. Therefore, from \eqref{Ladp}, we deduce the following:
	\begin{align}\label{Ladp1}
		\frac{\mathpzc{V}_n(t+\eps,\Y_n(t+\eps))}
		{\mathpzc{V}_n(t,\y)}\geq
		e^{-n\big(\mathpzc{h}(\|\Y_n(t+\eps)\|_{\V})+
			\varphi(t+\eps,\Y_n(t+\eps))\big)}
		e^{n\big(\mathpzc{h}(\|\y\|_{\V})+\varphi(t,\y)\big)}.
	\end{align}
	Then, on taking the expectation on both sides of \eqref{Ladp1} and using the Markov property of the process $\Y_n(\cdot)$, we obtain
	\begin{align*}
		e^{-n\big(\mathpzc{h}(\|\y\|_{\V})+\varphi(t,\y)\big)}\geq
		\E\Big[e^{-n\big(\mathpzc{h}(\|\Y_n(t+\eps)\|_{\V})+
			\varphi(t+\eps,\Y_n(t+\eps))\big)}\Big].
	\end{align*}
	On dividing both sides of above inequality by $\eps>0$, we rewrite it as 
	\begin{align}\label{Ladp2}
		0\geq
		\E\frac{1}{\eps}\left[e^{-n\big(\mathpzc{h}(\|\Y_n(t+\eps)\|_{\V})+
			\varphi(t+\eps,\Y_n(t+\eps))\big)}-e^{-n\big(\mathpzc{h}(\|\y\|_{\V})+\varphi(t,\y)\big)}\right].
	\end{align}
	Let us write 
	\begin{align*}
		\mathpzc{H}(t,\y)=e^{-n\big(\mathpzc{h}(\|\y\|_{\V})+
			\varphi(t,\y)\big)}.
	\end{align*}
	From Remark \eqref{reFrede}, we write the following Formal Fr\'echet derivatives:
	\begin{align*}
		&\mathcal{D}_{\y}\mathpzc{H}(t,\y)\w
		\nonumber\\&=-n
		\mathpzc{H}(t,\y)\bigg[
		\frac{\mathpzc{h}'(\|\y\|_{\V})}{\|\y\|_{\V}}
		((\A+\I)^{\frac12}\y,(\A+\I)^{\frac12}\w)
		+\mathcal{D}_{\y}\varphi(t,\y)\w\bigg],
		\ \text{ for } \ \w\in\V,
	\end{align*}
	and
	\begin{align*}
		&\mathcal{D}^2_{\y}\mathpzc{H}(t,\y)
		(\w_1,\w_2)
		\nonumber\\&=
		n^2\mathpzc{H}(t,\y)
		\bigg[
		\frac{\mathpzc{h}'(\|\y\|_{\V})}{\|\y\|_{\V}}
		((\A+\I)^{\frac12}\y,(\A+\I)^{\frac12}\w_1)
		+\mathcal{D}_{\y}\varphi(t,\y)\w_1\bigg]
		\nonumber\\&\qquad\times
		\bigg[
		\frac{\mathpzc{h}'(\|\y\|_{\V})}{\|\y\|_{\V}}
		((\A+\I)^{\frac12}\y,(\A+\I)^{\frac12}\w_2)
		+\mathcal{D}_{\y}\varphi(t,\y)\w_2\bigg]
		\nonumber\\&\quad-n\mathpzc{H}(t,\y)
		\left(-\frac{\mathpzc{h}'(\|\y\|_{\V})}{\|\y\|_{\V}^3}
		+\frac{\mathpzc{h}''(\|\y\|_{\V})}{\|\y\|_{\V}^2}\right)
		((\A+\I)^{\frac12}\y,
		(\A+\I)^{\frac12}\w_1)
		((\A+\I)^{\frac12}\y,(\A+\I)^{\frac12}
		\w_2)
		\nonumber\\&\quad-n
		\mathpzc{H}(t,\y)
		\frac{\mathpzc{h}'(\|\y\|_{\V})}{\|\y\|_{\V}}
		((\A+\I)^{\frac12}\w_1,(\A+\I)^{\frac12}\w_2)
		-n\mathpzc{H}(t,\y)
		\mathcal{D}_{\y}^2\varphi(t,\y)(\w_1,\w_2),
		\ \text{ for } \ \w_1,\w_2\in\V.
	\end{align*}
	Note that 
	\begin{align}\label{traceLDP}
		&\Tr(\Q\mathcal{D}^2\mathpzc{H}(t,\y))
		\nonumber\\&=-n\mathpzc{H}(t,\y)
		\Tr(\Q+\Q_1)\|\y\|_{\V}^2\left(-\frac{\mathpzc{h}'(\|\y\|_{\V})}{\|\y\|_{\V}^3}
		+\frac{\mathpzc{h}''(\|\y\|_{\V})}{\|\y\|_{\V}^2}\right)
		\nonumber\\&\quad-
		n\mathpzc{H}(t,\y)\Tr(\Q+\Q_1)\frac{\mathpzc{h}'(\|\y\|_{\V})}
		{\|\y\|_{\V}}-n\mathpzc{H}(t,\y)
		\Tr(\Q\mathcal{D}_{\y}^2\varphi(t,\y))
		\nonumber\\&\quad+
		n^2\mathpzc{H}(t,\y)\frac{(\mathpzc{h}'(\|\y\|_{\V}))^2}
		{\|\y\|_{\V}^2}\Tr(\Q+\Q_1)\|\y\|_{\V}^2+
		n^2\mathpzc{H}(t,\y)\Tr\big(\Q(\mathcal{D}_{\y}\varphi(t,\y)\otimes\mathcal{D}_{\y}\varphi(t,\y))\big)
		\nonumber\\&\quad+
		2n^2\mathpzc{H}(t,\y)
		\Tr\big(\Q((\A+\I)\y\otimes\mathcal{D}_{\y}\varphi(t,\y)\big)
		\nonumber\\&=
		-n\mathpzc{H}(t,\y)
		\Tr(\Q+\Q_1)\left(-\frac{\mathpzc{h}'(\|\y\|_{\V})}{\|\y\|_{\V}}
		+\mathpzc{h}''(\|\y\|_{\V})\right)
		\nonumber\\&\quad-
		n\mathpzc{H}(t,\y)\Tr(\Q+\Q_1)\frac{\mathpzc{h}'(\|\y\|_{\V})}
		{\|\y\|_{\V}}-n\mathpzc{H}(t,\y)
		\Tr(\Q\mathcal{D}_{\y}^2\varphi(t,\y))
		\nonumber\\&\quad+
		n^2\mathpzc{H}(t,\y)(\mathpzc{h}'(\|\y\|_{\V}))^2
		\Tr(\Q+\Q_1)+
		n^2\mathpzc{H}(t,\y)\Tr\big(\Q(\mathcal{D}_{\y}\varphi(t,\y)\otimes\mathcal{D}_{\y}\varphi(t,\y))\big)
		\nonumber\\&\quad+
		2n^2\mathpzc{H}(t,\y)
		\Tr\big(\Q((\A+\I)\y\otimes\mathcal{D}_{\y}\varphi(t,\y)\big).
	\end{align}

	On applying the infinite dimensional  It\^o formula to the function 
	$e^{-n\big(\mathpzc{h}(\|\cdot\|_{\V})+
		\varphi(s,\cdot)\big)}$, for $s\in[t,t+\eps]$, and to the process $\Y_n(\cdot)$, we get
	\begin{align}\label{Ladp5}
		&e^{-n\big(\mathpzc{h}(\|\Y_n(t+\eps)\|_{\V})+
			\varphi(t+\eps,\Y_n(t+\eps))\big)}
		\nonumber\\&=
		e^{-n\big(\mathpzc{h}(\|\y\|_{\V})+\varphi(t,\y)\big)}-n
		\int_t^{t+\eps}
		e^{-n\big(\mathpzc{h}(\|\Y_n(s)\|_{\V})+
			\varphi(s,\Y_n(s))\big)}\varphi_t(s,\Y_n(s))\d s
		\nonumber\\&
		-n\int_t^{t+\eps}
		e^{-n\big(\mathpzc{h}(\|\Y_n(s)\|_{\V})+
			\varphi(s,\Y_n(s))\big)}
		\bigg(-\mu\A\Y_n(s)-\mathcal{B}(\Y_n(s))-\alpha\Y_n(s)-
		\beta\mathcal{C}(\Y_n(s))
		\nonumber\\&\qquad+
		\f(s),\frac{\mathpzc{h}'(\|\Y_n(s)\|_{\V})}{\|\Y_n(s)\|_{\V}}
		(\A+\I)\Y_n(s)+\D\varphi(s,\Y_n(s))\bigg)\d s
		\nonumber\\&+
		\sqrt{n}\int_t^{t+\eps}
		e^{-n\big(\mathpzc{h}(\|\Y_n(s)\|_{\V})+
			\varphi(s,\Y_n(s))\big)}\bigg(\Q^{\frac12}\d\mathbf{W}(s),
		\frac{\mathpzc{h}'(\|\Y_n(s)\|_{\V})}{\|\Y_n(s)\|_{\V}}
		(\A+\I)\Y_n(s)+\D\varphi(s,\Y_n(s))\bigg)
		\nonumber\\&\quad+
			\frac{1}{2n}\int_t^{t+\eps}
			\Tr(\Q\mathcal{D}^2\mathpzc{H}(s,\Y_n(s)))\d s.
	\end{align}
	On dividing both sides of \eqref{Ladp5} by $\eps>0$, then taking expectation, and using \eqref{Ladp2}, we obtain
	\begin{align}\label{Ladp6}
		0\geq&-\frac{n}{\eps}
		\E\int_t^{t+\eps}
		e^{-n\big(\mathpzc{h}(\|\Y_n(s)\|_{\V})+
			\varphi(s,\Y_n(s))\big)}\varphi_t(s,\Y_n(s))\d s 
		\nonumber\\&-
		\frac{n}{\eps}\E\int_t^{t+\eps}
		e^{-n\big(\mathpzc{h}(\|\Y_n(s)\|_{\V})+
			\varphi(s,\Y_n(s))\big)}
		\bigg(-\mu\A\Y_n(s)-\mathcal{B}(\Y_n(s))-\alpha\Y_n(s)-
		\beta\mathcal{C}(\Y_n(s))
		\nonumber\\&\qquad+
		\f(s),\frac{\mathpzc{h}'(\|\Y_n(s)\|_{\V})}{\|\Y_n(s)\|_{\V}}
		(\A+\I)\Y_n(s)+\D\varphi(s,\Y_n(s))\bigg)\d s
		\nonumber\\&\quad+\frac{1}{2n\eps}
			\E\int_t^{t+\eps}
			\Tr(\Q\mathcal{D}^2\mathpzc{H}(s,\Y_n(s)))\d s.
	\end{align}
	Let us write $\mathpzc{g}(s,\Y_n(s)):=e^{-n\big(\mathpzc{h}(\|\Y_n(s)\|_{\V})+\varphi(s,\Y_n(s))\big)}$, for $s\in[t,t+\eps]$. Note that from the assumptions on $\mathpzc{h}$ (see Definition \ref{testD}), there exists a constant $\varkappa>0$ such that 
	$\frac{\mathpzc{h}'(\theta)}{\theta}\geq\varkappa$ for $\theta\in(0,+\infty)$. Therefore, on rearranging the terms in \eqref{Ladp6} and using the equality \eqref{torusequ}, we rewrite it as follows:
	\begin{align}\label{Ladp7}
		&\frac{n\mu\varkappa}{\eps}
		\E\int_{t}^{t+\eps} 
		\mathpzc{g}(s,\Y_n(s))
		(\|\A\Y_n(s)\|_{\H}^2+\|\nabla\Y_n(s)\|_{\H}^2)\d s+
		\frac{n\alpha\varkappa}{\eps}
		\E\int_{t}^{t+\eps} 
		\mathpzc{g}(s,\Y_n(s))
		\|\Y_n(s)\|_{\V}^2\d s
		\nonumber\\&\quad+
		\frac{n\beta\varkappa}{\eps}
		\E\int_{t}^{t+\eps} 
		\mathpzc{g}(s,\Y_n(s))
		\||\Y_n(s)|^{\frac{r-1}{2}}\nabla\Y_n(s)\|_{\H}^{2}\d s+
		\frac{n\beta\varkappa}{\eps}
		\E\int_{t}^{t+\eps} 
		\mathpzc{g}(s,\Y_n(s))
		\|\Y_n(s)\|_{\wi\L^{r+1}}^{r+1}\d s
		\nonumber\\&\leq
		\frac{n}{\eps}
		\E\int_t^{t+\eps}
		e^{-n\big(\mathpzc{h}(\|\Y_n(s)\|_{\V})+
			\varphi(s,\Y_n(s))\big)}\varphi_t(s,\Y_n(s))\d s 
		\nonumber\\&\quad
		-\frac{n\varkappa}{\eps}
		\E\int_{t}^{t+\eps} 
		\mathpzc{g}(s,\Y_n(s))
		(\mathcal{B}(\Y_n(s)),(\A+\I)\Y_n(s))\d s
		\nonumber\\&\quad-
		\frac{n\varkappa}{\eps}
		\E\int_{t}^{t+\eps} 
		\mathpzc{g}(s,\Y_n(s))
		(\f(s),(\A+\I)\Y_n(s))\d s
		\nonumber\\&\quad-
		\frac{n}{\eps}
		\E\int_{t}^{t+\eps} 
		\mathpzc{g}(s,\Y_n(s))
		(\mu\A\Y_n(s)+\mathcal{B}(\Y_n(s))+\alpha\Y_n(s)+
		\beta\mathcal{C}(\Y_n(s))-
		\f(s),\D\varphi(s,\Y_n(s)))\d s
		\nonumber\\&\quad-
			\frac{1}{2n\eps}\E\int_t^{t+\eps}
			\Tr(\Q\mathcal{D}^2\mathpzc{H}(s,\Y_n(s)))\d s.
	\end{align}
	From Hypotheses \ref{trQ1} and\ref{fhyp}, and an  application of H\"older's and Young's inequality yield
	\begin{align}
		|((\mu\A+\alpha\I)\Y_n,\D\varphi(\cdot,\Y_n))|&\leq \frac{\mu\varkappa}{4}\|\A\Y_n\|_{\H}^2+\mathpzc{C}(1+\|\Y_n\|_{\H}^2),\label{Ladp7.0}\\
		|(\mathcal{B}(\Y_n),\A\Y_n)|
		&\leq\frac{\mu}{4}
		\|\A\Y_n\|_{\H}^2+\frac{\beta}{4}\||\Y_n|^{\frac{r-1}{2}}\nabla\Y_n\|_{\H}^2+\varrho_1\|\nabla\Y_n\|_{\H}^2,\label{Ladp7.1}\\
		|(\mathcal{B}(\Y_n),\D\varphi(\cdot,\Y_n))|&\leq \mathpzc{C}(1+\|\Y_n\|_{\H}^2)
		+\frac{\beta\varkappa}{8}\||\Y_n|^{\frac{r-1}{2}}\nabla\Y_n\|_{\H}^2+\varrho_2\|\nabla\Y_n\|_{\H}^2,\label{Ladp7.2}\\
		|(\mathcal{C}(\Y_n),\D\varphi(\cdot,\Y_n))|&\leq	
		\frac{\varkappa}{8}\||\Y_n|^{\frac{r-1}{2}}
		\nabla\Y_n\|_{\H}^2+\mathpzc{C}(1+\|\Y_n\|_{\H})^{r+1}+ \frac{\varkappa}{2}\|\Y_n\|_{\wi\L^{r+1}}^{r+1},\label{Ladp7.3}\\
		|(\f(\cdot),(\A+\I)\Y_n)|&\leq \mathpzc{C}\|\Y_n\|_{\V},\label{Ladp7.4}\\
		|(\f(\cdot),\D\varphi(\Y_n))|&\leq \mathpzc{C}\big(1+\|\Y_n\|_{\H}\big),\label{Ladp7.5}
	\end{align}
	where 
	$\varrho_2:=\frac{r-3}{\varkappa(r-1)}\left[\frac{8}{\beta\varkappa(r-1)}\right]^{\frac{2}{r-3}}$ and  $\varrho_1:=\frac{r-3}{\mu(r-1)}\left[\frac{8}{\beta
		\mu(r-1)}\right]^{\frac{2}{r-3}}$. 
	On substituting \eqref{Ladp7.0} to \eqref{Ladp7.5} in \eqref{Ladp7}, we obtain
	\begin{align}\label{Ladp9}
		&\frac{n\mu\varkappa}{\eps}
		\E\int_{t}^{t+\eps} 
		\mathpzc{g}(s,\Y_n(s))
		(\|\A\Y_n(s)\|_{\H}^2+\|\nabla\Y_n(s)\|_{\H}^2)\d s+
		\frac{n\alpha\varkappa}{\eps}
		\E\int_{t}^{t+\eps} 
		\mathpzc{g}(s,\Y_n(s))
		\|\Y_n(s)\|_{\V}^2\d s
		\nonumber\\&\quad+
		\frac{n\beta\varkappa}{\eps}
		\E\int_{t}^{t+\eps} 
		\mathpzc{g}(s,\Y_n(s))
		\||\Y_n(s)|^{\frac{r-1}{2}}\nabla\Y_n(s)\|_{\H}^{2}\d s+
		\frac{n\beta\varkappa}{2\eps}
		\E\int_{t}^{t+\eps} 
		\mathpzc{g}(s,\Y_n(s))
		\|\Y_n(s)\|_{\wi\L^{r+1}}^{r+1}\d s
		\nonumber\\&\leq\mathpzc{C},  
	\end{align} 
	where $\mathpzc{C}>0$ is a constant, which is independent of $\eps$. By making the use of \eqref{ssee1}, we further deduce the following bound from \eqref{Ladp9}:
	\begin{align}\label{Ladp9.1}
		\frac{1}{\eps}\E\int_{t}^{t+\eps} 
		\mathpzc{g}(s,\Y_n(s))
		\|(\A+\I)\Y_n(s)\|_{\H}^2\d s\leq\mathpzc{C}.
	\end{align}
	Let $M>0$ be a number with $M>\|\y\|_{\V}$ such that if we consider the following set:
	\begin{align*}
		\mathpzc{S}:=\left\{\omega\in\Omega: \sup\limits_{t\leq s\leq T}\|\Y_n(s)\|_{\V}\leq M\right\},
	\end{align*}
	Then, by an application of Markov's inequality, we have $\P(	\mathpzc{S})>0$. It then follows from \eqref{Ladp9.1} that
	\begin{align*}
		\frac{1}{\eps}\E\int_{t}^{t+\eps} 
		\|(\A+\I)\Y_n(s)\|_{\H}^2\mathds{1}_{\mathpzc{S}}\d s \leq\mathpzc{C}.
	\end{align*}
	Therefore, there exists a sequence $\{t_m\}_{m\in\N}$ with $t_m\to t$ as $m\to\infty$ such that $\Y_n(t_m)$ is bounded in $\mathrm{L}^2(\mathpzc{S};\mathcal{D}(\I+\A))$. Thus, the Banach-Alaoglu theorem yields the existence of a sequence $\{t_m\}_{m\in\N}$ (still denoted by the same symbol) such that 
	\begin{align*}
		\Y_n(t_m)\rightharpoonup\wi\y \ \text{ in } \mathrm{L}^2(\mathpzc{S};\mathcal{D}(\I+\A)) \ \text{ as }  \ m\to\infty.
	\end{align*}
	However, from \eqref{ctsdep0.1}, we also have 
	\begin{align*}
		\Y_n(t_m)\to\y \ \text{in} \ \mathrm{L}^2(\mathpzc{S};\H)\text{ as } m\to\infty.
	\end{align*}
	So, finally, the uniqueness of weak limits gives $\y=\wi\y\in\mathcal{D}(\I+\A)$.
	
	\vskip 0.2cm
	\noindent
	\textbf{The subsolution inequality.} Let us now pass the limit as `$\eps\to0$' in \eqref{Ladp6} in order to get the subsolution inequality. We rewrite \eqref{Ladp6} as follows:
	\begin{align}\label{convLDP}
		0\geq&
		-\frac{n}{\eps}
		\E\int_t^{t+\eps}
		\mathpzc{g}(s,\Y_n(s))\varphi_t(s,\Y_n(s))\d s
		\nonumber\\&
		-\frac{n}{\eps}\E\int_t^{t+\eps}
		\mathpzc{g}(s,\Y_n(s))\frac{\mathpzc{h}'(\|\Y_n(s)\|_{\V})}{\|\Y_n(s)\|_{\V}}
		\big(-\mu\A\Y_n(s)-\mathcal{B}(\Y_n(s))-\alpha\Y_n(s)
		\nonumber\\&\quad-
		\beta\mathcal{C}(\Y_n(s))+
		\f(s),(\A+\I)\Y_n(s)\big)\d s
		\nonumber\\&
		-\frac{n}{\eps}\E\int_t^{t+\eps}\mathpzc{g}(s,\Y_n(s))
		\big(-\mu\A\Y_n(s)-\mathcal{B}(\Y_n(s))-\alpha\Y_n(s)-
		\beta\mathcal{C}(\Y_n(s))
		\nonumber\\&\quad+
		\f(s),\D\varphi(s,\Y_n(s))\big)\d s
		+\frac{1}{2n\eps}\E\int_t^{t+\eps}
			\Tr(\Q\mathcal{D}^2\mathpzc{H}(s,\Y_n(s)))\d s.
	\end{align}
	We now examine the convergence of each term on the right hand side of \eqref{convLDP} individually.
	We note that
	\begin{align}\label{convLDP1}
		&\frac{1}{\eps}
		\E\int_{t}^{t+\eps}
		\mathpzc{g}(s,\Y_n(s))\varphi_t(s,\Y_n(s))\d s-
		\mathpzc{g}(t,\y)\varphi_t(t,\y)
		\nonumber\\&=
		\frac{1}{\eps}\E\int_{t}^{t+\eps}
		(\mathpzc{g}(s,\Y_n(s))-\mathpzc{g}(t,\y))\varphi_t(s,\Y_n(s))\d s
		\nonumber\\&\quad+
		\frac{1}{\eps}\E\int_{t}^{t+\eps}
		\mathpzc{g}(t,\y)(\varphi_t(s,\Y_n(s))-\varphi_t(t,\y))\d s.
	\end{align}
	Let us choose $\mathfrak{K}_{s,t}$ in such a way that $$\mathpzc{h}(\|\Y_n(s)\|_{\V})+\varphi(s,\Y_n(s))<\mathfrak{K}_{s,t}<
	\mathpzc{h}(\|\y\|_{\V})+\varphi(t,\y),$$ for $s\in[t,t+\eps]$. By applying mean value theorem and using the fact that $\mathpzc{h}\in\C^2([0,\infty))$ and $\varphi$ is uniformly continuous and bounded, we write
	\begin{align}\label{convLDP2}
		&|\mathpzc{g}(s,\Y_n(s))-\mathpzc{g}(t,\y)|
		\nonumber\\&\leq
		ne^{-n\mathfrak{K}_{s,t}}\big(\underbrace{|\mathpzc{h}(\|\Y_n(s)\|_{\V})-\mathpzc{h}(\|\y\|_{\V})|}_{\text{ use mean value theorem}}+
		\underbrace{|\varphi(s,\Y_n(s))-\varphi(t,\y)|}_{\text{use modulus of continuity}}\big)
		\nonumber\\&\leq
		ne^{-n\mathfrak{K}_{s,t}}\left(\max\limits_{\xi\in[t,t+\eps]}|\mathpzc{h}'(\xi)|
		\big|\|\Y_n(s)\|_{\V}-\|\y\|_{\V}\big|+
		\omega_{\varphi}(\eps+\|\Y_n(s)-\y\|_{\H})\right),
	\end{align}
	where $\omega_{\varphi}$ is some local modulus of continuity of $\varphi$.
	We choose $$\mathfrak{y}_1:=\max\limits_{s\in[t,t+\eps]}e^{-n\mathfrak{K}_{s,t}}
	\max\left\{1,\max\limits_{\xi\in[t,t+\eps]}|\mathpzc{h}'(\xi)|\right\}.$$ Then, \eqref{convLDP2} yields
	\begin{align}\label{convLDP3}
		|\mathpzc{g}(s,\Y_n(s))-\mathpzc{g}(t,\y)|
		\leq
		\mathfrak{y}_1\left(\|\Y_n(s)-\y\|_{\V}+
		\omega_{\varphi}(\eps+\|\Y_n(s)-\y\|_{\H})\right).
	\end{align}
	Further, since $\varphi_t$ is uniformly continuous and bounded, we can deduce the following estimate for $s\in[t,t+\eps]$
	\begin{align}\label{convLDP4}
		|\varphi_t(s,\Y_n(s))-\varphi_t(t,\y)|\leq
		\omega_{\varphi_t}(\eps+\|\Y_n(s)-\y\|_{\H}),
	\end{align} 
	where $\omega_{\varphi_t}$ is some local modulus of continuity of $\varphi_t$. By incorporating \eqref{convLDP3}-\eqref{convLDP4} into \eqref{convLDP1}, along with \eqref{ctsdep0.1}-\eqref{ctsdep0.2}, and noting that $\y\in\mathcal{D}(\I+\A)$ and $\varphi_t$ is uniformly continuous, we obtain
	\begin{align}\label{convLDP5}
		&\left|\frac{1}{\eps}
		\E\int_{t}^{t+\eps}
		\mathpzc{g}(s,\Y_n(s))\varphi_t(s,\Y_n(s))\d s-
		\mathpzc{g}(t,\y)\varphi_t(t,\y)\right|
		\nonumber\\&=
		\max\limits_{s\in[t,t+\eps]}|\varphi_t(s,\Y_n(s))|
		\left(\frac{1}{\eps}\E\int_{t}^{t+\eps}\big|\mathpzc{g}(s,\Y_n(s))
		-\mathpzc{g}(t,\y)\big|\d s\right)
		\nonumber\\&\quad+
		|\mathpzc{g}(t,\y)|\left(\frac{1}{\eps}\E\int_{t}^{t+\eps}
		|\varphi_t(s,\Y_n(s))-\varphi_t(t,\y)|\d s\right).
		\nonumber\\&\leq
		\mathfrak{y}_1\max\limits_{s\in[t,t+\eps]}|\varphi_t(s,\Y_n(s))|
		\left(\frac{1}{\eps}\int_{t}^{t+\eps}\big(\E\|\Y_n(s)-\y\|_{\V}+
		\E[\omega_{\varphi}(\eps+\|\Y_n(s)-\y\|_{\H})]\big)\d s\right)
		\nonumber\\&\quad+
		|\mathpzc{g}(t,\y)|\left(\frac{1}{\eps}\int_{t}^{t+\eps}
		\E[\omega_{\varphi_t}(\eps+\|\Y_n(s)-\y\|_{\H})]\d s\right)
		\nonumber\\&\leq
		\mathfrak{y}_1\max\limits_{s\in[t,t+\eps]}|\varphi_t(s,\Y_n(s))|
		\left(\frac{1}{\eps}\int_{t}^{t+\eps}\sqrt{ \omega_{\y}(\eps)}\d s+
		\frac{1}{\eps}\int_{t}^{t+\eps}\E[\omega_{\varphi}(\eps+\|\Y_n(s)-\y\|_{\H})]
		\d s\right)
		\nonumber\\&\quad+
		|\mathpzc{g}(t,\y)|\left(\frac{1}{\eps}\int_{t}^{t+\eps}
		\E[\omega_{\varphi_t}(\eps+\|\Y_n(s)-\y\|_{\H})]\d s\right).
	\end{align}
     Therefore, the validity of the first integral in \eqref{convLDP} follows in the limit as $\eps\to0$ (see \eqref{convLDP5}). 
		Furthermore, invoking the linearity of the trace operator, the assumptions on $\mathpzc{h}$ (see the definition of test function \ref{testD}), together with  \eqref{ctsdep0.1}-\eqref{ctsdep0.2}, one can deduce the following for the final integral term in \eqref{convLDP}:
		\begin{align*}
			\frac{1}{2n\eps}\E\int_t^{t+\eps}
			\Tr(\Q\mathcal{D}^2\mathpzc{H}(s,\Y_n(s)))\d s
			\to     
			\frac{1}{2n}\Tr(\Q\mathcal{D}^2\mathpzc{H}(t,\y)) \ \text{ in } \ 
			\mathrm{L}^2(\Omega;\H) \ \text{ as } \ \eps\to0.
	\end{align*}
	Using \eqref{Ladp9.1}, we calculate  
		\begin{align}\label{banachLDP}
			&\E\left\|\frac{1}{\eps}\int_t^{t+\eps}
			(\mathpzc{g}(s,\Y_n(s)))^{\frac12}\left(\frac{\mathpzc{h}'(\|\Y_n(s)\|_{\V})}
			{\|\Y_n(s)\|_{\V}}\right)^{\frac12}
			(\A+\I)\Y_n(s)\d s\right\|_{\H}^2
			\nonumber\\&\leq
			\frac{1}{\eps}\E\int_t^{t+\eps}
			\mathpzc{g}(s,\Y_n(s))\frac{\mathpzc{h}'(\|\Y_n(s)\|_{\V})}{\|\Y_n(s)\|_{\V}}
			\|(\A+\I)\Y_n(s)\|_{\H}^2\d s.
			\nonumber\\&\leq\left(\max\limits_{s\in[t,T]}\mathpzc{h}''
			(\|\Y_n(s)\|_{\V})\right)\left[
			\frac{1}{\eps}\E\int_t^{t+\eps}
			\mathpzc{g}(s,\Y_n(s))
			\|(\A+\I)\Y_n(s)\|_{\H}^2\d s\right]
			\nonumber\\&\leq
			\mathpzc{C},
		\end{align}
		where we have used the mean value theorem to $\mathpzc{h}'$ and the fact that $\mathpzc{h}'(0)=0$. %Note that for small $\theta$, we have $\lim\limits_{\theta\to0}\frac{\mathpzc{h}'(\theta)}{\theta}=\mathpzc{h}''(0)>0$. While for large $\theta$, as mentioned in the beginning of the proof, one can take $\mathpzc{h}(\theta)=\theta^2$ so that $\frac{\mathpzc{h}'(\theta)}{\theta}=2$. 
	%    In each of these case, by using \eqref{Ladp9.1}, we can bound the right hand side of \eqref{banachLDP} by a constant independent of $\mathpzc{C}$.    
	Thus, on applying the Banach-Alaoglu theorem, we obtain a sequence $\eps_n\to0$ and an element $\wi\y\in\mathrm{L}^2(\Omega;\H)$ such that 
	\begin{align}\label{wknm}
		\wi\y_n:=\frac{1}{\eps_n}\int_{t}^{t+\eps_n}
		e^{-\frac{n}{2}\big(\mathpzc{h}(\|\Y_n(s)\|_{\V})+\varphi(s,\Y_n(s))\big)}
		\left(\frac{\mathpzc{h}'(\|\Y_n(s)\|_{\V})}
		{\|\Y_n(s)\|_{\V}}\right)^{\frac12}(\A+\I)\Y_n(s)
		\d s\rightharpoonup\wi\y,
	\end{align}
	in $\mathrm{L}^2(\Omega;\H)$ as $n\to+\infty$. One can write
	\begin{align}\label{wknmdiff}
		&(\A+\I)^{-1}\wi\y_n-(\mathpzc{g}(t,\y))^{\frac12}
		\left(\frac{\mathpzc{h}'(\|\y\|_{\V})}
		{\|\y\|_{\V}}\right)^{\frac12}\y
		\nonumber\\&=
		\frac{1}{\eps_n}\int_{t}^{t+\eps_n}
		(\mathpzc{g}(s,\Y_n(s)))^{\frac12}
		\left(\frac{\mathpzc{h}'(\|\Y_n(s)\|_{\V})}
		{\|\Y_n(s)\|_{\V}}\right)^{\frac12}\Y_n(s)\d s-
		(\mathpzc{g}(t,\y))^{\frac12}
		\left(\frac{\mathpzc{h}'(\|\y\|_{\V})}
		{\|\y\|_{\V}}\right)^{\frac12}\y
		\nonumber\\&=\underbrace{
			\frac{1}{\eps_n}\int_{t}^{t+\eps_n}
			\big((\mathpzc{g}(s,\Y_n(s)))^{\frac12}-(\mathpzc{g}(t,\y))^{\frac12}\big)
			\left(\frac{\mathpzc{h}'(\|\Y_n(s)\|_{\V})}
			{\|\Y_n(s)\|_{\V}}\right)^{\frac12}\Y_n(s)\d s}_{I}
		\nonumber\\&\quad+\underbrace{
			\frac{1}{\eps_n}\int_{t}^{t+\eps_n}
			(\mathpzc{g}(t,\y))^{\frac12}\left[
			\left(\frac{\mathpzc{h}'(\|\Y_n(s)\|_{\V})}
			{\|\Y_n(s)\|_{\V}}\right)^{\frac12}-
			\left(\frac{\mathpzc{h}'(\|\y\|_{\V})}
			{\|\y\|_{\V}}\right)^{\frac12}\right]\Y_n(s)\d s}_{II}
		\nonumber\\&\quad+\underbrace{
			\frac{1}{\eps_n}\int_{t}^{t+\eps_n}
			(\mathpzc{g}(t,\y))^{\frac12}
			\left(\frac{\mathpzc{h}'(\|\y\|_{\V})}
			{\|\y\|_{\V}}\right)^{\frac12}(\Y_n(s)-\y)\d s}_{III}.
	\end{align}
	Using \eqref{convLDP3}, \eqref{ctsdep0.1}-\eqref{ctsdep0.2} and the fact that $\y\in\mathcal{D}(\I+\A)$, it follows that $I$ and $III$ converges to 0 in $\mathrm{L}^2(\Omega;\H)$ as $n\to+\infty$. Furthermore, by the assumption on $\mathpzc{h}$ (see Definition \ref{testD}) and the application of mean value theorem, we find
	\begin{align}\label{wknmdiff2}
		&\left|\left(\frac{\mathpzc{h}'(\|\Y_n(s)\|_{\V})}
		{\|\Y_n(s)\|_{\V}}\right)^{\frac12}-
		\left(\frac{\mathpzc{h}'(\|\y\|_{\V})}
		{\|\y\|_{\V}}\right)^{\frac12}\right|
		\nonumber\\&\leq
		\frac{1}{\left(\frac{\mathpzc{h}'(\|\y\|_{\V})}
			{\|\y\|_{\V}}\right)^{\frac12}}\frac{1}{\|\Y_n(s)\|_{\V}}\left[
		\max\limits_{\theta\in[t,t+\eps]}|\mathpzc{h}''(\theta)|+
		\frac{\mathpzc{h}'(\|\y\|_{\V})}
		{\|\y\|_{\V}}\right]\|\Y_n(s)-\y\|_{\V}.
	\end{align}
	Using \eqref{wknmdiff2} and \eqref{ctsdep0.1}-\eqref{ctsdep0.2}, the integral $II$ converges to $0$ as $n\to+\infty$ in $\mathrm{L}^2(\Omega;\H)$. Consequently, from \eqref{wknmdiff}, one can deduce that  
	\begin{align*}
		(\A+\I)^{-1}\wi\y_n\to(\mathpzc{g}(t,\y))^{\frac12}
		\left(\frac{\mathpzc{h}'(\|\y\|_{\V})}
		{\|\y\|_{\V}}\right)^{\frac12}\y \ \text{ as } \ n\to+\infty.
	\end{align*}
	Therefore, by the uniqueness of the weak limit, it follows that  $$\wi\y=(\mathpzc{g}(t,\y))^{\frac12}
	\left(\frac{\mathpzc{h}'(\|\y\|_{\V})}
	{\|\y\|_{\V}}\right)^{\frac12}(\A+\I)\y.$$
	Moreover, in view of \eqref{ssee1}, \eqref{ctsdep0.1}-\eqref{ctsdep0.2}, one can also verify the following limits:
	\begin{align}\label{wknm1.1}
		\frac{1}{\eps_n}\int_{t}^{t+\eps_n}\mathpzc{g}(s,\Y_n(s))\frac{\mathpzc{h}'(\|\Y_n(s)\|_{\V})}{\|\Y_n(s)\|_{\V}}
		\|\Y_n(s)\|_{\H}^2\d s
		\to \mathpzc{g}(t,\y)\frac{\mathpzc{h}'(\|\y\|_{\V})}{\|\y\|_{\V}}
		\|\y\|_{\H}^2
	\end{align}
	and 
	\begin{align}\label{wknm1.2}
		\frac{1}{\eps_n}\int_{t}^{t+\eps_n} \mathpzc{g}(s,\Y_n(s))\frac{\mathpzc{h}'(\|\Y_n(s)\|_{\V})}{\|\Y_n(s)\|_{\V}}\|\nabla\Y_{n}(s)\|_{\H}^2\d s \to\mathpzc{g}(t,\y)\frac{\mathpzc{h}'(\|\y\|_{\V})}{\|\y\|_{\V}}
		\|\nabla\y\|_{\H}^2,
	\end{align} 
	in $\mathrm{L}^2(\Omega;\H)$ as $n\to+\infty$. From \eqref{banachLDP} and \eqref{wknm}, using Jensen's inequality and  the weak lower semicontinuity property of norm, we obtain
	\begin{align}\label{wknm1}
		&\liminf_{n\to+\infty}\E\frac{1}{\eps_n}\int_{t}^{t+\eps_n}\mathpzc{g}(s,\Y_n(s))\frac{\mathpzc{h}'(\|\Y_n(s)\|_{\V})}{\|\Y_n(s)\|_{\V}}
		\|(\A+\I)\Y_n(s)\|_{\H}^2
		\d s\nonumber\\&\geq\liminf_{n\to+\infty} \E\left\|\frac{1}{\eps_n}\int_{t}^{t+{\eps_n}}
		(\mathpzc{g}(s,\Y_n(s)))^{\frac12}\left(\frac{\mathpzc{h}'(\|\Y_n(s)\|_{\V})}
		{\|\Y_n(s)\|_{\V}}\right)^{\frac12}
		(\A+\I)\Y_n(s)\d s \right\|_{\H}^2\nonumber\\&\geq
		\mathpzc{g}(t,\y)\frac{\mathpzc{h}'(\|\y\|_{\V})}{\|\y\|_{\V}}\|(\A+\I)\y\|_{\H}^2.
	\end{align}
	A similar argument as we performed  above yields 
	\begin{align}\label{wknm2}
		\frac{1}{\eps_n}\int_{t_0}^{t_0+\eps_n}(\A+\I)\Y_{n}(s)\d s\rightharpoonup(\A+\I)\y_0 \  \text{ in } \ \mathrm{L}^2(\Omega;\H)
		\  \text{ as } \ n\to+\infty.
	\end{align}
	Moreover, since $\|\A\Y_{n}\|_{\H}^2=\|(\A+\I)\Y_{n}\|_{\H}^2-\|\Y_{n}\|_{\H}^2-2\|\nabla\Y_{n}\|_{\H}^2$, therefore by using \eqref{wknm1.1}-\eqref{wknm1}, we have the following lower bound:
	\begin{align}\label{wknm2.2}
		&\liminf_{n\to+\infty}\E\frac{1}{\eps_n}\int_{t_0}^{t_0+\eps_n} \mathpzc{g}(s,\Y_n(s))\frac{\mathpzc{h}'(\|\Y_n(s)\|_{\V})}{\|\Y_n(s)\|_{\V}}
		\|\A\Y_{n}(s)\|_{\H}^2\d s
		\nonumber\\&\geq
		\mathpzc{g}(t,\y)\frac{\mathpzc{h}'(\|\y\|_{\V})}{\|\y\|_{\V}}\|\A\y\|_{\H}^2.
	\end{align}
	Finally, in view of \eqref{wknm1.1}-\eqref{wknm1.2} and \eqref{wknm2.2}, the following limit is immediate:
	\begin{align}\label{wknmpas8}
		&\liminf_{n\to+\infty}\E\frac{1}{\eps_n}\int_{t_0}^{t_0+\eps_n}
		\mathpzc{g}(s,\Y_n(s))\frac{\mathpzc{h}'(\|\Y_n(s)\|_{\V})}{\|\Y_n(s)\|_{\V}}
		\big((\mu\A+\alpha\I)\Y_n(s),(\A+\I)\Y_n(s)\big)\d s
		\nonumber\\&\geq   
		\mathpzc{g}(t,\y)\frac{\mathpzc{h}'(\|\y\|_{\V})}{\|\y\|_{\V}}
		\big((\mu\A+\alpha\I)\y,(\A+\I)\y\big).
	\end{align}
	The second and third integrals in \eqref{convLDP}, whose integrands involve the bilinear and nonlinear operators $\mathcal{B}(\cdot)$ and $\mathcal{C}(\cdot)$ respectively, can now be addressed by invoking \eqref{wknmpas8} along with \eqref{wknm2} and \eqref{wknmdiff2}. The convergence of the exponential factor is established similarly to the argument in \eqref{convLDP3}. We refer to \cite[Theorem 6.1]{smtm1} for a more detailed account of the these arguments. Hence, by denoting the test function $\psi(t,\y)=\varphi(t,\y)+\mathpzc{h}(\|\y\|_{\V})$, and on passing $\liminf$ into \eqref{convLDP}, we finally obtain a following subsolution inequality:
	\begin{align*}
		0\geq&
		-n\mathpzc{g}(t,\y)\left[\psi_t(t,\y)+\big(-\mu\A\y-\mathcal{B}(\y)-\alpha\y-
		\beta\mathcal{C}(\y)+\f(s),\D\psi(t,\y)\big)\right]
		\nonumber\\&+
		\mathpzc{g}(t,\y)\frac{1}{2n}
		\Tr\left[\Q\big(-n\mathcal{D}\psi(t,\y)\otimes\mathcal{D}\psi(t,\y)+
		\mathcal{D}^2\psi(t,\y)\big)\right].
	\end{align*} 
	The uniqueness of $\mathpzc{U}_n$ follows from Proposition \ref{LapunLh} and Theorem \ref{comparison}.
\end{proof}
%{\color{Maroon}
%	\begin{remark} 
%		While proving the supersolution inequality, we get the following bound instead of \eqref{Ladp9.1}:
%		\begin{align}\label{Ladp9.11}
%			\frac{1}{\eps}\E\int_{t}^{t+\eps} 
%			\mathfrak{g}(s,\Y_n(s))
%			\|(\A+\I)\Y_n(s)\|_{\H}^2\d s\leq\mathpzc{C},
%		\end{align}
%		where $\mathfrak{g}(s,\Y_n(s)):=e^{n\big(\mathpzc{h}(\|\Y_n(s)\|_{\V})-
%			\varphi(s,\Y_n(s))\big)}$,
%\end{remark}}

%\subsection{Identifying Laplace limit: Viscosity solution to first order HJB}
\subsection{Limiting HJB equation}\label{passnhjb}
 On passing the limit into \eqref{LDP1}, we formally obtain the following first order equation:
\begin{equation}\label{LDP2}
	\left\{
	\begin{aligned}
	&	\mathcal{U}_t-
		\frac12\|\Q^{\frac12}\D \mathcal{U}\|_{\H}^2+ (-\mu\A\y-\mathcal{B}(\y)-\alpha\y-\beta\mathcal{C}(\y)+\f(t),\D \mathcal{U}) =0, \\
	&	\mathcal{U}(T,\y)=g(\y), \ \text{ in } \ (0,T)\times\V.
	\end{aligned}
	\right. 
\end{equation}
Note that, since
\begin{align*}
	-\frac12\|\Q^{\frac12}\mathpzc{p}\|_{\H}^2=
	\inf_{\z\in\H}\left\{(\Q^{\frac12}\z,\mathpzc{p})
	+\frac12\|\z\|_{\H}^2\right\},
\end{align*}
we can interpret \eqref{LDP2} as the first order HJB equation associated with the optimal control problem of CBF equations which is given below.

Let $t\in[0,T]$ and $\y\in\H$. We consider the following controlled CBF equations described by the velocity vector field $\Y(\cdot):[t,T]\times\mathbb{T}^d\to\R^d$:
\begin{equation}\label{stapdet}
	\left\{
	\begin{aligned}
		\frac{\d\Y(s)}{\d s}&= -\mu\A\Y(s)-\mathcal{B}(\Y(s))-\alpha\Y(s)-\beta\mathcal{C}(\Y(s))+
		\f(s)+\Q^{\frac12}\u(s) ,  \  \text{ in } \ (t,T)\times\H, \\
		\Y(t)&=\y\in\H,
	\end{aligned}
	\right.
\end{equation}
where  $\u(\cdot)\in\mathrm{L}^2(t,T;\H)$ is some given control function. We want to minimize the following cost functional associated with the state equation \eqref{stapdet}:
\begin{align}\label{costLDP}
	\mathcal{J}(t,\y;\u(\cdot)):=\frac12\int_t^T \|\u(s)\|_{\H}^2\d s+
	g(\Y(T)),
\end{align}
over all controls $\u(\cdot)\in\mathrm{L}^2(t,T;\H)$. It is well-known from \cite[Theorem 4.2]{SMTM} that the controlled CBF system  \eqref{stapdet} has a unique strong solution 
\begin{align*}
	\Y\in \C([t,T];\V)\cap\mathrm{L}^2(t,T;\mathcal{D}(\A))\cap\mathrm{L}^{r+1}(t,T;\wi{\L}^{p(r+1)}),
\end{align*}  
where $p\in[2,\infty)$ for $d=2$ and $p=3$ for $d=3$, 
for any $\u(\cdot)\in\mathrm{L}^2(t,T;\H)$ and $\y\in\V$. Therefore, the cost functional \eqref{costLDP} is well-defined.
Further, since $g$ is bounded, the minimization of the cost functional \eqref{costLDP} can be restricted to the following class of functions:
\begin{align*}
	\u(\cdot)\in\mathscr{M}_t:=\left\{\u(\cdot)\in\mathrm{L}^2(t,T;\H):\int_t^T \|\u(s)\|_{\H}^2\d s\leq K=2\|g\|_{\infty}\right\}.
\end{align*}
The value function of the optimal control problem \eqref{stapdet}-\eqref{costLDP} is defined as 
\begin{align}\label{valueLDP}
	\mathcal{V}(t,\y)=\inf\limits_{\u(\cdot)\in\mathrm{L}^2(t,T;\H)}
	\mathcal{J}(t,\y;\u(\cdot))=
	\inf\limits_{\u(\cdot)\in\mathscr{M}_t}
	\mathcal{J}(t,\y;\u(\cdot)).
\end{align}
Note that one can verify that the value function $\mathcal{V}(\cdot)$ satisfies the following dynamic programming principle:
	\begin{align}\label{dynprp}
		\mathcal{V}(t,\y)=\inf\limits_{\z(\cdot)\in\mathrm{L}^2(t,T;\H)}
		\left\{\int_t^{\upeta}\|\z(s)\|_{\H}^2\d s+ \mathcal{V}(\upeta,\Y(\upeta;t,\y,\z(\cdot)))\right\},
	\end{align}
	for all $0\leq t\leq\upeta\leq T$ and $\y\in\V$.
	
Our aim is first to demonstrate that $\mathcal{V}$ is indeed the viscosity solution of \eqref{LDP2}. Once this is achieved, we can then conclude that the Laplace limit can be identified with the limit obtained from the convergence of viscosity solution of \eqref{LDP1}. Let us first recall the well-posedness results for the CBF equations. The following proposition addresses the continuity properties and energy estimates of the solution to controlled CBF equations \eqref{stapdet}. The proof of this proposition is standard. A comprehensive explanation together with the full proof can be found in \cite{SMTM,smtm2}.
\begin{proposition}
	Assume that Hypotheses \ref{trQ1} and \ref{fhyp} hold. Let $\u(\cdot)\in\mathscr{M}_t$. Then
	\begin{enumerate}
		\item For any initial data $\y\in\H$, there exists a \emph{unique weak solution} $\Y(\cdot)=\Y(\cdot;t,\y,\u(\cdot))$ of the state equation \eqref{stapdet}. Moreover, the following uniform energy estimates for $s\in[t,T]$ holds:
		\begin{align}\label{engest11}
			&\|\Y(s)\|_{\H}^2
			+\mu\int_t^{s}\|\nabla\Y(\tau)\|_{\H}^2\d\tau+
			\alpha\int_t^{s}\|\Y(\tau)\|_{\H}^2\d\tau+
			\beta\int_t^{s}\|\Y(\tau)\|_{\widetilde{\L}^{r+1}}^{r+1}\d\tau
			\nonumber\\&\leq
			 \mathpzc{C}(\alpha,
			\|\Q^{\frac12}\|_{\mathscr{L}(\H;\H)},R,\|\y\|_{\H},
			\|\u\|_{\mathrm{L}^2(t,T;\H)})
			e^{s-t},
		\end{align}
		for all $s\in[t,T]$ and
		\begin{align}\label{engest22}
			&\sup\limits_{s\in[t,T]}\|\Y(s)\|_{\H}^2
			+\mu\int_t^{T}\|\nabla\Y(\tau)\|_{\H}^2\d\tau+
			\alpha\int_t^{T}\|\Y(\tau)\|_{\H}^2\d\tau+
			\beta\int_t^{T}\|\Y(\tau)\|_{\widetilde{\L}^{r+1}}^{r+1}\d\tau
			\nonumber\\&\leq 
			 \mathpzc{C}(\alpha,\|\Q^{\frac12}\|_{\mathscr{L}(\H;\H)},R,\|\y\|_{\H},
			\|\u\|_{\mathrm{L}^2(t,T;\H)},T).
		\end{align}
		\item Furthermore, for any $\y\in\V$, there exists a \emph{unique strong solution} $\Y(\cdot)=\Y(\cdot;t,\y,\u(\cdot))$ of the state equation \eqref{stapdet} satisfying the following uniform energy estimates:
		\begin{align}\label{engest1}
			&\|\nabla\Y(s)\|_{\H}^2
			+\mu\int_t^{s}\|\A\Y(\tau)\|_{\H}^2\d\tau+
			\alpha\int_t^{s}\|\nabla\Y(\tau)\|_{\H}^2\d\tau
			\nonumber\\&\quad+
			\beta\int_t^{s}
			\||\Y(\tau)|^{\frac{r-1}{2}}\nabla\Y(\tau)\|_{\H}^{2}\d\tau
			\nonumber\\&\leq \big(\|\nabla\y\|_{\H}^2+ \mathpzc{C}(\mu,\beta,
			\|\A^{\frac12}\Q^{\frac12}\|_{\mathscr{L}(\H;\H)},R,
			\|\u\|_{\mathrm{L}^2(t,T;\H)})(s-t)\big)
			e^{\varrho(s-t)},
		\end{align}
		for all $s\in[t,T]$, and
		\begin{align}\label{engest2}
			&\sup\limits_{s\in[t,T]}\|\nabla\Y(s)\|_{\H}^2+ \int_t^{T}\|\A\Y(\tau)\|_{\H}^2\d\tau
			+\int_t^{T}\||\Y(\tau)|^{\frac{r-1}{2}}\nabla\Y(\tau)\|_{\H}^{2} \d\tau
			\nonumber\\&\leq  \mathpzc{C}(T,\mu,\beta,\|\y\|_{\V},
			\|\A^{\frac12}\Q^{\frac12}\|_{\mathscr{L}(\H;\H)},R,
			\|\u\|_{\mathrm{L}^2(t,T;\H)}).
		\end{align}
		\item  For each $\y_1,\y_2\in\V$, there exists a constant $\mathpzc{C}$ independent of $t,\y_1,\y_2$ and $\u(\cdot)$ such that 
		\begin{align}\label{ctsdep0d}
			&\|\Y_1(s)-\Y_2(s)\|_{\H}^2+\int_t^s \|\nabla(\Y_1-\Y_2)(\tau)\|_{\H}^2\d\tau+\int_t^s \|\Y_1(\tau)-\Y_2(\tau)\|_{\wi\L^{r+1}}^{r+1}\d\tau\nonumber\\&\leq
			\|\y_1-\y_2\|_{\H}^2 e^{\mathpzc{C}(s-t)},
			\ \text{ for all } \ s\in[t,T],
		\end{align} 
		where $\Y_1(\cdot)=\Y_1(\cdot;t,\y_1,\u(\cdot))$ and $\Y_2(\cdot)=\Y_2(\cdot;t,\y_2,\u(\cdot))$ are two strong solutions of \eqref{stapdet}, with $\Y_1(t)=\y_1$ and $\Y_2(t)=\y_2$.  
		
		\item  For every $\y\in\V$, we have
		\begin{align}\label{ctsdep0.1d}
			\|\Y(s)-\y\|_{\H}^2\leq  \mathpzc{C}\left(\mu,\beta,T,R,\|\y\|_{\V},\|\u\|_{\mathrm{L}^2(t,T;\H)},
			\Tr(\Q)\right)(s-t), 
		\end{align} 
		for all $s\in[t,T]$.
		
		\item  For every $\y\in\V$ and $\mathpzc{M}>0$, there exists a modulus $\omega$ such that if 
		$\|\u\|_{\mathrm{L}^2(t,T;\H)}\leq\mathpzc{M}$, then
		\begin{align*}
			\|\Y(s)-\y\|_{\V}^2\leq\omega_{\y,\mathpzc{M}}(s-t), 
		\end{align*}
		for all  $s\in[t,T]$.
	\end{enumerate}
\end{proposition}

\begin{proposition}\label{comparisondet}
	Let $g\in\mathrm{Lip}_b(\H)$ and $\u(\cdot)\in\mathscr{M}_t$. Moreover, assume that Hypotheses \ref{trQ1} and \ref{fhyp} are satisfied. Then 
	\begin{itemize}
		\item[a.)] the value function $\mathcal{V}$ defined in \eqref{valueLDP} is bounded;
		\item[b.)] there exists a constant $\mathpzc{D}_1$ and, for every $R>0$, a constant $\mathpzc{D}_2=\mathpzc{D}_2(R)$ such that 
		\begin{align}\label{valuedff5}
			|\mathcal{V}(t_1,\y_1)-\mathcal{V}(t_2,\y_2)|\leq
			\mathpzc{D}_1\|\y_1-\y_2\|_{\H}+\mathpzc{D}_2
			(\max\{\|\y_1\|_{\V},\|\y_2\|_{\V}\})|t_1-t_2|^{\frac12},
		\end{align}
		for all $\y_1,\y_2\in\V$ and all  $t_1,t_2\in[0,T]$.
	\end{itemize} 
\end{proposition}

\begin{proof}
	Since $g\in\mathrm{Lip}_b(\H)$ and control $\u(\cdot)\in\mathscr{M}_t$, therefore, the boundedness of the value function $\mathcal{V}(\cdot)$ is immediate.
	From \eqref{costLDP}-\eqref{valueLDP} and the fact that $g\in\mathrm{Lip}_b(\H)$, we find
	\begin{align}\label{valuedff7}
		|\mathcal{V}(t,\y_1)-\mathcal{V}(t,\y_2)|&=
		\left|\inf\limits_{\u(\cdot)\in\mathscr{M}_t}
		\mathcal{J}(t,\y_1;\u(\cdot))-
		\inf\limits_{\u(\cdot)\in\mathscr{M}_t}
		\mathcal{J}(t,\y_2;\u(\cdot))\right|
		\nonumber\\&\leq
		\sup\limits_{\u(\cdot)\in\mathscr{M}_t}
		\big|\mathcal{J}(t,\y_1;\u(\cdot))-
		\mathcal{J}(t,\y_2;\u(\cdot))\big|
		\nonumber\\&=
		\sup\limits_{\u(\cdot)\in\mathscr{M}_t}
		|g(\Y(T;t,\y_1,\u(\cdot)))-g(\Y(T;t_2,\y_2,\u(\cdot)))|
		\nonumber\\&\leq
		\|g\|_{\text{Lip}}\|\Y(T;t,\y_1,\u(\cdot))-\Y(T;t_2,\y_2,\u(\cdot))\|_{\H}
		\nonumber\\&\leq
		\mathpzc{C}\|g\|_{\text{Lip}}\|\y_1-\y_2\|_{\H},
	\end{align}
	where in the last step we have used \eqref{ctsdep0d}. 
	Now, let $t_1,t_2\in[0,T]$, with $t_1\leq t_2$. For any optimal control $\u(\cdot)\in\mathscr{M}_{t_1}$, we obtain
		\begin{align}\label{valuedff1}
			\mathcal{V}(t_1,\y)\leq\int_{t_1}^T\|\u(s)\|_{\H}^2\d s+
			g(\Y(T;t_1,\y,\u(\cdot))).
		\end{align}
		We define
		\begin{equation*}
			\wi\u(t):=
			\left\{
			\begin{aligned}
				\boldsymbol{0}, \ \text{ if } \ t_1\leq t<t_2,\\
				\u(t),  \ \text{ if } \ t_2\leq t<T.
			\end{aligned}
			\right.
		\end{equation*}
		Then, we find
		\begin{align*}
			\frac12\int_{t_2}^T \|\wi\u(s)\|_{\H}^2\d s=
			\frac12\int_{t_1}^T \|\wi\u(s)\|_{\H}^2\d s\leq
			\frac12\int_{t_1}^T \|\u(s)\|_{\H}^2\d s,
		\end{align*}
		which implies that $\wi\u(\cdot)\in\mathscr{M}_{t_2}$ and we have
		\begin{align}\label{valuedff2}
			\mathcal{V}(t_2,\y)\leq\int_{t_2}^T\|\wi\u(s)\|_{\H}^2\d s+
			g(\Y(T;t_2,\y,\wi\u(\cdot)))\leq
			\int_{t_2}^T\|\u(s)\|_{\H}^2\d s+
			g(\Y(T;t_2,\y,\u(\cdot))).
		\end{align}
		Therefore, from \eqref{valuedff1}-\eqref{valuedff2} and \eqref{ctsdep0.1d}, we estimate
		\begin{align}\label{valuedff3}
			\mathcal{V}(t_2,\y)-\mathcal{V}(t_1,\y)
			&\leq-\int_{t_2}^{t_1}\|\u(s)\|_{\H}^2\d s +g(\Y(T;t_2,\y,\u(\cdot)))-g(\Y(T;t_1,\y,\u(\cdot)))
			\nonumber\\&\leq
			|g(\Y(T;t_2,\y,\u(\cdot)))-g(\Y(T;t_1,\y,\u(\cdot)))|
			\nonumber\\&\leq
			\|g\|_{\text{Lip}}\|\Y(T;t_2,\y,\u(\cdot))-\Y(T;t_1,\y,\u(\cdot))\|_{\H}
			\nonumber\\&\leq
			\mathpzc{C}\|g\|_{\text{Lip}}|t_2-t_1|^{\frac12}.
		\end{align}
		On changing the role of $t_1$ and $t_2$, we get
		\begin{align}\label{valuedff4}
			\mathcal{V}(t_1,\y)-\mathcal{V}(t_2,\y)
			\leq
			\mathpzc{C}\|g\|_{\text{Lip}}|t_1-t_2|^{\frac12}.
		\end{align} 
		On combining \eqref{valuedff3}-\eqref{valuedff4}, we arrive at
		\begin{align}\label{valuedff6}
			|\mathcal{V}(t_1,\y)-\mathcal{V}(t_2,\y)|
			\leq
			\mathpzc{C}\|g\|_{\text{Lip}}|t_1-t_2|^{\frac12}.
	\end{align}
	Finally, combining \eqref{valuedff7} and \eqref{valuedff6}, one can obtain \eqref{valuedff5}.
	%       	\begin{align*}
		%       		|\mathcal{V}(t_1,\y)-\mathcal{V}(t_2,\y)|&=
		%       	\left|\inf\limits_{\u(\cdot)\in\mathscr{M}_{t_1}}
		%       	\mathcal{J}(t_1,\y;\u(\cdot))-
		%       	\inf\limits_{\u(\cdot)\in\mathscr{M}_{t_2}}
		%       	\mathcal{J}(t_2,\y;\u(\cdot))\right|
		%       	\nonumber\\&=
		%       	\left|\inf\limits_{\u(\cdot)\in\mathscr{M}_{t_1}}
		%       	\mathcal{J}(t_1,\y;\u(\cdot))-
		%       	\inf\limits_{\u(\cdot)\in\mathscr{M}_{t_1}}
		%       	\mathcal{J}(t_2,\y;\u(\cdot))\right|
		%       	\nonumber\\&\leq
		%       	\sup\limits_{\u(\cdot)\in\mathscr{M}_{t_1}}
		%       	\big|\mathcal{J}(t_1,\y;\u(\cdot))-
		%       	\mathcal{J}(t_2,\y;\u(\cdot))\big|
		%       	\nonumber\\&=
		%       	\sup\limits_{\u(\cdot)\in\mathscr{M}_{t_1}}
		%       	\bigg|\int_{t_1}^{T} \|\u(s)\|_{\H}^2\d s-\int_{t_2}^{T} \|\u(s)\|_{\H}^2\d s
		%       	\nonumber\\&\qquad+ g(\Y(T;t_1,\y,\u(\cdot)))-g(\Y(T;t_2,\y,\u(\cdot)))\bigg|
		%       	\end{align*}
\end{proof}

\begin{proposition}\label{Lapvscso1}
	The value function $\mathcal{V}$ defined in \eqref{valueLDP} is the unique bounded viscosity solution of \eqref{LDP2} satisfying \eqref{bdd12}. 
\end{proposition}

\begin{proof}
	The boundedness of the value function $\mathcal{V}$ follows from part (a) of Proposition \ref{comparisondet}. It remains to prove that $\mathcal{V}$ is a viscosity solution to the HJB equation \eqref{LDP2}. We will focus on sketching the proof of viscosity supersolution only. The steps of the proof of viscosity subsolution follow analogously. Let $\psi(t,\y)=\varphi(t,\y)-\mathpzc{h}(\|\y\|_{\V})$ be a test function and suppose that $\mathcal{V}-\psi$ attains a local minimum at $(t,\y)$. From \eqref{dynprp}, for every $0<\eps<T-t$, there exists a control $\u_{\eps}(\cdot)$ such that 
	\begin{align*}
		\mathcal{V}(t,\y)+\eps^2>\int_{t}^{t+\eps}
		\|\u_\eps(s)\|_{\H}^2\d s+ \mathcal{V}(t+\eps,\Y_\eps(t+\eps)),
	\end{align*}
	where $\Y_\eps(\cdot)=\Y(\cdot;t,\y,\u_\eps(\cdot))$. Since, $\mathcal{V}-\psi$ has a local minimum at $(t,\y)$, we write from above strict inequality
	\begin{align}\label{ULDP1}
		\eps&\geq\frac{1}{\eps}(\underbrace{\varphi(t+\eps,\Y_\eps(t+\eps))-
			\varphi(t,\y)-\mathpzc{h}(\|\Y_\eps(t+\eps)\|_{\V})
			+\mathpzc{h}(\|\y\|_{\V})}_{\text{use \cite[Proposition 5.5, Chapter 2]{XJJMY}}})+
		\frac{1}{\eps}\int_{t}^{t+\eps}\|\u_\eps(s)\|_{\H}^2\d s
		\nonumber\\&\geq
		\frac{1}{\eps}\int_{t_0}^{t_0+\eps}\bigg(\varphi_t(s,\Y_\eps(s))-\big(\mu\A\Y_\eps(s)+\mathcal{B}(\Y_\eps(s))+\alpha\Y_\eps(s)+
		\beta\mathcal{C}(\Y_\eps(s)),\D\varphi(s,\Y_\eps(s))\big)
		\nonumber\\&\qquad+
		\big(\f(s),\D\varphi(s,\Y_\eps(s))\big)
		+\big(\Q^{\frac12}\u_\eps(s),\D\varphi(s,\Y_\eps(s))\big)\bigg)\d s
		\nonumber\\&\qquad+\frac{1}{\eps}\int_{t_0}^{t_0+\eps}
		\underbrace{\frac{\mathpzc{h}'(\|\Y_\eps(s)\|_{\V})}
			{\|\Y_\eps(s)\|_{\V}}}_{\geq\varkappa>0}
		\bigg(\big(\mu\A\Y_\eps(s)+\alpha\Y_\eps(s),(\A+\I)\Y_\eps(s)\big)+
		(\mathcal{B}(\Y_\eps(s)),(\A+\I)\Y_\eps(s))
		\nonumber\\&\qquad+
		\beta(\mathcal{C}(\Y_\eps(s)),(\A+\I)\Y_\eps(s))-\big(\f(s),
		(\A+\I)\Y_\eps(s)\big)-\big(\Q^{\frac12}\u_\eps(s),
		(\A+\I)\Y_\eps(s)\big)\bigg)\d s
		\nonumber\\&\qquad+  
		\frac{1}{\eps}\int_{t}^{t+\eps}\|\u_\eps(s)\|_{\H}^2\d s.
	\end{align}
	From the assumptions on $\varphi$ (see Definition \ref{testD} of test function ), we deduce 
	\begin{align}\label{unsol1}
		\|\D\varphi(\cdot,\Y_\eps)\|_{\H}\leq \mathpzc{C}(1+\|\Y_\eps\|_{\H}), \ \ 
		|\varphi_t(\cdot,\Y_\eps)|\leq \mathpzc{C}(1+\|\Y_\eps\|_{\H}).
	\end{align} 
	Applying the Cauchy-Schwarz inequality together with \eqref{unsol1}, and using $\|\Q^{\frac12}\|_{\mathscr{L}(\H)}^2\leq\Tr(\Q)$, we obtain the following calculation:
	\begin{align}\label{ULDP}
		&\big(\u_\eps(s),\Q^{\frac12}\D\varphi(s,\Y_\eps(s))\big)+
		\big(\u_\eps(s),\Q^{\frac12}(\A+\I)\Y_\eps(s)\big)
		\nonumber\\&\geq
		-\frac{1}{2}\|\u_\eps(s)\|_{\H}^2-\Tr(\Q)\|(\A+\I)\Y_\eps(s)\|_{\H}^2
		-\mathpzc{C}(1+\|\Y_\eps(s)\|_{\H}^2).
	\end{align}
	Furthermore, by employing the Cauchy Schwarz and Young's inequalities, together with \eqref{unsol1}, Remark \ref{rg3L3r}, the identity \eqref{torusequ} and Hypothesis \ref{fhyp}, we estimate following:
	\begin{align}
		\mu|(\A\Y_\eps,\D\varphi(\cdot,\Y_\eps))|&\leq \frac{\mu\varkappa}{4}\|\A\Y_\eps\|_{\H}^2+\mathpzc{C}
		(1+\|\Y_\eps\|_{\H}^2),
		\label{vdp5}\\
		|(\mathcal{B}(\Y_\eps),(\A+\I)\Y_\eps)|
		&\leq\frac{\mu}{2}
		\|\A\Y_\eps\|_{\H}^2+\frac{\beta}{4}\||\Y_\eps|^{\frac{r-1}{2}}
		\nabla\Y_\eps\|_{\H}^2+\varrho_3\|\nabla\Y_\eps\|_{\H}^2,\label{vdp5.1}\\
		|(\mathcal{B}(\Y_\eps),\D\varphi(\cdot,\Y_\eps))|&\leq \mathpzc{C}(1+\|\Y_\eps\|_{\H}^2)
		+\frac{\beta\varkappa}{4}\||\Y_\eps|^{\frac{r-1}{2}}\nabla\Y_\eps\|_{\H}^2+\varrho_4\|\nabla\Y_\eps\|_{\H}^2,\label{vdp5.2}\\
		|(\mathcal{C}(\Y_\eps),\D\varphi(\cdot,\Y_\eps))|&\leq
		\frac{\varkappa}{4}\||\Y_\eps|^{\frac{r-1}{2}}
		\nabla\Y_\eps\|_{\H}^2+
		\mathpzc{C}(1+\|\Y_\eps\|_{\H})^{r+1}+
		\varkappa\|\Y_\eps\|_{\wi\L^{r+1}}^{r+1}.
		\label{vdp55}\\
		(\mathcal{C}(\Y_\eps),(\A+\I)\Y_\eps)&\geq
		\||\Y_\eps(s)|^{\frac{r-1}{2}}\nabla\Y_\eps(s)\|_{\H}^2+
		\|\Y_\eps(s)\|_{\wi\L^{r+1}}^{r+1},\label{vdp5.3}\\
		|(\f(\cdot),\D\varphi(\cdot,\Y_\eps))|&\leq 
		\mathpzc{C}\big(1+\|\Y_\eps\|_{\H}\big), \
		|(\f(\cdot),(\A+\I)\Y_\eps)|\leq \mathpzc{C}\|\Y_\eps\|_{\V},\label{vdp6LDP}
	\end{align}
	where 
	$\varrho_4:=\frac{r-3}{2\varkappa(r-1)}\left[\frac{4}{\beta\varkappa (r-1)}\right]^{\frac{2}{r-3}}$ and  $\varrho_3:=\frac{r-3}{2\mu(r-1)}\left[\frac{4}{\beta
		\mu(r-1)}\right]^{\frac{2}{r-3}}$. By combining \eqref{ULDP}–\eqref{vdp6LDP}, substituting the result into \eqref{ULDP1}, and applying the energy estimate \eqref{engest1}, we obtain the following bound:
	\begin{align*}
		\frac{1}{\eps}\int_{t}^{t+\eps}\|\A\Y_\eps(s)\|_{\H}^2\d s\leq\mathpzc{C}=\mathpzc{C}(\mu,\alpha,R,\beta,\|\y\|_{\V},\Tr(\Q),
		\Tr(\Q_1)),
	\end{align*}
	where we have used the boundedness of the control term $\u_\eps(\cdot)$, that is, the boundedness of the term $\|\u_\eps\|_{\mathrm{L}^2(t,t+\eps;\H)}$. Thus, there exist sequences $\{\eps_n\}_{n\in\N}$ with $\eps_n\to0$ and $\{t_n\}_{n\in\N}\subset(t,t+\eps_n)$ such that
	\begin{align*}
		\|\Y_{\eps_n}(s)\|_{\V_2}^2\leq\mathpzc{C}.
	\end{align*}
	Then, by the Banach-Alaoglu theorem, we have (along a subsequence, still denoted by the same) $\Y_{\eps_n}(t_n)\rightharpoonup\bar{\y}$ in $\V_2$  as $n\to+\infty.$ On the other hand, from  \eqref{ctsdep0.1d}, we have $\Y_{\eps_n}(t_n)\to\y$ in $\H$ as $n\to+\infty$ and hence $\y=\bar{\y}\in\V_2$.
	
	We now take the $\liminf$ as $\eps\to0$ in \eqref{ULDP1}. This is justified by using the boundedness of $\|\u_\eps(\cdot)\|_{\mathrm{L}^2(t,t+\eps;\H)}$, energy estimates \eqref{engest1}, \eqref{engest2} and arguments analogous to those used in the proof of \cite[Theorem 6.1]{smtm2} (see also \cite[Theorem 6.3]{AS2}).
		Hence, we obtain the following supersolution inequality
		\begin{align*}
			0\geq&
			-\varphi_t(t,\y)+\big(-\mu\A\y-\mathcal{B}(\y)-\alpha\y-
			\beta\mathcal{C}(\y)+\f(s),\D\psi(t,\y)\big)
			-\frac{1}{2}\|\Q^{\frac12}\D\psi(t,\y)\|_{\H}^2.
		\end{align*}
		Furthermore, since $\mathcal{V}$ is Lipschitz continuous with respect to the space variable (see Proposition \ref{comparisondet}), the uniqueness of $\mathcal{V}$ follows immediate from Theorem \ref{comparison}.
\end{proof}

The following result is an immediate consequence of Proposition \ref{Lapvscso}, \ref{Lapvscso1}  and Theorem \ref{comparisonappr}, which shows the convergence of the sequence of viscosity solutions $\{\mathpzc{U}_n\}_{n\in\N}$ of \eqref{LDP1} and identifies its limit as a viscosity solution of the limiting HJB equation \eqref{LDP2}.

\begin{corollary}\label{comparisonds}
	Under Hypotheses \ref{trQ1} and \ref{fhyp}, and for $g \in \mathrm{Lip}_b(\mathcal{H})$, let $\mathpzc{U}_n$ and $\mathcal{V}$ denote the unique bounded viscosity solutions of \eqref{LDP1} and \eqref{LDP2}, respectively. Then we have:
	\begin{align*}
		\lim\limits_{n\to+\infty}\|\mathpzc{U}_n-
		\mathcal{V}\|_{\infty}=0. 
	\end{align*}
\end{corollary}

\subsection{Existence of a Laplace limit at a single time}
Let us now establish the existence of a Laplace limit at a single time for the original equation \eqref{stap}. For a function $g\in\mathrm{Lip}_b(\H)$ and $t>0$, we define
\begin{align}\label{vngyd}
	\mathscr{V}_n(t)g(\y):=\frac{1}{n}\log\E\big[e^{-ng(\Y_n(t))}\big|
	\Y_n(0)=\y\big],
\end{align} 
where $\Y_n(\cdot)=\Y_n(\cdot;0,\y)$ is the solution of the system \eqref{stap} on $[0,+\infty)$. Using \eqref{LapLacefun}, with terminal time $T=t$ and starting at $0$, we can write \eqref{vngyd} as 
	\begin{align*}
		\mathscr{V}_n(t)g(\y)=-\mathpzc{U}_n(0,\y).
\end{align*}
From this point onward, and throughout this section, we assume that Hypotheses \ref{trQ1} and \ref{fhyp} are satisfied for all subsequent results. 
%In view of Propositions \ref{LapunLh}, \ref{Lapvscso}, \ref{comparisondet} and \ref{Lapvscso1}
Considering Corollary \ref{comparisonds}, one can establish the following result:
\begin{lemma}\label{vfLDP1} 
	Let $t\geq0$, $\Y_n(t)=\Y_n(t;0,\y)$ and $g\in\mathrm{Lip}_b(\H)$. Then
	\begin{align*}
		\mathscr{V}(t)g(\y):=\lim\limits_{n\to+\infty}
		\mathscr{V}_n(t)g(\y)=-\mathcal{V}(0,\y),
	\end{align*}   
	uniformly on bounded subsets of $\V$, where $\mathcal{V}$ is the value function, defined in \eqref{dynprp}, with terminal time $T=t$.
	     exists uniformly on bounded subsets of $\V$.
\end{lemma}

The following results (Lemma \ref{vfLDP2} and Proposition \ref{vfLDP3}) extends the previous result to a broader class of functions $g$, with the motivation drawn from \cite{AS2}.
\begin{lemma}\label{vfLDP2}
	Let $g$, $g_m$ are weakly sequentially continuous functions on $\V$ such that 
	\begin{itemize}
		\item the sequence $\{g_m\}_{m\in\N}$ satisfies
		$\|g_m\|_{\infty}\leq M$, for some $M>0$;
		\item $g_m\to g$ uniformly on bounded subsets of $\V$.
	\end{itemize}
	Then, for every $\mathpzc{R}>0$, $\eps>0,$ there exists $m_0\in\N$ such that 
	\begin{align}\label{expLDP1}
		\text{ if } m,n\geq m_0 \ \text{ then } 
		\sup\limits_{\|\y\|\leq \mathpzc{R}}|\mathscr{V}_n(t)g_m(\y)- \mathscr{V}_n(t)g(\y)|\leq\eps.
	\end{align} 
\end{lemma}

\begin{proof}
	Let $\|\y\|_{\V}\leq K$ and choose $\mathpzc{R}>K$. Define a set
	\begin{align*}
		\mathcal{S}_1^{n,\mathpzc{R}}:=
		\left\{\omega\in\Omega:\|\Y_n(t)\|_{\V}>\mathpzc{R}\right\}.
	\end{align*}
		Then, by an application of Markov's inequality and exponential moment estimate \eqref{expmoest}, we find
	\begin{align*}
		\P(	\mathcal{S}_1^{n,\mathpzc{R}})&=
		\P\left\{\omega\in\Omega:\sup\limits_{0\leq s \leq t} \|\Y_n(s)\|_{\V}>\mathpzc{R}\right\}
		\nonumber\\&=
		\P\left\{\omega\in\Omega:e^{n\mathpzc{c}_1\sup\limits_{0\leq s \leq t} \|\Y_n(s)\|_{\V}^2}> e^{n\mathpzc{c}_1{\mathpzc{R}}^2}\right\}
		\nonumber\\&\leq
		\frac{1}{e^{n\mathpzc{c}_1{\mathpzc{R}}^2}}
		\E\bigg[e^{n\mathpzc{c}_1\sup\limits_{0\leq s \leq t} \|\Y_n(s)\|_{\V}^2}\bigg]=
		\frac{1}{e^{n\mathpzc{c}_1{\mathpzc{R}}^2}}
		\E\bigg[\sup\limits_{0\leq s \leq t} e^{n\mathpzc{c}_1 \|\Y_n(s)\|_{\V}^2}\bigg]
		\nonumber\\&\leq
		\frac{1}{e^{n\mathpzc{c}_1{\mathpzc{R}}^2}}
		e^{n\mathpzc{c}_2}.
	\end{align*}
	In the final step, we have used the continuity and strict monotonicity of the exponential function, along with the continuity of the norm $\|\cdot\|_{\V}$. These properties allow us to interchange the supremum and the exponential.
	Therefore, there exists $\mathpzc{R}>\widetilde{\mathpzc{R}}(K)$, depending on $K$ such that 
	\begin{align}\label{expLDP11}
		\P(\mathcal{S}_1^{n,\mathpzc{R}})\leq e^{-3Mn}.
	\end{align}
	Given that $g_m\to g$ uniformly on bounded subsets of $\V$. It means, for every $\eps>0$ and for every bounded subset, say $B_{\mathpzc{R}}:=\{\y\in\V:\|\y\|_{\V}\leq \mathpzc{R}\}$, of $\V$, there exists $m_1\in\N$ such that for all $m\geq m_1$, we have
	\begin{align}\label{expLDP12}
		\sup\limits_{\y\in B_{\mathpzc{R}}}|g_m(\y)-g(\y)|\leq\frac{\eps}{2}.
	\end{align}
	We write
	\begin{align}\label{expLDP13}
		\E \big[e^{-ng_m(\Y_n(\cdot))}\big]&=
		\E \big[e^{-ng(\Y_n(\cdot))}\big]+
		\E\bigg[\big(e^{-ng_m(\Y_n(\cdot))}-e^{-ng(\Y_n(\cdot))}
		\big)\mathds{1}_{\Omega\setminus\mathcal{S}_1^{n,\mathpzc{R}}}\bigg]
		\nonumber\\&\quad+
		\E\bigg[\big(e^{-ng_m(\Y_n(\cdot))}-e^{-ng(\Y_n(\cdot))}
		\big)\mathds{1}_{\mathcal{S}_1^{n,\mathpzc{R}}}\bigg]
		\nonumber\\&=
		\E e^{-ng_m(\Y_n(\cdot))}+
		\E\bigg[e^{-ng(\Y_n(\cdot))}\big
		(e^{-n(g_m(\Y_n(\cdot))-g(\Y_n(\cdot)))}-1\big) \mathds{1}_{\Omega\setminus\mathcal{S}_1^{n,\mathpzc{R}}}\bigg]
		\nonumber\\&\quad+
		\E\bigg[\big(e^{-ng_m(\Y_n(\cdot))}-e^{-ng(\Y_n(\cdot))}
		\big)\mathds{1}_{\mathcal{S}_1^{n,\mathpzc{R}}}\bigg].
	\end{align} 
	Utilizing \eqref{expLDP11} and \eqref{expLDP12} in \eqref{expLDP13}, we obtain the following lower and upper bound in:
	\begin{align*}
		&\E\big[e^{-ng(\Y_n(\cdot))}\big]+
		\E\big[e^{-ng(\Y_n(\cdot))}\big](e^{-\frac{n\eps}{2}}-1)-
		2e^{-2Mn}
		\nonumber\\&\leq
		\E \big[e^{-ng_m(\Y_n(\cdot))}\big]
		\nonumber\\&\leq
		\E \big[e^{-ng(\Y_n(\cdot))}\big]+
		\E\big[e^{-ng(\Y_n(\cdot))}\big](e^{\frac{n\eps}{2}}-1)+
		2e^{-2Mn}.
	\end{align*}
	On simplifying the above inequality, we obtain 
	\begin{align}\label{expLDP14}
		&\E\big[e^{-ng(\Y_n(\cdot))}\big]e^{-\frac{n\eps}{2}}
		\left(1-\frac{2e^{-2Mn}}
		{\E\big[e^{-ng(\Y_n(\cdot))}\big]e^{-\frac{n\eps}{2}}}\right)
		\nonumber\\&\leq
		\E\big[e^{-ng_m(\Y_n(\cdot))}\big]
		\nonumber\\&\leq
		\E\big[e^{-ng(\Y_n(\cdot))}\big]e^{\frac{n\eps}{2}}
		\left(1+\frac{2e^{-2Mn}}
		{\E\big[e^{-ng(\Y_n(\cdot))}\big]e^{\frac{n\eps}{2}}}\right).
	\end{align}
	On taking logarithm in \eqref{expLDP14} and employing Lemma \ref{LOG}, we find
	\begin{align}\label{expLDP15}
		&\frac{1}{n}\log\E\big[e^{-ng(\Y_n(\cdot))}\big]-\frac{\eps}{2}-
		\frac{1}{n}\frac{4e^{-3Mn}}
		{\E\big[e^{-ng(\Y_n(\cdot))}\big]e^{-\frac{n\eps}{2}}}
		\nonumber\\&\leq
		\frac{1}{n}\log\E\big[e^{-ng_m(\Y_n(\cdot))}\big]
		\nonumber\\&\leq
		\frac{1}{n}\log\E\big[e^{-ng(\Y_n(\cdot))}\big]+\frac{\eps}{2}+
		\frac{1}{n}\frac{2e^{-3Mn}}
		{\E\big[e^{-ng(\Y_n(\cdot))}\big]e^{\frac{n\eps}{2}}}.
	\end{align}
	Finally, from \eqref{expLDP11}, we deduce the following from \eqref{expLDP15}:
	\begin{align*}
		&\frac{1}{n}\log\E\big[e^{-ng(\Y_n(\cdot))}\big]-\frac{\eps}{2}-
		\frac{4}{n}e^{-Mn}e^{\frac{n\eps}{2}}
		\nonumber\\&\leq
		\frac{1}{n}\log\E\big[e^{-ng_m(\Y_n(\cdot))}\big]
		\nonumber\\&\leq
		\frac{1}{n}\log\E\big[e^{-ng(\Y_n(\cdot))}\big]+\frac{\eps}{2}+
		\frac{2}{n}e^{-Mn}e^{-\frac{n\eps}{2}},
	\end{align*}
	which completes the proof of \eqref{expLDP1}. 
\end{proof}

\begin{proposition}\label{vfLDP3}
	Let $g$ be bounded and weakly sequentially continuous on $\V$. Then the following statements hold true:
	\begin{itemize}
		\item[(a)] For every $n\in\N$, the function $\mathscr{V}_n(t)g$ is weakly sequentially continuous on $\V$.
		\item[(b)] For every $\y\in\V$, the following limit
		\begin{align*}
			\mathscr{V}(t)g(\y)=\lim\limits_{n\to+\infty}
			\mathscr{V}_n(t)g(\y),
		\end{align*}
		exists and is uniform on bounded subsets of $\V$. In particular, $\mathscr{V}(t)g(\y)$ is weakly sequentially continuous on $\V$.
		\item[(c)] If $g_n$'s are weakly sequentially continuous on $\V$, such that $\|g_n\|_{\infty}\leq M$ for all $n\in\N$, and $g_n\to g$ uniformly on bounded subsets of $\V,$ then
		\begin{align}\label{expLDP2233}
			\lim\limits_{n\to+\infty} \mathscr{V}_n(t)g_n(\y)=
			\mathscr{V}(t)g(\y),
		\end{align}
		uniformly on bounded subsets of $\V$.
	\end{itemize}
\end{proposition}

\begin{proof}
	\noindent
	\textbf{(a).}
	Let us fix $\mathpzc{R}>0$ and take $\y_1,\y_2$ in $\V$ with $\|\y_1\|_{\V}\leq\mathpzc{R},\|\y_2\|_{\V}\leq\mathpzc{R}$. Let $\Y_n^1$ and $\Y_n^2$ be two solutions of \eqref{stap} in the sense of Definition \ref{def-var-strong} such that 
	$\Y_n^1(0)=\y_1$ and $\Y_n^2(0)=\y_2$. Let $m>0$ be any large number. Then, by an application of Markov's inequality and \eqref{ssee1}, we calculate
	\begin{align}\label{expLDP2}
		\P\big(\max\{\|\Y_n^1\|_{\V},\|\Y_n^2\|_{\V}\}>m\big)&=
		\P\big((\max\{\|\Y_n^1\|_{\V},\|\Y_n^2\|_{\V}\})^2>m^2\big)
		\nonumber\\&\leq
		\frac{1}{m^2}\E[(\max\{\|\Y_n^1\|_{\V},\|\Y_n^2\|_{\V}\})^2]
		\nonumber\\&\leq
		\frac{1}{m^2}\E[\|\Y_n^1\|_{\V}^2+\|\Y_n^2\|_{\V}^2]
		\nonumber\\&\leq
		\frac{\mathpzc{C}}{m^2},
	\end{align}
	where $\mathpzc{C}=\mathpzc{C}(\mathpzc{R},T,\Q_1,\mu,\alpha,\beta)$. 
	Let $\uptheta_m$ be the local modules of continuity of $e^{-n g}$ in the $\|\cdot\|_{\H}-$norm on set $\{\z\in\V:\|\z\|_{\V}\leq m\}$. By definition it means 
	\begin{align*}
		\uptheta_m(\delta)=\sup\{|e^{-n g(\z_1)}-e^{-n g(\z_2)}|:\|\z_1\|_{\V},\|\z_2\|_{\V}\leq m, \ \|\z_1-\z_2\|_{\H}\leq\delta\}.
	\end{align*}
	Furthermore, by the properties of modulus of continuity, we have
	\begin{align}\label{expLDP21}
		|e^{-n g(\z_1)}-e^{-n g(\z_2)}|\leq \uptheta_m(\|\z_1-\z_2\|_{\H}),
	\end{align}
	for all $\z_1,\z_2$ in the set $\{\z\in\V:\|\z\|_{\V}\leq m\}$. Let us define the set
		\begin{align*}
		\mathcal{S}_{2,m}:=
		\left\{\omega\in\Omega:\|\Y_n^1(t)\|_{\V},\|\Y_n^2(t)\|_{\V}\leq m \right\}.
	\end{align*}
	Now, from \eqref{ctsdep0}, \eqref{expLDP2} and \eqref{expLDP21} and the concavity of $\uptheta_m$,
	we calculate 
	\begin{align}\label{expLDP22}
		&|e^{n\mathscr{V}_n(t)g(\y_1)}-e^{n\mathscr{V}_n(t)g(\y_2)}|
		\nonumber\\&=
		\big|\E\big[e^{-ng(\Y_n^1(t))}\big]-\E\big[e^{-ng(\Y_n^2(t))}\big]\big|
		\nonumber\\&\leq
		\big|\E\big[(e^{-ng(\Y_n^1(t))}-e^{-ng(\Y_n^2(t))})
		\mathds{1}_{\mathcal{S}_{2,m}^c}\big]\big|
		+\big|\E\big[(e^{-ng(\Y_n^1(t))}-e^{-ng(\Y_n^2(t))})
		\mathds{1}_{\mathcal{S}_{2,m}}\big]\big|
		\nonumber\\&\leq
		\frac{2\mathpzc{C}}{m^2}e^{n\|g\|_{\infty}}+
		\E\big[\uptheta_m(\|\Y_n^1(t)-\Y_n^2(t)\|_{\H})\big]
		\nonumber\\&\leq
		\frac{2\mathpzc{C}}{m^2}e^{n\|g\|_{\infty}}+
		\uptheta_m\big(\E\big[\|\Y_n^1(t)-\Y_n^2(t)\|_{\H}\big]\big)
		\nonumber\\&\leq
		\frac{2\mathpzc{C}}{m^2}e^{n\|g\|_{\infty}}+
		\uptheta_m\big(\mathpzc{C}\|\y_1-\y_2\|_{\H}\big).
	\end{align}
	Note that for sufficiently large $m$, and using the compact embedding 
	$\V\hookrightarrow\H$, it follows from \eqref{expLDP22} that the map $\y\mapsto e^{n\mathscr{V}_n(t)g(\y)}$ is weakly sequentially continuous. Since the logarithm is a continuous function, we conclude that $\mathscr{V}_n(t)g$ is weakly sequentially continuous in $\H$.  
	\vskip 0.2cm
	\noindent
	\textbf{(b).} 
	Note that since $g:\V\to\R$ is bounded and weakly sequentially continuous, we can find functions $g_m\in\mathrm{Lip}_b(\H)$ which converge to $g$ uniformly on bounded subsets of $\V$, that is,
	\begin{align}\label{denseLP}
		\sup\limits_{\|\y\|_{\V}\leq\mathpzc{k}}|g_m(\y)-g(\y)|\to0 \ \text{ as } m\to\infty, \ \text{ for every } \mathpzc{k}>0.
	\end{align}
	Let us first establish \eqref{denseLP}. To do this, for every $m\geq1$, we first define $\tilde{g}:\H\to\R$ by 
	\begin{equation*}
		\tilde{g}(\y)=
		\left\{
		\begin{aligned}
			g(\y), \ \ \|\y\|_{\V}\leq m,\\
			-\|g\|_{\infty}, \ \text{ elsewhere }.
		\end{aligned}
		\right.
	\end{equation*}
	We claim that $\tilde{g}$ \emph{is upper-semicontinuous on} $\H$.
	For this, define a closed ball
	\begin{align*}
		B_m:=\{\y\in\V:\|\y\|_{\V}\leq m\}.
	\end{align*}
	
	%       Let $\{\y_n\}_{n\in\N}$ be a sequence in $B_m$ such that $\y_n\to\y$ in $\H$. Since, $\|\y_n\|_{\V}\leq m$, the Banach-Alaoglu theorem yields a subsequence  $\{\y_{n_j}\}_{j\in\N}$ in $\V$ such that $\y_{n_j}\rightharpoonup\hat{\y}$ in $\V$. Further, the compact embedding $\V\hookrightarrow\H$ gives $\y_{n_j}\to\hat{\y}$ in $\H$. By uniqueness of weak limits, we have $\hat{\y}=\y$. By lower semicontinuity of $\V-$norm, we have $\|\y\|_{\V}\leq m$, which implies that $\y\in B_m$. Thus, $B_m$ is closed ball in $\H$.
	Note that $B_m$ is closed in $\H$. We consider the following two cases:
	
	\textbf{Case-I:} \emph{When} $\y\in B_m$.  Let $\{\y_n\}_{n\in\N}$ be a sequence in $\H$ such $\y_n\to\y$ in $\H$. If $\y_n\in B_m$, for each $n\in\N$, the Banach-Alaoglu theorem yields a subsequence  $\{\y_{n_j}\}_{j\in\N}$ in $\V$ such that $\y_{n_j}\rightharpoonup\hat{\y}$ in $\V$ as $j\to\infty$. Further, the compact embedding $\V\hookrightarrow\H$ gives $\y_{n_j}\to\hat{\y}$ in $\H$ as $j\to\infty$. By the uniqueness of weak limits, we have $\hat{\y}=\y$. Then, along a subsequence, by weak sequential continuity of $g$, we find 
	\begin{align*}
		\limsup\limits_{j\to\infty} \tilde{g}(\y_{n_j})=
		\limsup\limits_{j\to\infty} g(\y_{n_j})=
		g(\y)=\tilde{g}(\y).
	\end{align*} 
	If $\y_n\notin B_m$ for some $n$, say $n=n_0$, then $\tilde{g}(\y_{n_0})=-\|g\|_{\infty}$. Then, we have
	\begin{align*}
		\limsup\limits_{n_0\to\infty} \tilde{g}(\y_{n_0})=
		-\|g\|_{\infty}\leq g(\y)=\tilde{g}(\y).
	\end{align*}
	
	\textbf{Case-II:} \emph{When} $\y\notin B_m$. Let $\{\y_n\}_{n\in\N}$ be a sequence in $\H$ such $\y_n\to\y$ in $\H$. Then, there exists $N\in\N$ such that $\y_n$ does not belong to $B_m$ for sufficiently large $n\geq N$. 
%	If it happens, then by the similar argument as we did above and the lower semicontinuity of $\V-$norm, we get $\y\in B_m$, which is a contradiction. 	So, $\y_n\notin B_m$ for sufficiently large $n\geq N$. 
Therefore, by definition of $\tilde{g}$, we find
	\begin{align*}
		\limsup\limits_{n\to+\infty} \tilde{g}(\y_{n})=
		-\|g\|_{\infty}=\tilde{g}(\y).
	\end{align*}
	Therefore, in all the cases, we conclude that, $\tilde{g}$ \emph{is upper-semicontinuous on} $\H$. Let us define, for any $\delta>0$, a $\delta$ sup-convolution of $\tilde{g}$ as
	\begin{align*}
		\tilde{g}_{\delta}(\y)=\sup\limits_{\z\in\H}\left\{\tilde{g}(\z)-
		\frac{\|\y-\z\|_{\H}^2}{\delta}\right\}.
	\end{align*}
	Then, $\tilde{g}_{\delta}\in\mathrm{Lip}_b(\H)$ and, since the embedding $\V\hookrightarrow\H$ is compact, one can have the following uniform convergence (see \cite[Appendix D.3]{GFAS}):
	\begin{align*}
		\tilde{g}_{\delta}\to\tilde{g}=g \ \text{ uniformly on } B_m \ \text{ as } \ \delta\to0.
	\end{align*}
	Therefore, for small $\delta=\delta(m)$, denoting $g_m:=\tilde{g}_{\delta(m)}$, we can have 
	\begin{align*}
		\sup\limits_{\|\y\|_{\V}\leq m}
		|g_m(\y)-g(\y)|\leq\frac{1}{m}.
	\end{align*}
	
	Consider the functions $g_m,m\in\N$ defined in the said way as above. Then, it follows from Lemma \ref{vfLDP2} that for every $\eps, R>0,$ there exists $n_0$ such that if $n,m_1,m_2\geq m_0$, we have
	\begin{align}\label{expLDP23}
		\sup\limits_{\|\y\|_{\V}\leq R} |\mathscr{V}_n(t)\g_{m_1}(\y)-
		\mathscr{V}_n(t)\g_{m_2}(\y)|\leq 2\eps.
	\end{align}    
	Moreover, since $g_{m_1}, g_{m_2} \in \mathrm{Lip}_b(\H)$, according to Lemma \ref{vfLDP1}, one can take the limit as $n \to +\infty$ in \eqref{expLDP23} to obtain
	\begin{align}\label{expLDP24}
		\sup\limits_{\|\y\|_{\V}\leq R} |\mathscr{V}(t)\g_{m_1}(\y)-
		\mathscr{V}(t)\g_{m_2}(\y)|\leq 2\eps,
	\end{align}
	for all $m_1,m_2\geq n_0$. This shows that the sequence $\{\mathscr{V}(t)\g_m\}_{m\in\N}$ is uniformly Cauchy in the ball $B_m$, and since, $B_m$ is closed in $\H$ with respect to $\|\cdot\|_{\V}-$norm, therefore, $\lim\limits_{m\to\infty}\mathscr{V}(t)g_m(\y)$ exists uniformly on bounded subsets of $\V$ (which is also clear from the representation formula for  $\mathscr{V}(t)g_m(\y)$). Again, using the convergence of $\mathscr{V}_n(t)g_{m_0}$ as $n\to+\infty$ provided by Lemma \ref{vfLDP1}, we obtain from \eqref{expLDP1}  that for $n\geq n_1=n_1(m_0)$
	\begin{align}\label{expLDP25}
		\sup\limits_{\|\y\|_{\V}\leq R} |\mathscr{V}(t)\g_{m_0}(\y)-
		\mathscr{V}_n(t)\g(\y)|\leq 2\eps.
	\end{align}
	Combining \eqref{expLDP24} and \eqref{expLDP25} we finally have that for 
	$n\geq n_1$
	\begin{align*}
		\sup\limits_{\|\y\|_{\V}\leq R} |\lim\limits_{m\to\infty}\mathscr{V}(t)\g_{m}(\y)-
		\mathscr{V}_n(t)\g(\y)|\leq 4\eps.
	\end{align*}
     This shows that 
		\begin{align}\label{expLDP27}
			\mathscr{V}(t)g(\y)=\lim\limits_{n\to+\infty}\mathscr{V}_n(t)\g(\y)=
			\lim\limits_{m\to\infty}\mathscr{V}(t)\g_{m}(\y),
		\end{align}
		uniformly on bounded sets of $\V$. The weak sequential continuity of $\mathscr{V}(t)g$ follows directly from \eqref{expLDP27} and part $(a)$.
	
	\vskip 0.2cm
	\noindent
	\textbf{(c).} Since 
	$$|\mathscr{V}_n(t)g_n(\y)-\mathscr{V}(t)g(\y)|
	\leq|\mathscr{V}_n(t)g_n(\y)-\mathscr{V}_n(t)g(\y)|+|\mathscr{V}_n(t)g(\y)-
	\mathscr{V}(t)g(\y)|,$$ therefore the convergence \eqref{expLDP2233} follows directly from Part $(b)$ and  Lemma \ref{vfLDP2}.
\end{proof}

\begin{remark} 
		According to Propositions \ref{vfLDP1} and \ref{vfLDP3}, if $g$ is bounded and weakly sequentially continuous on $\V$, then $\mathscr{V}(t)g(\y)$ can be represented as follows:
		\begin{align*}
			\mathscr{V}(t)g(\y)=-\inf\limits_{\u(\cdot)\in\mathrm{L}^2(0,t;\H)}
			\left\{\frac12\int_0^t \|\u(s)\|_{\H}^2\d s+g(\Y(t;0,\y))\right\}.
	\end{align*}
\end{remark}

\section{Large deviation principle}\setcounter{equation}{0}\label{LaDePr}
In this section, we show that the sequence $\{\Y_n(\cdot)\}$, where $\Y_n(\cdot)$ is a solution to \eqref{stap} on $[0,+\infty)$ with $\Y_n(0)=\y\in\H$, satisfies the LDP in the space $\C([0,+\infty);\H)$ (equipped with the topology of local uniform convergence). As we mentioned in Subsection \ref{FKLDP}, we first prove the LDP in the path space $\D([0,+\infty);\H)$ and then by using the $\C-$exponential tightness, we establish the LDP in $\C([0,+\infty);\H)$.

\subsection{Existence of a Laplace limit at multiple times}\label{LapLmuL}
%      
%      The following is a key result to establish the large deviation principle in the space $\mathfrak{D}([0,\infty);\mathcal{S})$. It asserts that proving the large deviation principle in $\mathfrak{D}([0,\infty);\mathcal{S})$ is equivalent to demonstrating two main conditions: the exponential tightness in $\mathfrak{D}([0,\infty);\mathcal{S})$ and the existence of a Laplace limit (see \eqref{expLmt}). 
Let us consider the functional 
\begin{align}\label{LDPsemgp}
	\widetilde{\mathscr{V}}_n(t)g(\y):=-\mathscr{V}_n(t)g(\y)=-\frac{1}{n}\log\E\big[e^{-ng(\Y_n(t))}\big|
	\Y_n(0)=\y\big],
\end{align} 
where $g\in\mathrm{Lip}_b(\H)$. By the Markov property of the solution $\Y_n(\cdot)$, one can establish the following semigroup identity: 
	\begin{align*}
		\widetilde{\mathscr{V}}_n(t+s)g(\y)&=-
		\frac{1}{n}\log\E\big[e^{-ng(\Y_n(t+s))}\big|\Y_n(0)=\y\big]
		\nonumber\\&=-
		\frac{1}{n}\log\E\big[\E\big[e^{-ng(\Y_n(t+s))}\big|\Y_n(s)\big]\big|\Y_n(0)=\y\big]
		\nonumber\\&=-
		\frac{1}{n}\log\E[e^{-n\widetilde{\mathscr{V}}_n(t)g(\Y_n(s))}\big|\Y_n(0)=\y]
		\nonumber\\&=
		\widetilde{\mathscr{V}}_n(s)(\widetilde{\mathscr{V}}_n(t)g(\y)),
	\end{align*}
	for all $s,t\geq0$. Moreover, for $0\leq t_1\leq\cdots\leq t_m$ and $g_1,\ldots,g_m\in\C_b(\H)$, we find by using the Markov property of the solution $\Y_n(\cdot)$ and by the property of conditional expectation
\begin{align}\label{LDPsemg1}
	\E\big[e^{-n(g_1(\Y_n(t_1))+\ldots+g_m(\Y_n(t_m)))}\big]&=
	\E\big[\E\big[e^{-n(g_1(\Y_n(t_1))+\ldots+g_m(\Y_n(t_m)))}\big]\big|\mathscr{F}_{t_{m-1}}\big]
	\nonumber\\&=
	\E\big[\E\big[\underbrace{e^{-n(g_1(\Y_n(t_1))+\ldots+g_{m-1}(\Y_n(t_{m-1})))}}_{\text{measurable with respect to $\mathscr{F}_{t_{m-1}}$}} e^{-n g_m(\Y_n(t_m))}\big]
\big|\mathscr{F}_{t_{m-1}}\big]
	\nonumber\\&=
	\E\big[e^{-n(g_1(\Y_n(t_1))+\ldots+g_{m-1}(\Y_n(t_{m-1})))}
	\underbrace{[\E[e^{-n g_m(\Y_n(t_m))}]\big|\mathscr{F}_{t_{m-1}}]}_{\text{Markov property}}\big]
	\nonumber\\&=
	\E\big[e^{-n(g_1(\Y_n(t_1))+\ldots+g_{m-1}(\Y_n(t_{m-1})))}\underbrace{\E[e^{-n g_m(\Y_n(t_{m-1}))}]}_{:=e^{n\mathscr{V}_n(t_m-t_{m-1})
			g_m(\Y_n(t_{m-1}))}}\big]
	\nonumber\\&=
	\E\big[e^{-n(g_1(\Y_n(t_1))+\ldots+g_{m-1}(\Y_n(t_{m-1}))-\mathscr{V}_n(t_m-t_{m-1})g_m(\Y_n(t_{m-1})))}\big].
\end{align}
In a similar way, we calculate further
\begin{align}\label{LDPsemg2}
	&\E\big[e^{-n(g_1(\Y_n(t_1))+\ldots+g_{m-1}(\Y_n(t_{m-1}))-
		\mathscr{V}_n(t_m-t_{m-1})g_m(\Y_n(t_{m-1})))}\big]
	\nonumber\\&=
	\E\big[\E\big[e^{-n(g_1(\Y_n(t_1))+\ldots+g_{m-1}(\Y_n(t_{m-1}))-
		\mathscr{V}_n(t_m-t_{m-1})g_m(\Y_n(t_{m-1})))}\big]\big|
	\mathscr{F}_{t_{m-2}}\big]
	\nonumber\\&=
	\E\big[\E\big[\underbrace{e^{-n(g_1(\Y_n(t_1))+\ldots+g_{m-2}(\Y_n(t_{m-2})))}}_{\text{measurable with respect to $\mathscr{F}_{t_{m-2}}$}} e^{-n (g_{m-1}(\Y_n(t_{m-1}))-\mathscr{V}_n(t_m-t_{m-1})g_m(\Y_n(t_{m-1})))}\big]
	\big|\mathscr{F}_{t_{m-2}}\big]
	\nonumber\\&=
	\E\big[e^{-n(g_1(\Y_n(t_1))+\ldots+g_{m-2}(\Y_n(t_{m-2})))}
	\underbrace{[\E\big[e^{-n (g_{m-1}(\Y_n(t_{m-1}))-\mathscr{V}_n(t_m-t_{m-1})g_m(\Y_n(t_{m-1})))}\big]\big|\mathscr{F}_{t_{m-2}}]}_{\text{Markov property}}\big]
	\nonumber\\&=
	\E\big[e^{-n(g_1(\Y_n(t_1))+\ldots+g_{m-2}(\Y_n(t_{m-2})))}
	\underbrace{\E\big[e^{-n(g_{m-1}(\Y_n(t_{m-2}))- \mathscr{V}_n(t_m-t_{m-1})g_m(\Y_n(t_{m-2})))}\big]}_{:=e^{n\mathscr{V}_n(t_{m-1}-t_{m-2})(g_{m-1}-\mathscr{V}_n(t_m-t_{m-1}))(\Y_n(t_{m-2}))}}\big]
	\nonumber\\&=
	\E\big[e^{-n(g_1(\Y_n(t_1))+\ldots+g_{m-2}(\Y_n(t_{m-2}))-\mathscr{V}_n(t_{m-1}-t_{m-2})(g_{m-1}-\mathscr{V}_n(t_m-t_{m-1})g_m)(\Y_n(t_{m-2})))}\big].
\end{align}
Proceeding with the same methodology as in \eqref{LDPsemg1}–\eqref{LDPsemg2}, and continuing the iteration up to time $t_1$, we ultimately arrive at the following result:
\begin{align}\label{LDPsemg3}
	&\E\big[e^{-n(g_1(\Y_n(t_1))+\ldots+g_m(\Y_n(t_m)))}\big]
	\nonumber\\&=
	e^{n\mathscr{V}_n(t_1)(g_1-\mathscr{V}_n(t_2-t_1)
		(g_2-\ldots-\mathscr{V}_n(t_m-t_{m-1})g_m)\ldots)(\y)}.
\end{align}
\vskip 0.2cm 
\noindent
\textbf{Verifying condition \eqref{expLmt} of Proposition \ref{exptightD}.} 
From \eqref{LDPsemg3}, we write
\begin{align}\label{LDPsemg33}
	&\frac{1}{n}\log\E\big[e^{-n(g_1(\Y_n(t_1))+\ldots+g_m(\Y_n(t_m)))}\big]
	\nonumber\\&=
	\mathscr{V}_n(t_1)(g_1-\mathscr{V}_n(t_2-t_1)
	(g_2-\ldots-\mathscr{V}_n(t_m-t_{m-1})g_m)\ldots)(\y).
\end{align}
Therefore, existence of the limit \eqref{expLmt} is equivalent to the existence of the following limit:
\begin{align}\label{LDPsemg4}
	\lim\limits_{n\to+\infty}[\mathscr{V}_n(t_1)(g_1-\mathscr{V}_n(t_2-t_1)
	(g_2-\ldots-\mathscr{V}_n(t_m-t_{m-1})g_m)\ldots)(\y)].
\end{align}
We now consider each of the individual term appearing in \eqref{LDPsemg4}.  Notice that from  Proposition \ref{vfLDP3}, for $g_m\in\mathrm{Lip}_b(\H)$, the functions $\mathscr{V}_n(t_m-t_{m-1})g_m$ are uniformly bounded and weakly sequentially continuous on $\V$. Moreover, the limit
\begin{align*}
	\lim\limits_{n\to+\infty}\mathscr{V}_n(t_m-t_{m-1})g_m(\y)=
	\mathscr{V}(t_m-t_{m-1})g_m(\y),  \ \text{uniformly on bounded subsets of } \V,
\end{align*}
exists. Similarly, the functions $\mathscr{V}_n(t_{m-1}-t_{m-2})(g_{m-1}-\mathscr{V}_n(t_m-t_{m-1})g_m)$, for $g_m,g_{m-1}\in\mathrm{Lip}_b(\H)$, are also uniformly bounded and weakly sequentially continuous in $\V$. Further, in view of the above the limit, the following limit:
\begin{align*}
	&\lim\limits_{n\to+\infty}\mathscr{V}_n(t_{m-1}-t_{m-2})(g_{m-1}-\mathscr{V}_n(t_m-t_{m-1})g_m)(\y)
	\nonumber\\&=
	\mathscr{V}(t_{m-1}-t_{m-2})(g_{m-1}-\mathscr{V}(t_m-t_{m-1})g_m)(\y),
	\ \text{uniformly on bounded subsets of } \V,
\end{align*}
exists. Continuing like this we finally conclude that
\begin{align}\label{LDPsemg5}
	&\lim\limits_{n\to+\infty}[\mathscr{V}_n(t_1)(g_1-\mathscr{V}_n(t_2-t_1)
	(g_2-\ldots-\mathscr{V}_n(t_m-t_{m-1})g_m)\ldots)(\y)].
	\nonumber\\&=
	\lim\limits_{n\to+\infty}[\mathscr{V}(t_1)(g_1-\mathscr{V}(t_2-t_1)
	(g_2-\ldots-\mathscr{V}(t_m-t_{m-1})g_m)\ldots)(\y)],
\end{align}
uniformly on bounded subsets of $\V$. Thus, from \eqref{LDPsemg33} and \eqref{LDPsemg5}, it follows that the limit \eqref{expLmt} exists.

\subsection{Exponential tightness}\label{exptghtpath} We now need to show that the sequence $\{\Y_n(\cdot)\}_{n\geq1}$ is exponentially tight in the space $\D([0,+\infty);\H)$. To achieve this, we invoke Theorem \ref{LDPsemg6}, which says that the exponential tightness of the $\H-$valued process $\{\Y_n(\cdot)\}_{n\geq1}$ is equivalent to that of real valued process  $\{g(\Y_n(\cdot))\}_{n\geq1}$, where $g:\H\to\R$ is some real-valued function.

The following lemma (see \cite[Lemma 8.2]{AS2}) is useful to get the exponential tightness.
\begin{lemma}\label{mathscrA}
	Let $m\in\N$. Define a set
	\begin{align*}
		\mathscr{B}=\bigg\{\sum\limits_{i=1}^m g_i(\|\y-\y_i\|_{\H}): \ 
		&\y_i\in\V, \ g_i\in\C^2([0,+\infty)),\\& \ g_i'(0)=0, \ g_i, \ g_i', \ g_i'' \ \text{ are bounded} \bigg\}.
	\end{align*}
	Then, the following statements are true:
	\begin{itemize}
		\item $\mathscr{B}$ is closed under addition and isolates points in $\H$.
		\item If $g\in\mathscr{B},$ then for every $r>0$, there exists a constant $\mathpzc{C}=\mathpzc{C}(g,r)\geq0$ such that
		\begin{align}\label{gLDP}
			\sup\limits_{\|\y\|_{\V}\leq r} \|\mathcal{D}g(\y)\|_{\V}\leq\mathpzc{C}.
		\end{align}
	\end{itemize} 
\end{lemma}

\begin{proof}
	Using  Definition \ref{isopoi} of a family of isolates points and the fact that $\V$ is dense in $\H$, one can prove that $\mathscr{B}$ isolates points in $\H$.
\end{proof}
\vskip 0.2cm 
\noindent
\textbf{Verifying exponential tightness.} 
Let us now conclude that the process $\{\Y_n(\cdot)\}_{n\in\N}$ is exponentially tight in 
$\D([0,+\infty);\H)$. For this, it is suffices to verify all the conditions stated in Theorem \ref{LDPsemg6}.

\begin{theorem}\label{exptightLDP6}
	The sequence $\{\Y_n(\cdot)\}_{n\in\N}$, where $\Y_n(\cdot)$ is a solution to the SCBF system \eqref{stap} with $\Y_n(0)=\y$, is exponentially tight in
	$\D([0,+\infty);\H)$.
\end{theorem}

\begin{proof}
	The proof is divided into the following number of steps:
	\vskip 0.2 cm
	\noindent
	\textbf{Step-I:}\emph{ $\{\Y_n(\cdot)\}_{n\geq1}$ satisfies the exponential compact containment condition.}
	Let us take $\mathcal{K}_{M,T}=\{\y\in\V: \|\y\|_{\V}\leq r\}$ for sufficiently big $r>0$, will be specified later. Note that the set $\mathcal{K}_{M,T}$ is compact in $\H$. 
	By using  Markov's inequality and \eqref{expmoest}, we calculate 
	\begin{align*}
		&\P\big(\{\text{there exists } \ 0\leq t\leq T \ \text{ such that } \ \Y_n(t)\notin\mathcal{K}_{M,T}\}\big)
		\nonumber\\&=
		\P\big(\{\omega\in\Omega: \|\Y_n(t)\|_{\V}>r\}\big)
		\nonumber\\&=
		\P\bigg(\{\omega\in\Omega: \sup\limits_{s\in[t,T]} e^{n\mathpzc{c}_1\|\Y_n(t)\|_{\V}^2}>e^{n\mathpzc{c}_1r^2}\}\bigg)
		\nonumber\\&\leq
		\frac{1}{e^{n\mathpzc{c}_1r^2}}
		\E\left[\sup\limits_{s\in[t,T]}e^{n\mathpzc{c}_1\|\Y_n(s)\|_{\V}^2}
		\right]
		\nonumber\\&\leq e^{n(\mathpzc{c}_2-\mathpzc{c}_1r^2)}.
	\end{align*}
	Consequently, we find
	\begin{align*}
		\frac{1}{n}\log\P\big(\{\text{there exists } \ 0\leq t\leq T \ \text{ such that } \ X_n(t)\notin\mathcal{K}_{M,T}\}\big)\leq\mathpzc{c}_2-\mathpzc{c}_1r^2
		\leq-M,
	\end{align*}
	provided $r\geq r_0:=\sqrt{\frac{\mathpzc{c}_2+M}{\mathpzc{c}_1}}$.     
	This shows that the $\H-$valued process $\{\Y_n(\cdot)\}_{n\geq1}$ satisfies the required  condition.
	\vskip 0.2 cm
	\noindent
	\textbf{Step-II:}\emph{ $\{\Y_n(\cdot)\}_{n\geq1}$ is weakly exponentially tight.}
	Let $\mathscr{B}$ be the family of functions which satisfies the requirement given in Lemma \ref{mathscrA}. To prove the weak exponential tightness of the solution $\Y_n(\cdot)$, we need to show that for every $g\in\mathscr{A}$, the sequence $\{g(\Y_n)\}_{n\in\N}$ is exponentially tight in $\D([0,+\infty);\R)$. As mentioned in \cite[Theorem 4.1, Chapter 4]{JFTGK}, it suffices to prove the following:
	
	 For $s>0$ and $\lambda\in\R$, there exist random variables $\mathfrak{q}_n(s,\lambda,T)$, non decreasing in $s$, such that for $0\leq t\leq t+s\leq T$, the inequality
		\begin{align}\label{exptght1}
			\E\big[e^{n\lambda(g(\Y_n(t+s))-g(\Y_n(t)))}\big|\mathscr{F}_t\big]
			\leq\E\big[e^{\mathfrak{q}_n(s,\lambda,T)}\big|\mathscr{F}_t\big]
		\end{align}
		holds, and in addition
		\begin{align}\label{exptght2}
			\lim\limits_{s\to0}\limsup\limits_{n\to+\infty} \frac{1}{n}
			\log\E\big[e^{\mathfrak{q}_n(s,\lambda,T)}\big]=0.
	\end{align}
	Let $r>0$ be sufficiently large such that if 
	\begin{align*}
		\mathpzc{S}_3^n:=
		\{\omega\in\Omega: \ \text{ there exists } \ 0\leq t\leq T \ \text{ such that } 
		\ \|\Y_n(t)\|_{\V}>r\},
	\end{align*}
	then we obtain the following bound:
	\begin{align}\label{exptght3}
		\E[\mathds{1}_{\mathpzc{S}_3^n}]=\P(\mathpzc{S}_3^n)\leq e^{-2n|\lambda|\|g\|_{\infty}}.
	\end{align}
	We denote $\mathpzc{S}_4^n:=\Omega\setminus\mathpzc{S}_3^n$. We write the left hand side of \eqref{exptght1} as follows:
	\begin{align}\label{exptght4}
		&\E\big[e^{n\lambda(g(\Y_n(t+s))-g(\Y_n(t)))}|\mathscr{F}_t\big]
		\nonumber\\&=
		\E\big[e^{n\lambda(g(\Y_n(t+s))-g(\Y_n(t)))}\mathds{1}_{\mathpzc{S}_3^n}|
		\mathscr{F}_t\big]+
		\E\big[e^{n\lambda(g(\Y_n(t+s))-g(\Y_n(t)))}\mathds{1}_{\mathpzc{S}_4^n}|
		\mathscr{F}_t\big]
		\nonumber\\&\leq
		\E\big[e^{2n\lambda\|g\|_{\infty}}\mathds{1}_{\mathpzc{S}_3^n}|
		\mathscr{F}_t\big]+\E\big[e^{n\lambda(g(\Y_n(t+s))-g(\Y_n(t)))}
		\mathds{1}_{\mathpzc{S}_4^n}|\mathscr{F}_t\big].
	\end{align}
		On applying It\^o's formula to the function 
	$g(\cdot)$ on $[t,t+s]$ for $t\in[0,T]$, and to the process $\Y_n(\cdot)$, we get
	\begin{align}\label{conLDP}
		&\E\big[\exp\big(n\lambda(g(\Y_n(t+s))-g(\Y_n(t)))\big)
		\mathds{1}_{\mathpzc{S}_4^n}|\mathscr{F}_t\big]
		\nonumber\\&=
		\E\bigg[\exp\bigg(n\lambda\int_t^{t+s}
		\big(-\mu\A\Y_n(\tau)-\mathcal{B}(\Y_n(\tau))-\alpha\Y_n(\tau)-
		\beta\mathcal{C}(\Y_n(\tau))+
		\f(s),\D g(\Y_n(\tau))\big)\d \tau
		\nonumber\\&\quad+
		\frac{n\lambda}{\sqrt{n}}\int_t^{t+s}
		\big(\Q^{\frac12}\d\mathbf{W}(s),\D g(\Y_n(\tau))\big)+
		\frac{n\lambda}{2n}\int_t^{t+s}
		\Tr(\Q\D^2g(\Y_n(\tau)))\d\tau\bigg)\mathds{1}_{\mathpzc{S}_4^n}
		\bigg|\mathscr{F}_t\bigg]
		\nonumber\\&=
		\E\bigg[\exp\bigg(n\lambda\int_t^{t+s}
		\big(-\mu\A\Y_n(\tau)-\mathcal{B}(\Y_n(\tau))-\alpha\Y_n(\tau)-
		\beta\mathcal{C}(\Y_n(\tau))+
		\f(s),\D g(\Y_n(\tau))\big)\d \tau
		\nonumber\\&\quad+
		\int_t^{t+s}\bigg(\frac{n\lambda}{2n}
		\Tr(\Q\D^2g(\Y_n(\tau)))+\frac{n\lambda^2}{2}\|\Q^{\frac12}
		\D g(\Y_n(\tau))\|_{\H}^2-\frac{n\lambda^2}{2}\|\Q^{\frac12}
		\D g(\Y_n(\tau))\|_{\H}^2\bigg)\d\tau
		\nonumber\\&\quad+
		\frac{n\lambda}{\sqrt{n}}\int_t^{t+s}
		\big(\Q^{\frac12}\d\mathbf{W}(s),\D g(\Y_n(\tau))\big)\bigg)\mathds{1}_{\mathpzc{S}_4^n}
		\bigg|\mathscr{F}_t\bigg]
		\nonumber\\&\leq
		\E\bigg[\exp\bigg(n\lambda\int_t^{t+s}
		\big(-\mu\A\Y_n(\tau)-\mathcal{B}(\Y_n(\tau))-\alpha\Y_n(\tau)-
		\beta\mathcal{C}(\Y_n(\tau))+
		\f(s),\D g(\Y_n(\tau))\big)\d \tau
		\nonumber\\&\quad+
		\int_t^{t+s}\bigg(\frac{\lambda}{2}
		\Tr(\Q\D^2g(\Y_n(\tau)))+\frac{n\lambda^2}{2}\|\Q^{\frac12}
		\D g(\Y_n(\tau))\|_{\H}^2\bigg)\d\tau\bigg)
		\nonumber\\&\quad\times
		\exp\bigg(-\frac{n\lambda^2}{2}\int_t^{t+s}\|\Q^{\frac12}
		\D g(\Y_n(\tau))\|_{\H}^2\d\tau+\sqrt{n}\lambda\int_t^{t+s}
		\big(\Q^{\frac12}\d\mathbf{W}(s),\D g(\Y_n(\tau))\big)\bigg)\mathds{1}_{\mathpzc{S}_4^n}
		\bigg|\mathscr{F}_t\bigg].
	\end{align}
	By using H\"older's inequality and Sobolev embedding $\V\hookrightarrow\wi\L^4$, we estimate 
	\begin{align}\label{bLDP}
		|(\mathcal{B}(\Y_n),\D g(\Y_n))|=|b(\Y_n,\D g(\Y_n), \Y_n)| \leq\|\Y_n\|_{\wi\L^4}^2\|\D g(\Y_n)\|_{\V}
		\leq\mathpzc{C}\|\Y_n\|_{\V}^2\|\D g(\Y_n)\|_{\V}.
	\end{align}  
	Since, the nonlinear operator $\mathcal{C}(\cdot)$ is well-defined from $\wi\L^{r+1}$ to $\wi\L^{\frac{r+1}{r}}$ and in view of the Sobolev embedding $\V\hookrightarrow\wi\L^{r+1}$, for $r\in[1,5]$, we have $\D g(\Y_n)\in\wi\L^{r+1}$. Therefore, we find
	\begin{align}\label{cLDP}
		|(\mathcal{C}(\Y_n),\D g(\Y_n))|&\leq\|\mathcal{C}(\Y_n)\|_{\wi\L^{\frac{r+1}{r}}}
		\|\D g(\Y_n)\|_{\L^{r+1}}=\|\Y_n\|_{\L^{r+1}}\|\D g(\Y_n)\|_{\L^{r+1}}
		\nonumber\\&\leq\mathpzc{C}\|\Y_n\|_{\V}^{r}\|\D g(\Y_n)\|_{\V}.
	\end{align}
	From Hypothesis \ref{fhyp}, \eqref{gLDP} and \eqref{bLDP}-\eqref{cLDP}, and the fact that $g\in\mathscr{A}$, we conclude from \eqref{conLDP} that
	\begin{align}\label{conLDP1}
		&\E[\exp\big(n\lambda(g(\Y_n(t+s))-g(\Y_n(t)))\big)
		\mathds{1}_{\mathpzc{S}_4^n}|\mathscr{F}_t]
		\nonumber\\&\leq
		\E\bigg[\exp\bigg(n\lambda\int_t^{t+s}
		\big(-\mu\A\Y_n(\tau)-\mathcal{B}(\Y_n(\tau))-\alpha\Y_n(\tau)-
		\beta\mathcal{C}(\Y_n(\tau))+
		\f(s),\D g(\Y_n(\tau))\big)\d \tau
		\nonumber\\&\quad+
		\int_t^{t+s}\bigg(\frac{\lambda}{2}
		\Tr(\Q\D^2g(\Y_n(\tau)))+\frac{n\lambda^2}{2}\|\Q^{\frac12}
		\D g(\Y_n(\tau))\|_{\H}^2\bigg)\d\tau\bigg)\mathcal{M}(t,t+s) \mathds{1}_{\mathpzc{S}_4^n}
		\bigg|\mathscr{F}_t\bigg]
		\nonumber\\&\leq
		\E\bigg[e^{n\wi{\mathpzc{C}}(\lambda,r,g)s}\mathcal{M}(t,t+s)
		\mathds{1}_{\mathpzc{S}_4^n}
		\bigg|\mathscr{F}_t\bigg]
		\nonumber\\&=e^{n\wi{\mathpzc{C}}(\lambda,r,g)s}
		\E\bigg[\mathcal{M}(t,t+s)\bigg|\mathscr{F}_t\bigg],
	\end{align}
	where $\wi{\mathpzc{C}}(\lambda,r,g)$ is some positive constant
	and 
	\begin{align*}
		\mathcal{M}(t,u)=\exp\bigg(\sqrt{n}\lambda\int_t^{u}
		\big(\Q^{\frac12}\d\mathbf{W}(s),
		\D g(\Y_n(\tau))\big)-\frac{n\lambda^2}{2}\int_t^{u}\|\Q^{\frac12}
		\D g(\Y_n(\tau))\|_{\H}^2\d\tau\bigg)
	\end{align*}
	is the  Dol\'eans-Dade exponential of the martingale $\sqrt{n}\lambda\int_t^{u}
	\big(\Q^{\frac12}\d\mathbf{W}(s),
	\D g(\Y_n(\tau))\big)$. It is worth mentioning that $\mathcal{M}(t,u)$ is a solution to the following stochastic differential equation 
		\begin{align*}
			\d\mathcal{M}(t,u)=\mathcal{M}(t,u)\left(\sqrt{n}\lambda\int_t^{u}
			\big(\Q^{\frac12}\d\mathbf{W}(s),
			\D g(\Y_n(\tau))\big)\right) \ \text{ with } \ 	\mathcal{M}(t,t)=1.
		\end{align*}
		Moreover, by the properties of  Dol\'eans-Dade exponential (see \cite[Chapter 8]{FCK}), $\mathcal{M}(t,u)$ is a martingale with respect to the filtration $\mathscr{F}_u$. Therefore, we have $\E[\mathcal{M}(t,t+s)|\mathscr{F}_t]=1$ and \eqref{conLDP1} yields
		\begin{align}\label{conLDP2}
			\E\big[\exp\big(n\lambda(g(\Y_n(t+s))-g(\Y_n(t)))\big)
			\mathds{1}_{\mathpzc{S}_4^n}|\mathscr{F}_t\big]
			\leq e^{n\wi{\mathpzc{C}}(\lambda,r,g)s}.
	\end{align}
	Thus \eqref{exptght4}, together with \eqref{conLDP2}, gives
	\begin{align}\label{exptght5}
		\E\big[e^{n\lambda(g(\Y_n(t+s))-g(\Y_n(t)))}\big|\mathscr{F}_t\big]
		\leq
		\E\big[e^{n\wi{\mathpzc{C}}(\lambda,r,g)s}+e^{2n\lambda\|g\|_{\infty}}
		\mathds{1}_{\mathpzc{S}_3^n}\big|
		\mathscr{F}_t\big].
	\end{align}
	Let us now take $\mathfrak{q}_n(s,\lambda,T)=\log\big(
	e^{n\wi{\mathpzc{C}}(\lambda,r,g)s}+e^{2n\lambda\|g\|_{\infty}}
	\mathds{1}_{\mathpzc{S}_3^n}\big)$. It is clear that $\mathfrak{q}_n(s,\lambda,T)$ is non decreasing in $s$ for all $\lambda\in\R$ and therefore condition \eqref{exptght1} is fulfilled. Moreover, from \eqref{exptght3} and Lemma \ref{LOG}, we find
	\begin{align*}
		\frac{1}{n}\log\E[e^{\mathfrak{q}_n(s,\lambda,T)}]&=
		\frac{1}{n}\log\E\big[e^{n\wi{\mathpzc{C}}(\lambda,r,g)s} +e^{2n\lambda\|g\|_{\infty}}
		\mathds{1}_{\mathpzc{S}_3^n}\big]
		\nonumber\\&\leq
		\frac{1}{n}\log(e^{n\wi{\mathpzc{C}}(\lambda,r,g)s} +1)
		\nonumber\\&=
		\wi{\mathpzc{C}}(\lambda,r,g)s+
		\frac{1}{n}\log(1+e^{-n\wi{\mathpzc{C}}(\lambda,r,g)s})
		\nonumber\\&=
		\wi{\mathpzc{C}}(\lambda,r,g)s+
		\frac{1}{n}e^{-n\wi{\mathpzc{C}}(\lambda,r,g)s},
	\end{align*}
	and hence condition \eqref{exptght2} is also fulfilled. Thus, we finally conclude that $\{g(\Y_n)\}_{n\in\N}$ is exponentially tight in  $\D([0,+\infty);\R)$ for every $g\in\mathscr{B
	}$ and this completes the proof.
\end{proof}  
We have shown in Theorems \ref{LDPsemg6} and \ref{exptightLDP6} that the sequence of stochastic process $\{\Y_n(\cdot)\}_{n\geq1}$ is exponentially tight in $\D([0,+\infty);\H)$. Moreover, the existence of Laplace limit is immediate from \eqref{LDPsemg33} and \eqref{LDPsemg5}. Therefore, in view of Proposition \ref{exptightD}, it follows that the sequence of stochastic processes $\{\Y_n(\cdot)\}_{n\geq1}$ satisfies LDP in the space $\D([0,+\infty);\H)$. 

\subsection{Large deviation principle in the continuous space} As mentioned in the introduction, we now conclude the main objective of this work, that is, establishing the LDP in the space $\C([0,+\infty);\H)$. 
\begin{theorem}(LDP in $\C([0,+\infty);\H)$)\label{mainthm}
	The sequence of stochastic process $\{\Y_n(\cdot)\}_{n\geq1}$, where $\Y_n(\cdot)=\Y_n(\cdot;0,\y)$ is a solution to the SCBF system \eqref{stap} with $\Y_n(0)=\y$, satisfies the large deviation principle in $\C([0,+\infty);\H)$.
\end{theorem}
\begin{proof}
	It is important to note that the process $\Y_n(\cdot)$ has $\P-$a.s. continuous paths in $\H$ (see Proposition \ref{weLLp}). It implies that 
	$\sup\limits_{s\in[t,T]}\|\Y_n(s)-\Y_n(s-)\|_{\H}=0$, $\P-$a.s. Consequently, for any $\eta>0$, we find that 
	$\P\left(\sup\limits_{s\leq T}\|\Y_n(s)-\Y_n(s-)\|_{\H}\geq\eta\right)=0$. 
	Therefore, the condition \eqref{cexptght} of Definition \ref{cextght} holds immediately, that is, the process $\Y_n(\cdot)$ is $\C-$exponentially tight in 
	$\D([0,+\infty);\H)$. Therefore, all the assumptions of Theorem \ref{refthm} are fulfilled, and it follows that the sequence $\{\Y_n(\cdot)\}_{n\geq1}$ satisfies the large deviation principle in $\C([0,+\infty);\H)$. Furthermore, the rate function admits the explicit representation given in \eqref{ratefun} (see Proposition \ref{exptightD}). This concludes the proof.
\end{proof}

	\begin{appendix}\renewcommand{\thesection}{\Alph{section}}
	\numberwithin{equation}{section}
	\section{A useful result}\label{useca}
The following inequalities are classical. For completeness, we record their proofs here.
	 \begin{lemma}\label{LOG}
	 	\begin{enumerate}
	 	\item 	For any $x\geq0$, we have 
	 		\begin{align*}
	 			 \log(1+x)\leq x.
	 		\end{align*}
	 		\item For $0<x\leq\frac{1}{2}$, we infer 
	 		\begin{align*}
	 			-2x\leq\log(1-x).
	 		\end{align*}
	 			\end{enumerate}
	\end{lemma}
\begin{proof}
	Part $(1)$ is an immediate consequence of $x-\log(1+x)$ being a non-decreasing function for $x\geq0$. Moreover, since
	\begin{align*}
		\log(1-x)=-\sum_{k=1}^{\infty} \frac{x^k}{k}\geq
		-\sum_{k=1}^{\infty} x^k=\frac{-x}{1-x}\geq-2x,
	\end{align*}
	where we used the fact that $\frac{-1}{1-x}\geq-2$ for $x\leq\frac12$, therefore, part $(2)$ follows.
\end{proof}

	\end{appendix}
	\medskip
	\noindent
	\textbf{Acknowledgments:} 
	The first author expresses sincere gratitude to the Ministry of Education, Government of India (MHRD), for financial support. The research work of M. T. Mohan was funded by the National Board for Higher Mathematics (NBHM), Department of Atomic Energy, Government of India, under Project No. 02011/13/2025/NBHM(R.P)/R\&D II/1137.

\end{document}